\documentclass[11pt]{preprint}
\usepackage[T1]{fontenc}
\usepackage[full]{textcomp}
\usepackage[osf]{newtxtext}
\usepackage{microtype}

\usepackage[a4paper,twoside,inner=30mm,outer=30mm,top=25mm,bottom=30mm]{geometry}

\usepackage{graphicx} \usepackage{subfiles}
\usepackage[scr=boondoxupr]{mathalfa}
\usepackage{amsmath, amssymb, amsfonts}
\usepackage{amsthm}
\usepackage{mathtools}
\usepackage{xcolor}
\usepackage{enumitem}
\usepackage{letters}
\usepackage{furthersymbols}
\usepackage{envs_theorems}
\usepackage{envs}
\usepackage{comment}
\usepackage[linktoc=page]{hyperref}

\usepackage{longtable}
\usepackage{booktabs}

\usepackage{cleveref}
\usepackage{nicefrac}
\usepackage[final,left]{showlabels}

\makeatletter
\renewcommand\paragraph{\@startsection{paragraph}{4}{\z@}{1.5ex \@plus .5ex \@minus .2ex}{-1em}{\normalfont\normalsize\bfseries}}
\makeatother

\newcommand{\scal}[1]{\langle #1 \rangle}
\newcommand{\Higgs}{ \Phi  }
\renewcommand{\Re}{\mathfrak{Re}}

\newcommand{\Coul}{\mathrm{C}}
\newcommand{\bplus}{\mathbin{\pmb{+}}}

\title{A variational approach for the 2D Abelian Yang--Mills--Higgs measure}
	\author{Nikolay Barashkov$^1$, Ajay Chandra$^2$, Ilya Chevyrev$^3$, Andreas Koller$^4$, Abdulwahab Mohamed$^1$}

	\institute{Max Planck Institute for Mathematics in the Sciences, Leipzig, Germany \and Purdue University, West Lafayette, USA
    \and SISSA, Trieste, Italy \and University of Oxford, Oxford, UK \\
		\email{barashkov@mis.mpg.de, ajay.chandra@gmail.com,\\ ichevyrev@gmail.com, andreas.koller@maths.ox.ac.uk, \\ abdulwahab.mohamed@mis.mpg.de}}

\begin{document}

\maketitle

	\begin{abstract}
        We use the Bou\'e--Dupuis variational method to construct the Abelian Yang--Mills--Higgs measure on the two-dimensional torus with polynomial Higgs potentials, giving, to our knowledge, the first rigorous construction of an interacting gauge theory by this approach.
        We first use the variational method to construct the Higgs field conditioned on the gauge field, with a covariant Gaussian free field as the reference measure.
        Integrating out the Higgs field produces a weight on the gauge field, which we control through a second application of the variational method.
        We also prove new exponential integrability estimates for the joint gauge--Higgs field and establish gauge covariance.
        Finally, we prove convergence of the Laplace transforms of the regularised measures and establish a new large-deviation principle for gauge-invariant observables in the semiclassical limit.

        \end{abstract}
    	\setcounter{tocdepth}{2}

\tableofcontents

\section{Introduction}

\subsection{The model and problem}

The Abelian Yang--Mills--Higgs (YMH) model is the simplest Euclidean gauge theory in which a gauge field is coupled to matter.
On the trivial $U(1)$-bundle over the $d$-dimensional torus $\T^d = \R^d/\Z^d$,
a configuration consists of an $\bfi\R$-valued \(1\)-form $A \colon \T^d \to (\bfi\R)^d$, called the \emph{gauge field}, and a complex-valued scalar field $\phi \colon \T^d \to \C$, called the \emph{Higgs field}.
The model is described by the action
\begin{equation}\label{eq:intro_action}
    \cS_{\YMH}(A,\phi)
    = \frac12\int_{\T^d}
    \left(
        |\diff A|^2
        + |\diff_A\phi|^2
    \right)\dif x
    +\int_{\T^d}\cV(|\phi|^2)\dif x,
\end{equation}
where $\diff_A$ is the covariant derivative and the potential $\cV(|\phi|^2)$ has leading term $c_n|\phi|^{2n}$, with $n\ge2$ and $c_n>0$, as in \eqref{eq:def_V}. We include any Higgs mass term in $\cV$.
The corresponding Gibbs measure should have the formal density
\begin{equation}\label{eq:intro_measure}
    \dif\mu^{\YMH}(A,\phi)
    \propto
    \exp\bigl(-\cS_{\YMH}(A,\phi)\bigr) \dif A \dif\phi.
\end{equation}
where $\dif A \dif \phi$ denotes a formal Lebesgue measure on the infinite-dimensional space of fields.
Ignoring for now questions of regularity,
both the action and, formally, the measure are invariant under the action of gauge transformations $g\colon \T^d\to U(1)$ given by
\[
    (A,\phi)\mapsto g\Cdot(A,\phi) := (A^g,g\phi),
    \qquad A^g=A-\diff g\,g^{-1}.
\]
The measure $\mu^{\YMH}$ defined in \eqref{eq:intro_measure}
is, at this stage, entirely formal and giving meaning to this measure is an important problem in constructive quantum field theory.
The Abelian YMH model on \(\R^2\) was constructed first by Brydges--Fr\"ohlich--Seiler
\cite{BFS79,BFS80,BFS81}.
These works moreover verified all the Osterwalder--Schrader (OS) axioms except clustering.
Later, King~\cite{King86a,King86b}, using renormalisation group techniques from \cite{Balaban1983III,Balaban81II,Balaban81I}, gave a construction on \(\R^3\) and \(\R^2\),
also verifying the OS axioms except clustering and rotation invariance.
More recently, another construction on \(\T^2\) was proposed in \cite{BC24_YM} based on the Langevin dynamic.
This work established bounds on the stochastic quantisation equations associated to \eqref{eq:intro_measure} (in the continuum) which are sufficiently strong to ensure existence of an invariant measure.

The purpose of the present paper is to give a direct and comparatively transparent continuum construction on $\T^2$ using the probabilistic variational method of \cite{BG18,Barash}.
The variational representation allows us to prove convergence for Laplace transforms for a large class of observables of the joint gauge--Higgs field as we remove UV-cutoffs.
As a consequence of our analysis, we furthermore prove a large deviation principle for the low-temperature limit of the measure with rate function $\cS_{\YMH}$, which provides a concrete link between our constructed measure and the YMH action \eqref{eq:intro_action}.

Making sense of $\mu^{\YMH}$ is non-trivial for several reasons.
First, the formal Lebesgue measure in \eqref{eq:intro_measure} does not exist.
A natural way to bypass this issue is to extract a Gaussian reference measure from the quadratic terms in the action \eqref{eq:intro_action} and treat the rest as a perturbation.
The second difficulty is that the reference Gaussian measure in dimensions $d\geq 2$ is supported on a space of distributions (not functions).
It is thus a priori unclear how to make sense of the non-quadratic terms in \eqref{eq:intro_action} and it turns out that they require renormalisation.
These diffiuclties are common to many interacting Euclidean field theories.

There is another difficulty specific to gauge theories: the quadratic form $A\mapsto\|\diff A\|_{L^2}^2$ vanishes on closed $1$-forms, so it does not define a Gaussian probability measure on all gauge fields.
This degeneracy reflects gauge invariance.

We henceforth take $d=2$ and use the Coulomb gauge, with the harmonic component restricted to $\bfi(0,2\pi]^2$:
\begin{equation}\label{eq:Coulomb}
\diff^*A=0,\qquad \int_{\T^2}A\,\dif x\in\bfi(0,2\pi]^2.
\end{equation}
A gauge transformation removes the exact part of the Hodge decomposition, and large gauge transformations shift the harmonic part by $2\pi\bfi\Z^2$.
The remaining constant gauge transformations act only on the Higgs field.
Appendix~\ref{app:gauge_fixing} gives the details, including the gauge-fixing statement for $H^1$ configurations.

Our reference measure $\mu^{\rmA}$ is the law of
\begin{equation}\label{eq:intro_pure_YM_field}
A=\vartheta+\diff^*(\diff\diff^*)^{-1}\xi,
\end{equation}
where $\xi$ is spatially mean-zero $\bfi\R$-valued white noise, viewed as a $2$-form, and $\vartheta$ is independent and uniform on $\bfi(0,2\pi]^2$.
This is the pure Yang--Mills reference field in the chosen gauge.

\subsection{Regularised measures and the main result} \label{sec:def_regularised_measure}
Let $A_T$ be the smooth finite-dimensional approximation of \eqref{eq:intro_pure_YM_field} specified in Section~\ref{sec:gauge_variational_setup}, with \(T\to\infty\) corresponding to the removal of the ultraviolet cut-off.
Writing $L_B=\diff_B^*\diff_B$ for the covariant Laplacian of a (smooth) connection $B$, we have $\int|\diff_{A_T}\phi|^2=\int\scal{\phi,L_{A_T}\phi}$ in \eqref{eq:intro_action}. The quadratic term $A_T^*A_T$ has divergent expectation as $T\to\infty$, so we renormalise by
\begin{equation}\label{eq:gauge_wick_counterterm}
L_{A_T}\rightsquigarrow L_{A_T}-\rmc_T,\qquad \rmc_T:=\E[|(A_T-\vartheta)(0)|^2].
\end{equation}
Here $A_T-\vartheta$ is the centred Gaussian part of $A_T$, and $\rmc_T\to+\infty$ as $T\to\infty$.
For a fixed smooth connection $A$, let $\lambda_{A,T}$ be the spectral shift defined by \eqref{eq:m_def} for the operator $L_A-\rmc_T$, so that the shifted operator is bounded below by $1$.
It is possible to choose \(\rmc_T\) so that \(\lambda_{A,T}\) is of order \(O(1)\).
We then define a truncated covariant Gaussian free field by
\begin{equation}\label{eq:intro_covariant_GFF}
    \phi_{A,N,T}
    =\int_0^N
    \bigl(L_A-\rmc_T+\lambda_{A,T}+\lambda\bigr)^{-1}
    \dif X^{\Phi}_\lambda,
\end{equation}
where the real and imaginary parts of $X^{\Phi}$ are independent standard cylindrical Brownian motions. As $N\to\infty$, the complex covariance of $\phi_{A,N,T}$ approaches twice the inverse of the positive renormalised covariant Laplacian \(L_A - \rmc_T + \lambda_{A,T}\).

Our regularisation of the Yang--Mills--Higgs measure is written as
\begin{equation}\label{eq:intro_regularised_measure}
    \dif\mu^{\YMH}_{N,T}(A,\phi)
    =\frac{1}{\cZ_{N,T}}
    \exp\Bigl(-\cQ_{N,T}(A)-\cV_{A,N,T}(\phi)\Bigr)
    \dif\mu^{\GFF}_{A,N,T}(\phi)\dif\mu_T^{\rmA}(A).
\end{equation}
Here $\mu_{A,N,T}^\GFF=\Law(\phi_{A,N,T})$ and $\cV_{A,N,T}$ is the Wick-renormalised Higgs potential defined by \eqref{eq:def_V_N_m} with $m=\lambda_{A,T}$,
while $\cQ_{N,T}$ is a renormalised logarithmic determinant ratio.
The reference mass $\lambda_{A,T}$ is subtracted in the Higgs potential.
The Wick constant from this subtraction combines with the mass counterterm in $\cQ_{N,T}$ to give $\rmp_N/2$, independent of both fields and absorbed into the normalisation.
The determinant term comes from comparing the normalising constants of the covariant and free Gaussian fields.
The formula for $\cQ_{N,T}$, including the cutoffs and counterterms, is given in \eqref{eq:determinant_computation} with $B_T=0$.

The main result of this paper may be summarised as follows.
Below, we write \(\cS'(\T^2)\) for the space of distributions on \(\T^2\) and \(\Omega^1\cS'(\T^2)\) for the space of distributional \(1\)-forms.

\begin{theorem}\label{thm:intro_main}
The Laplace transforms of the measures \(\mu^{\YMH}_{N,T}\) converge as \(N,T\to\infty\).
In particular, $\mu^{\YMH}_{T,N}$ converges weakly to a probability measure $\mu^{\YMH}$ on $\Omega^1 \cS'(\mathbb{T}^2)\times \cS'(\mathbb{T}^2)$. Furthermore, the gauge marginal of $\mu^{\YMH}$ is absolutely continuous with respect to $\mu^{\rmA}$.
\end{theorem}

The convergence is proved in Theorem~\ref{thm:main}, while the density of the gauge marginal with respect to $\mu^{\rmA}$ is given in \eqref{eq:gauge_marginal_density} in Section~\ref{sec:densities}.

We construct the joint law of the gauge and Higgs fields in the Coulomb gauge \eqref{eq:Coulomb} and describe its disintegration. Conditioned on the enhanced gauge field, the Higgs field has a renormalised $P(\phi)_2$ law with the covariant Gaussian free field as reference measure. Integrating out the Higgs field gives the gauge marginal, which is the pure Yang--Mills measure weighted by the conditional Higgs partition function and the renormalised determinant ratio.
\subsection{Main ingredients} \label{sec:intro_summary}
We now describe the main elements of the proof of Theorem  \ref{thm:intro_main}.
\paragraph{Bou\'e-Dupuis formula.}
The principal tool that we use in our construction is a variational representation of the free energy of a probability measure defined in terms of an underlying Brownian motion. To state this variational formula in sufficiently general terms, let $(X_t)_{t\geq 0}$ be a cylindrical $L^2$ Brownian motion defined on a filtered probability space $(\Omega,\cF,\bP)$, and we set $Y_\Cdot:=\int_0^\Cdot \sigma_t \dif X_t$ for a suitable operator $\sigma_t$. Let $\bH_a$ be the progressively measurable processes in $L^2(\R^+ \times\mathbb{T}^2)$, and we also define
\[
\|u\|_{\bH}^2:=\int_0^\infty \|u_s\|_{L^2}^2\dif s, \qquad \rmJ_\cdot(u)=\int_0^\cdot \sigma_s u_s\dif s.
\]
We say that $W$ is a \emph{tame} functional, if
\[
\E[|W(Y_\Cdot)|^p]+\E[e^{-qW(Y_\Cdot)}]<\infty
\]
holds for some $p,q\ge 1$ such that $1/p+1/q=1$. For such a tame functional $W$, the Bou\'e-Dupuis formula, see~\cite{BoueDupuis98,Ustunel14}, then provides that
\[
-\log \bE[e^{-W(Y_\Cdot)}]=\inf_{u\in \bH_a} \E\Big[W(Y_\Cdot+\rmJ_\Cdot(u))+\frac 1 2 \|u\|_{\bH}^2\Big],
\]
where the expectation is with respect to the underlying measure $\bP$.

For a suitable choice of test function, this formula yields an iterated variational representation of the Laplace transform of $\mu^\YMH_{N,T}$. More precisely, we first apply the Bou\'e-Dupuis formula to the outer expectation with respect to the Gaussian part of $\mu_T^\YM$. This leads us to study the free energy of the conditional Higgs field (later called $\cW_N$) and the ratio of determinants $Q_{N,T}$ for the gauge field $A_T$ shifted by a random Cameron-Martin drift, which is represented throughout by the symbols $B_T$ or (in pathwise results) $B$. We apply the Bou\'e-Dupuis formula again to obtain a variational problem for the free energy of the conditional Higgs field in terms of the interaction potential $V(\phi)$ and a Gaussian field (later called $\phi_{\bA,B,N}$), whose covariance is determined by the shifted gauge field $A_T+B_T$.

This approach encounters the following technical challenge. Controlling the inner variational problem is aided by the coercivity of the interaction potential supplied by its leading term $c_n |\phi|^{2n}$. A corresponding term is absent in the outer variational problem, in which the only coercivity is provided by the norm of the drift, effectively, $\tfrac 1 2 \norm{B_T}_{\Omega^1 H^1}^2$. We therefore take care to prove all relevant bounds with sufficiently mild (sub-quadratic) dependence in the norm of the gauge drift. We comment on some key steps in this overall plan next.

\paragraph{The renormalised covariant Laplacian.}
We introduce spaces of rough (enhanced) gauge fields $\bA = (A,A^2)\in\cX^{-\kappa} = \Omega^1_\Coul C^{-\kappa/2}\times C^{-\kappa/2}$
($\Coul$ is for `Coulomb'),
for which we construct the operator corresponding formally to $L_{A} = \diff_{A}^*\diff_{A}$.
The space is moreover stable under Cameron--Martin shifts $B\in\Omega^1H^1$
and $L_{A+B}$,
after adding a suitable mass, is positive and has a resolvent satisfying the same elliptic estimates
as a classical second-order operator, up to polynomial factors in the enhanced noise $\bA$ and a sufficiently small power of $\|B\|_{H^1}$,
see Theorems \ref{thm:resolvent_estimate} and \ref{thm:resolvent_estimate_pushed}.
These estimates use the Sobolev spaces $H_B^s$ defined by $L_B$.
Comparison with ordinary Sobolev norms lets us apply standard multiplication estimates with controlled dependence on $B$, as in Lemma~\ref{lem:maps_0_order}.
We also prove continuity of the resolvent with respect to the enhanced field.
These estimates are central to our
analysis of the variational problem for both the Higgs and gauge fields.

\paragraph{The conditional Higgs measure.}
For a frozen rough gauge field, we extend the construction in \eqref{eq:intro_covariant_GFF} using the renormalised covariant Laplacian from Section~\ref{sec:domain_renorm_laplacian}.
In the notation of Section~\ref{sec:conditional_Higgs}, $\phi_{\bA,B,N}$ denotes this field, with $\bA$ recording the gauge field and its renormalisation data, and $B$ the more regular shift used in the variational problem.
To remove the cutoff $N$, we decompose it into the classical covariant Gaussian free field for $B$ and a more regular correction:
\begin{equation}\label{eq:intro_GFF_decomposition}
    \phi_{\bA,B,N}=\phi_{\0,B,N}+Z_{\bA,B,N},
\end{equation}
where $\phi_{\0,B,N}$ is the Gaussian free field associated with the classical covariant Laplacian $L_B$,
while $Z_{\mathbb{A},B,N}$ is uniformly bounded in $H_B^s$ for $s<1$ and depends continuously on $\bA$ in stochastic $L^p$-spaces, see Theorem \ref{thm:cov_GFF_decomposition}.
The singular part of the field is therefore a two-dimensional Gaussian free field associated to the classical covariant Laplacian $L_B$.
A caveat is that our Wick renormalisation constants are taken independent of $\bA,B$,
so an important step is to control the difference between the `flat' Wick constant for $L_0$ and the covariant Wick constant associated to $L_B$.
In the end, we obtain convergence of the Laplace transform of the conditional Higgs measure to a limit that depends continuously on $(\bA,B)$,
see Theorem \ref{thm:higgs_convergence}.

\paragraph{The determinant ratio.} The covariant Gaussian reference has a formal partition function proportional to $(\det L_{\mathbb A})^{-1}$, since the Higgs field has two real components. We extract a finite part by formally multiplying by $\det L_0$. Appendix~\ref{app:galerkin} gives the corresponding renormalised finite-dimensional identity. The primary difficulty is then to show that $\left(\frac{\det{L_{\mathbb{A}}}}{\det{L_0}}\right)^{-1}$, is finite for a.e $\mathbb{A}$, regular as one varies $\mathbb{A}$, and integrable w.r.t to taking expectation over $\mathbb{A}$. To achieve the first two points, we will write \begin{equation}
 \frac{\det L_{\mathbb{A}}}{\det{L_0}}=\exp(\Tr \log L_{\mathbb{A}} -\Tr \log L_{0}). \label{eq:trace-det}\end{equation}
 We then apply the resolvent identity three times and estimate the resulting terms using resolvent analysis. The remainder of the expansion can be estimated using resolvent analysis as well. To show integrability we again take advantage of the Boué-Dupuis formula. This will reduce the problem to studying
\[ \Tr \log L_{\mathbb{A}+B}-\Tr \log L_{0}, \]
where $B\in H^{1}$, and we have to achieve a bound which grows sub-quadratically in $\|B\|_{H^1}$, which is the main difficulty of the analysis. These steps are carried out in detail in Section \ref{sec:determinants}.
Keeping the gauge cutoff $T$ and a realisation of the smooth gauge field $A_T$ fixed, we can instead approximate the Higgs field by Fourier projection.
This gives finite-dimensional measures with ordinary determinant ratios.
With matching renormalisation, adapting the proof of Theorem~\ref{thm:higgs_convergence} gives the same conditional Higgs limit, while Lemma~\ref{lem:ratio_det_Galerkin} in Appendix~\ref{app:galerkin} identifies \eqref{eq:trace-det} as a limit of the ratio of finite dimensional determinants, with a suitable (finite) quadtatic counterterm, see Remark \ref{rem:det-counterterm}.
This connects our construction to finite-dimensional regularisations of the Yang--Mills--Higgs action.
We use the resolvent cutoff in \eqref{eq:intro_covariant_GFF} because it preserves gauge covariance.

Finally having discussed these ingredient we assemble them in Section \ref{sec:construction_measure} to obtain Theorem \ref{thm:intro_main}.
\subsection{Properties of the measure}
Our methods allow us to deduce interesting properties of the measure beyond Theorem \ref{thm:intro_main}, which we now discuss.  These results are, to the best of our knowledge, new results for the Abelian Yang--Mills--Higgs measure on $\T^2$.

\paragraph{Exponential integrability.}
Theorem~\ref{thm:main}, with the test-function class $C_{{\rm test}}$ defined in Section~\ref{sec:gauge_variational_setup} by \eqref{eq:ymh_test_functions}, gives the following exponential integrability estimate.
\begin{theorem}[Exponential integrability]
    Let $\kappa>0$ and $n$ be the degree of the polynomial $\cV$ defined in \eqref{eq:def_V}. There exists $\eta_\rmA,\eta_\Phi>0$ such that
    \[
    \E_{\mu^\YMH}[e^{\eta_\rmA\|A\|_{C^{-\kappa}}^2+ \eta_\Phi \|\phi\|_{C^{-\kappa}}^{2n}}]<\infty.
    \]
\end{theorem}

\paragraph{Gauge covariance.}
Our regularised measure has a natural gauge covariance property that actually allows us to construct the measure in gauges beyond the Coulomb gauge. Given $g\in C(\T^2;U(1))$ we can construct a smooth approximation $g_T\to g$ as $T\to\infty$ in $C(\T^2;U(1))$. We define the (gauge-transformed) measure
\begin{equation}\label{eq:intro_regularised_measure_gauge_transformed}
    \dif\mu^{\YMH[g]}_{N,T}(A,\phi)
    =\frac{1}{\cZ_{N,T}^{[g]}}
    \exp\bigl(-\cQ_{N,T}(A)-\cV_{A,N,T}(\phi)\bigr)
    \dif\mu^{\GFF}_{A,N,T}(\phi)\dif\mu_T^{\rmA[g]}(A),
\end{equation}
where $\mu_T^{\rmA[g]}$ is the law of $A_T^{g_T}$ when $A_T$ has law $\mu_T^{\rmA}$.
\begin{proposition}\label{prop:gauge_covariance}
    The measure $\mu^{\YMH[g]}_{N,T}$ converges (in the same sense as $\mu^\YMH_{N,T}$) as $N,T\to\infty$. Furthermore, for any continuous and bounded function $F$ on $\Omega^1\cS'\times \cS'$
    \[
    \int F(A,\phi)\dif \mu^{\YMH[g]}(A,\phi)=\int F(A^g,\phi^g) \dif\mu^\YMH(A,\phi).
    \]
  In particular, $\mu^{\YMH[g]}=\mu^{\YMH}$ on gauge-invariant observables, i.e. whenever $F$ satisfies $F(A,\phi)=F(A^g,\phi^g)$.
\end{proposition}

\begin{remark}
Proposition~\ref{prop:coulomb_representative} gives the gauge-fixing statement for $H^1$ configurations. The corresponding formal integral calculation is given in Appendix~\ref{app:gauge_fixing}, see \eqref{eq:formal_computation_gauge_change}.
\end{remark}
We use Proposition~\ref{prop:gauge_covariance} to construct measures in linear covariant gauges, with formal action
\begin{align}\label{eq:YMH_action_alpha}
\cS_\YMH^{\alpha}(A,\phi):=\cS_\YMH(A,\phi)+\alpha \|\diff^*A\|_{L^2}^2,
\end{align}
where $\alpha>0$. Let $\zeta$ be a $\bfi\R$-valued white noise with zero spatial mean independent of the random variables defined above. We denote its law by $\nu$.  We define the random variable $\scg^{\alpha^{-1/2}\zeta}=e^{-u}$ where $u=u^{\alpha^{-1/2}\zeta}$ satisfies
\[
\diff u=\diff (\diff\diff^*)^{-1}(\alpha^{-1/2}\zeta).
\]
 We choose a smooth  approximation $u_T\to u$ which naturally yields $\scg_T^{\alpha^{-1/2}\zeta}$. We call this measure $\mu^{\YMH(\alpha)}_{N,T}$ and it is defined by
\[
\mu^{\YMH(\alpha)}_{N,T}:=\int\mu^{\YMH[\scg^{\alpha^{-1/2}\zeta}]}_{N,T}  \dif\nu(\zeta).
\]
Then this allows us to conclude with the following theorem:
\begin{theorem}\label{thm:linear_covariant_gauges}
The measure $\mu^{\YMH(\alpha)}_{N,T}$ has a unique limit as $N,T\to\infty$ to a probability measure denoted by $\mu^{\YMH(\alpha)}$. In fact, one has for any bounded continuous function $F$ on $\Omega^1\cS'\times \cS'$
\[
\int F(A,\phi)\dif\mu^{\YMH(\alpha)}(A,\phi)=\int F(A^{\scg^{\alpha^{-1/2}\zeta}}, \phi^{\scg^{\alpha^{-1/2}\zeta}})  \dif\mu^{\YMH}(A,\phi)  \dif\nu(\zeta).
\]
In particular, the measures coincide on gauge-invariant observables. Finally,
\[
\lim_{\alpha\to \infty} \int F(A,\phi)\dif\mu^{\YMH(\alpha)}(A,\phi)= \int F(A,\phi)  \dif\mu^{\YMH}(A,\phi) .
\]
\end{theorem}

Proposition~\ref{prop:gauge_covariance} and Theorem~\ref{thm:linear_covariant_gauges} are proved in Section~\ref{sec:gauge_covariance}, starting from the finite-cutoff identities \eqref{eq:gauge_covariance_cutoff} and \eqref{eq:linear_covariant_gauges_cutoff}, respectively.

\begin{remark}
    Formally, our Colomb gauge construction from Theorem~\ref{thm:intro_main} can be obtained $\alpha=\infty$ in the action~\eqref{eq:YMH_action_alpha} for the formal measure in consideration. In fact, the previous theorem makes that rigorous.
\end{remark}

\paragraph{Large deviations in the semi-classical limit.} Since our analysis is based on the Boué-Dupuis formula, it provides us with a variational representation of the measure. This allows us to deduce a large deviation principle in a natural way. In particular consider $\mu^{\YMH;\varepsilon}$, which is formally given by  \begin{equation*}
    \dif\mu^{\YMH;\varepsilon}(A,\phi)
    \propto
    \exp\bigl(-\varepsilon^{-1}\cS_{\YMH}(A,\phi)\bigr) \dif A \dif\phi.
\end{equation*}
Our construction straightforwardly extends to the $\varepsilon\neq 1$ case. In Section \ref{sec:large_deviations} we will show the following theorem:
\begin{theorem}[Large deviations]\label{thm:large_deviations}
    The measures $\mu^{\YMH;\varepsilon}$ satisfy a Laplace principle principle for every bounded, Lipschitz and gauge invariant function $f=f(A,\phi)$. More precisely
    \begin{align*}
    &-\lim_{\varepsilon \to 0} \varepsilon\log \int \exp(-\varepsilon^{-1}f(A,\phi))\dif\mu^{\YMH;\varepsilon}\\  =& \inf_{(A,\phi)\in\Omega^1 H^1\times H^1}
\bigl\{f(A,\phi)+\cS_{\YMH}(A,\phi)\bigr\}
-\inf_{(A,\phi)\in\Omega^1 H^1\times H^1}
\cS_{\YMH}(A,\phi).\end{align*}
\end{theorem}

\subsection{Related work}
We now give a (far from complete) overview of  related literature, we also point to \cite{chatterjee2019} for a review of work on gauge theory targeted towards probabilists.
\begin{itemize}
\item Some key earlier probabilistic constructions of two dimensional pure Yang-Mills are \cite{Driver89,GKS89} on the plane and \cite{Sengupta97,Levy03} on compact surfaces.
An approach that has inspired much later work is \cite{Chevyrev19} which  constructs the theory on a torus as a random distribution with optimal regularity. Some very recent follow-ups combine this approach with ideas from Morse theory to apply to general compact surfaces \cite{chhaibi2026yangmills} and prove discrete to continuum universality results \cite{DangNohra26}.
Another active probabilistic approach is to represent expectations of Wilson loop observables as sums over random surfaces, see \cite{PPSY26} where this is carried out in the continuum.

\item The variational method has conceptual similarities with stochastic quantisation. In that context, local dynamics for gauge theories were constructed in two and three dimensions in \cite{CCHS2d,CCHS3d}, we also point to \cite{Sourav_state,Sourav_flow} for closely related work.
A paracontrolled approach to the two-dimensional Yang--Mills dynamics was developed in \cite{BC23}. For pure Yang--Mills theory in two dimensions, Chevyrev and Shen \cite{CS23} used Bourgain's invariant measure argument to prove global existence on gauge orbits and invariance of the Yang--Mills measure. For the two-dimensional abelian YMH model, global well-posedness of the dynamics was proved in \cite{BC24_YM}.
Their approach and ours use covariant objects and spaces adapted to the gauge-field remainder to avoid estimates that grow with its size.
A different approach using loop expansions gave a construction of the gauge ($A$) marginal in \cite{CC24}.

\item The renormalised covariant Laplacian associated to the Gaussian part is analysed in~\cite{morin_2d_2022} as an unbounded operator. However, we have to refine their analysis since we obtain good bounds in $B$ as described above. This falls into a broader class of random (low regularity) perturbations of the laplacian, which also includes the Anderson Hamiltonian. The $\Phi^4_2$ measure build from the Anderson hamiltonian \cite{allezchouk}, somewhat akin to our Higgs field, has been constructed in \cite{barashkov2023invariant,EMR24}, however an additional difficulty present for us, it to show strong enough bounds to be able to integrate with respect to the gauge field. In fact the random gauge Laplacian is expected to posses significantly better tails, courtesy of the the diamagnetic inequality.
\item The determinant ratio has been studied for a smooth gauge fields in \cite{Ito1980abs,SP24} as well as in \cite{BFS79} for lattice approximations. We point out that the gauge laplacian unfortunately does not fall into the class of geometric operators where explicit formulas for the determinant are available, unlike the Laplace- Beltrami operator \cite{OSGOOD1988148,POLYAKOV1981207}.
\item The variational method for constructing QFTs, we employ here has been developed in \cite{BG18}, and subsequently expanded upon in \cite{barashkov2022stochastic,barashkov2023variational,chandra2026stochastic}. We further build on this method by considering an iterated variational formulation for the  construction of the joint Yang--Mills--Higgs measure.
A related approach is through Forward-Backward  SDEs, which represent the Euler-Lagrange equations corresponding to the variational problem. This has been discussed in \cite{barashkov2023variational,DV2025stochastic,GubinelliMeyer2024FBSDE,Meyer2026FBSDE}. Both of these approaches are intimately linked with the Polchinski equation, and corresponding stochastic dynamics. See \cite{bauerschmidt2024stochastic,duch2024construction} for recent uses of the Polchinski equation in constructive QFT and probability, and \cite{frob2016all} for uses in perturbative renormalisation of gauge theory.

\item Large deviations theorems have been by now derived for a variety for QFTs in the semi-classical limit  \cite{Barash,barashkov2023variational,barashkov2022stochastic,klose2024large,lacoin2022semiclassical}. Of particular relevance here is \cite{levy2006large} where an LDP is derived for the YM measure without the Higgs  field. See also \cite{dang2026semiclassical} for a study of the semi-classical limit of the Yang-Mills measure on general 2 dimensional surfaces.

\end{itemize}

\subsection{Notation}

\paragraph{Basic notations.}
We use $\bR$, $\bC$ and $\bfi \R$ for the real, complex and purely imaginary numbers, respectively. We write $a\lesssim b$ for given $a,b\in\R$ if there exists $C>0$ such that $a\leq Cb$. We also write $a\lesssim_{\alpha,\beta,...}b$ to emphasise the dependence of parameters $\alpha,\beta,...$ on the implicit constant $C$. A common used notation is $\tau(\kappa)$ for given $\kappa\in (0,1)$ for any quantity such that $\tau(\kappa)\to 0$ as $\kappa\to0$. We also use $\rmn(\kappa)$ for a quantity that is allowed to be large.

\paragraph{Function spaces.} For $s\in\R$ and $p,q\in [1,\infty]$, we denote by $B^s_{p,q}$ the usual Besov space and $W^{s,p}$ the Bessel potential space (see e.g.~\cite{BookChemin}). Furthermore, we set $H^s:=W^{s,2}$ and $C^s:=B^s_{\infty,\infty}$ for any $s\in\R$. We write $C^\infty$ for smooth functions.
We recall the usual Young multiplication estimate: if $p\in (1,\infty)$, $s,\alpha\in\R$ satisfy $s+\alpha>0$, and $\delta>0$, then pointwise multiplication of smooth functions admits a continuous bilinear extension
$
C^\alpha\times W^{s,p}
\longrightarrow
W^{s\wedge\alpha-\delta,p},
$
and
\[
\|uv\|_{W^{s\wedge\alpha-\delta,p}}
\lesssim_{\alpha,s,\delta,p}
\|u\|_{C^\alpha}\|v\|_{W^{s,p}}.
\]

\phantomsection\label{def:one_forms}Given a function space $X$, we denote by $\Omega^1 X$ the space of $1$-forms with components in $X$. Our $1$-forms are $\bfi\R$-valued unless another target space is specified, in which case we write $\Omega^1 X(\T^2;\bK)$ where $\bK$ is either $\R$ or $\bC$. We set
\[
\|K\|_{\Omega^1X}:=\left(\sum_{j=1}^2 \|K_j\|_{X}^2\right)^{1/2},
\]
and frequently, we abbreviate to $\|K\|_X$ whenever it is clear that $K$ is a $1$-form.

\phantomsection\label{def:coulomb_forms}Finally, we denote by $\Omega^1_\Coul X$ for the elements $A\in \Omega^1 X$ in the Coulomb gauge, i.e.~$\diff^*A=0$.

\paragraph{Adjoints.}
For $z\in\bC$, we write $z^*\in\bC$ for the complex conjugation. We equip space of functions with the standard complex inner product $\av{\Cdot,\Cdot}$. Moreover, for the space of $1$-forms, we equip the componentwise standard inner product. If $K$ is a smooth $1$-form, we write $K^*$ for the formal adjoint of the operation $C^\infty(\T^2;\bC)\ni u\mapsto Ku\in \Omega^1C^\infty(\T^2;\bC)$.  In particular, given $K,F\in \Omega^1C^\infty(\T^2;\bC)$, we have
\[
(K^*F)(x)=\sum_{j=1}^2 K_j(x)^* F_j(x).
\]
Finally, $\diff$ is the standard exterior derivative and $\diff_B$ for a given $B\in\Omega^1 X$ is the covariant derivative. Their formal adjoints are given in terms of $\diff_B^*$ and $\diff^*$.

\paragraph{Linear operators.}
For Banach spaces $X,Y$, we write $\rmL(X;Y)$ as the space of bounded operators $X\to Y$ equipped with the standard operator norm, which we denote by $\|X\|_{X\to Y}$. Moreover, when we have an operator from a Banach space $H\to H$, we write $\|\Cdot\|_{\HS}$ for the Hilbert-Schmidt norm and $\Tr|\Cdot|$ for the trace norm. In the latter case, the trace $\Tr(\Cdot)$ is well-defined. Traces, determinants and Hilbert--Schmidt norms on complex spaces are taken over $\bC$.

An unbounded operator on a Banach space $X$ is denoted by $(L,\scD(L))$, where $L:\scD(L)\to X$ and $\scD(L)\subset X$.

\paragraph{Fourier transform.}
Our Fourier transforms (and inversions) on $\T^2$ and $\R^2$ are given by
\begin{equation*}
\begin{split}
\hat f(k) &= \int_{\mathbb{T}^2}\; f(x) e^{-i k\cdot x}\,\dif x
\,\textup{, with inverse }\;
f(x) = \sum_{k\in 2\pi\Z^2} \hat f(k) e^{ik\cdot x}\;,\\
\hat g(k) &= \int_{\R^2}\; g(x) e^{-ik \cdot x} \,\dif x \,\textup{, with inverse }\; g(x) = \frac{1}{4\pi^2}\int_{\R^2}\; \hat g(k) e^{ik \cdot x} \,\dif k\;.
\end{split}
\end{equation*}

\medskip

\paragraph{Acknowledgments.}
I.C. acknowledges support from the European Research Council (ERC) via the Starting Grant SQGT 101116964. A.K. acknowledges support from UK Research and Innovation via the grant ``StochFields'' EP/Z534328/1, and further thanks the Max Planck Institute for Mathematics in the Sciences for its hospitality during the preparation of part of this work.
\paragraph{AI usage disclosure.}
 The results of Appendix \ref{app:galerkin} were obtained using ChatGPT 5.6 Sol and 6. These same models, in addition to Claude Fable 5, were also used in generating the symbolic index, auditing specific arguments and to streamline or combine various technical estimates derived by the authors. Other than Appendix~\ref{app:galerkin}, all other arguments originate with the authors. The authors take full responsibility for the content of the paper.

\section{Classical covariant analysis}

\subsection{The classical covariant Laplacian }
For a $1$-form $B\in\Omega^1 C^\infty$, the covariant Laplacian is
\[
L_B u = \diff_B^*\diff_Bu =-\Delta + B^*\diff u+\diff^*(Bu)+B^*Bu, \qquad u\in H^2,
\]
with domain $H^2$ on $L^2$. For $B\in\Omega^1 L^2$, we define $L_B$ as the self-adjoint operator associated with the closure of the quadratic form $u\mapsto\|\diff u+Bu\|^{2}_{L^{2}}$ on $C^{\infty}(\T^{2})$. The form is closable since $\diff+B$ on $C^{\infty}(\T^{2})$ has a densely defined adjoint. By \cite[Thm~15.4]{simon2005functional}, its closure agrees with the maximal form, so $C^{\infty}(\T^{2})$ is a form core. For $B\in\Omega^1H^1$, standard elliptic arguments show that $(L_B,H^2)$ is a self-adjoint non-negative unbounded operator with core $C^\infty$.

Since $L_B+1$ has spectrum bounded below by $1$, the usual construction defines the associated Sobolev spaces $H^s_B$, $s\in\R$, with norm
\[
\|u\|_{H^s_B}:=\|(L_B+1)^{s/2}u\|_{L^2}, \qquad u\in H^s_B.
\]
For $s=2$, this is the domain of $L_B$, so $H^2_B=H^2$.

For an unbounded linear operator $(X,\scD(X))$, define its resolvent set by
\[
\rho(X):=\{z\in \bC\colon X+z \text{ invertible with bounded inverse } L^2\to L^2\}.
\]
\phantomsection\label{def:classical_resolvent}For $z\in \rho(X)$, write $R(X,z):=(X+z)^{-1}$. For $(L_B,H^2_B)$, we use the notation $R_B(z):=R(L_B,z)$. The following estimate holds.
\begin{lemma}[Classical resolvent estimate]\label{lem:resolvent_estimate_classical}
Let $B\in\Omega^1H^1$,
and $s\in \R$. For any $\Re(z)>0$ such that $|z|\geq 1$ and $\eta\in [0,2]$,  we have
\[
\|R_B(z)f\|_{H^s_B}\lesssim |z|^{\frac\eta 2-1}\|f\|_{H^{s-\eta}_B}.
\]
\end{lemma}

\begin{proof}
    This is standard as we are working on the ``{intrinsic}'' Sobolev spaces (alternatively,  one can prove it with functional calculus).  \end{proof}

\subsection{Diamagnetic inequalities}
We use the semigroup representation of the covariant resolvent. Let $p_B(t,x,y)$ be the covariant heat kernel on $\mathbb{T}^2$, that is, the fundamental solution of $(\partial_t+L_B)u=0$. We omit the spatial variables when referring to the corresponding operator. Thus $p_0$ is the classical heat kernel for $L_0=-\Delta$, and by translation invariance we also write $p_0(t,x-y)$ for $p_0(t,x,y)$. We recall the following resolvent formulas from \cite[Sec.~15]{simon2005functional}.

\begin{lemma}\label{lem:semi-group}
    For any $\lambda>0$ and any $B\in \Omega^1L^{2}$ taking values in $i \mathbb{R}$ almost everywhere, the operator $(L_B+\lambda)$ is self-adjoint, positive-definite, and generates a strongly continuous semigroup. The resolvent of $L_B$ has the following $L^{2}$-kernel representation and, for any $\gamma>0$, fractional-power operator representation:
    \begin{align*}
    R_B(\lambda)(x,y) &= \int_0^\infty\; e^{-\lambda t}p_B(t,x,y)\,\diff t, \\
    R_B (\lambda)^\gamma &= \frac{1}{\Gamma(\gamma)}\int_0^\infty\; t^{\gamma -1} e^{-\lambda t} p_B(t)\,\diff t.
    \end{align*}
    Here $\Gamma$ is the usual Gamma function. The second identity holds for the kernels as well.
\end{lemma}

The following diamagnetic inequality is well-known in the case of $B\in\Omega^1C^\infty$, but since our $B$ is not so smooth, we include the proof. \begin{lemma}[Diamagnetic inequalities]\label{lem:diamagnetic}
Let $B \in \Omega^1 H^1$ with $B$ taking values in $\bfi \mathbb{R}$ almost everywhere. Then, for all $t > 0$, $\lambda>0$, $\gamma>0$ and almost every $(x,y)\in\mathbb{T}^2 \times \mathbb{T}^2$,
    \[
    |p_B(t,x,y)| \le p_0(t,x-y), \qquad
    |R_B(\lambda)^\gamma(x,y)| \le R_0(\lambda)^\gamma(x-y).
    \]
\end{lemma}
\begin{proof}
We first recall the Feynman-Kac representation for smooth $B\in\Omega^1C^\infty$
\begin{equation}\label{eq:FK}
p_B(t,x,y) = p_0(t,x,y) \bfE_{t;y\to x}\Big[\exp\Big(-\int_0^t\; B^*(W_s) \,\diff W_s -\int_0^t\; \diff^* B(W_s) \,\diff s\Big)\Big]\;,
\end{equation}
where $W$ is a Brownian bridge under which the expectation $\bfE_{t;y\to x}$ goes from $y$ to $x$ on the time interval $[0,t]$.
Note that the right-hand side of \eqref{eq:FK} is defined for every $t>0$ and $x,y\in\T^{2}$ (and is unchanged if $B$ is changed on a null set), while for  $B\in \Omega^1H^1$ the left-hand side is an integral kernel and hence only defined up to null sets; \eqref{eq:FK} asserts that the former is a representative of the latter. For continuous $B$ and $\diff^{*}B$, \eqref{eq:FK} holds pointwise for all $x,y$, and the general case follows by taking $B_{\eps}\in\Omega^1C^{\infty}$ with $B_{\eps}\to B$ and $\diff^{*}B_{\eps}\to\diff^{*}B$ in $L^{2}(\T^{2})$: since the exponent is purely imaginary the expectations converge by dominated convergence along a subsequence, while the semigroups converge weakly by strong resolvent convergence; see \cite[Sec.~15]{simon2005functional} for details.\footnote{Note that in \cite[Thms~15.5 and~15.6]{simon2005functional}, the results are given on $\mathbb{R}^{d}$ rather than the torus and the extension to rough $B$ there takes $B \in L^{2}_{\mathrm{loc}}$ but in Coulomb gauge. However, the arguments therein clearly extend to our present setting.}

Since $B$, and hence $\diff^* B$, is $i \mathbb{R}$-valued, the exponent in \eqref{eq:FK} is purely imaginary, so the expectation in \eqref{eq:FK} has modulus at most one, which gives the first bound. Inserting it into the representation of Lemma~\ref{lem:semi-group} gives the second.
\end{proof}

Lemma~\ref{lem:semi-group}, the diamagnetic inequality and the flat-Laplacian traces give:
\begin{lemma}\label{lem:neg_power_trace}
Let $B\in\Omega^1H^1$ and $\gamma>1$. For any $\lambda\geq 1$, the trace bound and its Hilbert--Schmidt consequence are
\[
\Tr\big( R_B(\lambda)^{\gamma}\big)\lesssim \lambda^{1-\gamma}, \qquad
\| R_B(\lambda)^{\gamma/2}\|_\HS \lesssim \sqrt{\lambda^{1-\gamma}}=\lambda^{\frac 1 2-\frac \gamma 2}.
\]
\end{lemma}

\subsection{Sobolev comparisons}
For $B\in\Omega^1H^1$ and $s\in\R$, we compare $H^s_B$ with the standard Sobolev spaces to use the usual Sobolev and Besov estimates. We first recall Kato's inequality.
\begin{lemma}[Kato]\label{lem:Kato}
Let $B\in\Omega^1 L^2$ and $u\in W^{1,\infty}$.
Then for almost every $x$
\[
|\diff|u| (x) |\leq |\diff_Bu(x)|.
\]
\end{lemma}
\subsubsection{$0$-forms}
The following is the Sobolev comparison for $0$-forms:
\begin{lemma}\label{lem:Sobolev_comparison}
Let $p>2$ and $B\in \Omega^1 L^p$. Then for any $s\in [-1,1]$, we have $H^s=H^s_B$ and the following inequalities:
    \begin{align*}
    \|u\|_{H^s}&\lesssim (1+\|B\|_{L^p})^{|s|} \|u\|_{H^s_B}, \\
    \|u\|_{H^s_B}&\lesssim (1+\|B\|_{L^p})^{|s|} \|u\|_{H^s}.
    \end{align*}
\end{lemma}

\begin{proof}
    By $C^\infty$ being a core of the covariant Laplacian, it is enough to consider $u\in C^\infty$. For $s=1$, we have
    \[
    \|u\|_{H^1_B}^2=\|\diff_Bu\|_{L^2}^2+\|u\|_{L^2}^2, \qquad
    \|\diff |u|\|_{L^2}\leq \|\diff_B u\|_{L^2}\leq \|u\|_{H^1_B}.
    \]
    For $r\geq 1$ with $\frac 1 r=\frac{1-\theta}{2}$ and $\theta\in[0,1]$, the Gagliardo--Nirenberg inequality gives
    \[
    \|u\|_{L^r}\lesssim \|\diff|u|\|_{L^2}^\theta \|u\|_{L^2}^{1-\theta}.
    \]
    Since $1-\theta=\frac 2r$, for $r\geq 2$ this gives
    \[
     \|u\|_{L^r}\lesssim \|\diff|u|\|_{L^2}^{1-\frac 2r} \|u\|_{L^2}^{\frac 2r}\lesssim \|u\|_{H^1_B}.
    \]
    Choose $r$ with $\frac 12=\frac 1p+\frac 1r$. The assumption $p>2$ excludes $r=\infty$. Then
    \begin{gather*}
    \|Bu\|_{L^2}\leq\|B\|_{L^p}\|u\|_{L^r}\lesssim \|B\|_{L^p}\|u\|_{H^1_B}, \\
    \|\diff u\|_{L^2}=\|\diff_B u+Bu\|_{L^2}\leq \|\diff_B u\|_{L^2}+\|Bu\|_{L^2}\lesssim (1+\|B\|_{L^p})\|u\|_{H^1_B}, \\
    \|u\|_{H^1}\leq (1+\|B\|_{L^p})\|u\|_{H^1_B}.
    \end{gather*}

    Conversely,
    \[
    \|\diff_B u\|_{L^2}\leq \|\diff u\|_{L^2}+\|Bu\|_{L^2}\lesssim \|u\|_{H^1}+\|B\|_{L^p}\|u\|_{L^r},
    \]
    and Sobolev embedding, $\|u\|_{L^r}\lesssim \|u\|_{H^1}$, gives
    \[
    \|u\|_{H^1_B}\lesssim (1+\|B\|_{L^p})\|u\|_{H^1}.
    \]

    This proves the case $s=1$. For $s\in[0,1]$, apply Lemma~\ref{lem:norm_comparison_pos_op} to $L_B+1$ and $L_0+1$ in either order. For $s\in[-1,0)$, use duality and boundedness of the identity map.
\end{proof}

\begin{remark}
    In all of these bounds later on we take $B\in\Omega^1H^1$ and use Sobolev embedding to bound $\|B\|_{L^p}\lesssim \|B\|_{H^1}$.
\end{remark}

\subsubsection{$1$-forms}
For the following $1$-form estimates, $B\in\Omega^1H^1$ is in the Coulomb gauge. For $1$-forms $K$, define
\[
\|K\|_{\Omega^1 H^s_B}^2:=\|(\diff_B\diff_B^*+\diff_B^*\diff_B+1)^{s/2} K \|_{L^2}^2, \qquad s\in \R.
\]
The operator $\diff_B\diff_B^*+\diff_B^*\diff_B+1$ with domain $\Omega^1 H^2$ is self-adjoint and has $\Omega^1C^\infty$ as a core, so the usual construction gives the associated Sobolev spaces.

The following proposition is crucial:
\begin{proposition}\label{prop:derivative_regularity_reduction}
    For any $s\in [1,2]$, it holds that
    \[
    \|\diff_B u\|_{H^{s-1}_B}\lesssim (1+\|B\|_{H^1})^{s-1}\|u\|_{H^s_B}.
    \]
\end{proposition}

\begin{proof}
    We prove the endpoints and then interpolate. The case $s=1$ follows from
    $\|\diff_B u\|_{L^2}\leq \|u\|_{H^1_B}$ by definition.

    For $s=2$, the identities $\diff_B^*\diff_B u=L_B$ and $\diff_B \diff_B u= \diff_B^2 u=F^Bu$ give
\begin{align}\label{eq:energy_d_B_identity}
\|\diff_B u\|_{\Omega^1 H^1_B}^2=\|\diff_Bu\|_{L^2}^2+\|L_Bu\|_{L^2}^2+\|F^Bu\|_{L^2}^2 .
\end{align}
We bound the three terms separately. The scalar Sobolev norm and Cauchy--Schwarz give
\begin{gather*}
\|L_Bu\|_{L^2}\leq \|(L_B+1)u\|_{L^2}=\|u\|_{H_B^2}, \qquad
\|u\|_{L^2}\leq \|u\|_{H_B^2}, \\
\|\diff_Bu\|_{L^2}^2=\langle L_Bu,u\rangle
\leq \|L_Bu\|_{L^2}\|u\|_{L^2}
\leq \|u\|_{H_B^2}^2.
\end{gather*}
It remains to bound the curvature term. By Hölder,
\[
\|F^Bu\|_{L^2}\leq \|F^B\|_{L^2}\|u\|_{L^\infty}\leq \|B\|_{H^1}\|u\|_{L^\infty}.
\]
The diamagnetic inequality gives
\[
|u|
=
|(L_B+1)^{-1}(L_B+1)u|
\leq
(L_0+1)^{-1}|(L_B+1)u|,
\]
and since in two dimensions the kernel of $(L_0+1)^{-1}$ belongs to $L^2$, Young's inequality yields
\[
\|u\|_{L^\infty}
\lesssim
\|(L_B+1)u\|_{L^2}
=
\|u\|_{H_B^2}.
\]
Combining these estimates gives
\[
\|F^Bu\|_{L^2}
\lesssim
\|B\|_{H^1}\|u\|_{H_B^2}, \qquad
\|\diff_Bu\|_{\Omega^1 H_B^1}^2
\lesssim
\bigl(1+\|B\|_{H^1}^2\bigr)\|u\|_{H_B^2}^2,
\]
proving the case $s=2$.

The proof for the general case $s\in [1,2]$ follows by interpolation from Lemma~\ref{lem:interpolation_hilbert_scale}.
\end{proof}

Now we can compare the Sobolev space $\Omega^1H^s_B$ with the classical $\Omega^1H^s$ for $s\in [-1,1]$.

\begin{lemma}\label{lem:Sobolev_comparison_1forms}
Let $s\in [-1,1]$. One has $\Omega^1H^s=\Omega^1H^s_B$, and for any $\delta>0$, there exists $C_\delta>0$ such that
\begin{align*}
\|K\|_{\Omega^1 H^s}&\leq C_\delta (1+\|B\|_{H^1})^{(1+\delta)|s|}\|K\|_{\Omega^1 H^s_B}, \\
\|K\|_{\Omega^1 H^s_B}&\leq C_\delta (1+\|B\|_{H^1})^{(1+\delta)|s|}\|K\|_{\Omega^1 H^s}.
\end{align*}
\end{lemma}

\begin{proof}
   We again consider smooth $K\in\Omega^1C^\infty$ and argue the statement by approximation. By definition and the triangle inequality,
    \begin{align*}
    \|K\|_{\Omega^1 H^1_B}^2&=\|\diff_B K\|_{L^2}^2+\|\diff_B^* K\|_{L^2}^2+\|K\|_{L^2}^2, \\
    \|K\|_{\Omega^1 H^1} &\leq  \|K\|_{\Omega^1 H^1_B}+\|BK\|_{L^2}+\|B^*K\|_{L^2}.
    \end{align*}
    For $p,r$ with $\frac 1p+\frac 1r=\frac 1 2$,
    \[
    \|BK\|_{L^2}\leq \|B\|_{L^p}\|K\|_{L^r}.
    \]
    For $\theta\in[0,1]$ with $\frac{1}{r}=\frac{1-\theta}{2}$,
    \[
    \|K\|_{L^r}\lesssim \|\nabla |K|\|_{L^2}^\theta \|K\|_{L^2}^{1-\theta}.
    \]
    Kato's inequality  $|\nabla |K||\leq |\nabla_BK|$, with $\nabla_B$ denoting component-wise covariant derivatives, followed by direct computation gives
    \begin{gather*}
\|\nabla |K|\|_{L^2}^2\leq \|\nabla_BK\|_{L^2}^2\lesssim \|\diff_B K\|_{L^2}^2+\|\diff^*_BK\|_{L^2}^2+\int |\av{F^BK_1,K_2}|.
    \end{gather*}
    Since
    \[
    \Re \int \av{F^BK_1,K_2}\lesssim \|F^B\|_{L^2}\|K\|_{L^4}^2, \qquad
    \|K\|_{L^4}^2\lesssim \|\nabla |K|\|_{L^2}\|K\|_{L^2},
    \]
    we obtain
    \[
    \|\nabla |K|\|_{L^2}^2\lesssim \|\diff_B K\|_{L^2}^2+\|\diff^*_BK\|_{L^2}^2+\|F^B\|_{L^2}\|\nabla |K|\|_{L^2}\|K\|_{L^2}.
    \]
    Young's inequality gives
    \[
    \|\nabla |K|\|_{L^2}^2\leq C(\|\diff_B K\|_{L^2}^2+\|\diff^*_BK\|_{L^2}^2+\|F^B\|_{L^2}^2\|K\|_{L^2}^2)+\frac 1 2\|\nabla |K|\|_{L^2}^2.
    \]
    Absorbing the last term on the right yields
    \begin{gather*}
    \|K\|_{L^r}\lesssim (1+\|B\|_{H^1})^{\theta}\|K\|_{\Omega^1H^1_B}^\theta \|K\|_{L^2}^{1-\theta}\lesssim (1+\|B\|_{H^1})^{\theta}\|K\|_{\Omega^1H^1_B}, \\
    \|K\|_{\Omega^1H^1}\lesssim (1+\|B\|_{H^1})^{1+\theta}\|K\|_{\Omega^1 H^1_B}.
    \end{gather*}
    Since $\frac{1-\theta}2=\frac 1 2-\frac 1 p$, we have $1-\theta=1-\frac 2 p$ and $\theta=\frac 2 p$. Hence
  \[
    \|K\|_{\Omega^1H^1}\lesssim (1+\|B\|_{H^1})^{1+\frac 2 p}\|K\|_{\Omega^1 H^1_B}.
    \]
    Taking $p$ arbitrarily large proves the claim.

    The reverse inequality is similar. As in Lemma~\ref{lem:Sobolev_comparison}, interpolation using Lemma~\ref{lem:norm_comparison_pos_op} and duality give both bounds for the remaining exponents.
\end{proof}

\section{Resolvent analysis}
We construct the renormalised covariant Laplacian associated with a rough enhanced gauge field and derive resolvent estimates used throughout the paper.

\subsection{Domain of renormalised covariant Laplacian}\label{sec:domain_renorm_laplacian}
For $\kappa\in (0,1]$ we define the space
\[
\cX^{-\kappa}\subset\Omega^1_\Coul C^{-\kappa/2}(\T^2;\bfi\R)\times C^{-\kappa/2}(\T^2;\R),
\]
as the closure of $\Omega^1_\Coul H^1(\T^2;\bfi\R)\times C^{-\kappa/2}(\T^2;\R)$ in the norm
\phantomsection\label{def:enhanced_distance}\[
\triple{(A,A^2)-(\bar A,\bar A^2)}_{-\kappa}:=\|A-\bar A\|_{C^{-\kappa/2}}+\|A^2-\bar A^2\|_{C^{-\kappa/2}}.
\]
Write $\bA$ for elements of $\cX^{-\kappa}$ and set $\0=(0,0)$. For $\bA\in\cX^{-\kappa}$ and $B\in \Omega^1_\Coul H^1$, define
\[
\bA\bplus B:=(A+B,A^2+2A^*B+B^*B).
\]
For such $B\in \Omega^1 H^1$, Besov embedding makes the canonical enhancement $\bfB:=\0\bplus B\in\cX^{-\kappa}$ well-defined. The map $(\bA,B)\mapsto\bA\bplus B$ is continuous: for $\bar\bA\in\cX^{-\kappa}$ and $\bar B\in\Omega^1_\rmC H^1$,
\[
\triple{(\bA\bplus B)-(\bar\bA\bplus \bar B)}_{-\kappa}\lesssim (1+\triple{\bA}_{-\kappa}+\triple{\bar\bA}_{-\kappa}+\|B\|_{H^1}+\|\bar B\|_{H^1})(\triple{\bA-\bar\bA}_{-\kappa}+\|B-\bar B\|_{H^1}).
\]

For $\bA\in\cX^{-\kappa}$, $B\in\Omega^1H^1$ and smooth $u$, the products $A^*\diff_B u$ and $A^2u$ are well-defined distributions. Define
\phantomsection\label{def:rough_perturbation}\begin{align*}
L_{\bA,B}u&:=L_Bu+2A^*\diff_Bu+A^2u, \\
J_{\bA,B}u&:=2A^*\diff_B u+A^2 u,
\end{align*}
so that $L_{\bA,B}=L_B+J_{\bA,B}$.
The following lemma controls the regularity of these terms.

\begin{lemma}\label{lem:bounds_JAB}
Let $\kappa\in(0,1)$, let $s>1+\kappa$, and let
$\bA,\bar\bA\in\cX^{-\kappa}$ and $B\in\Omega^1H^1$.
Then, for every $\theta\in[0,1]$,
\[
\begin{aligned}
\|(J_{\bA,B}-J_{\bar\bA,B})u\|_
{H_B^{-(1-\theta)\kappa-\theta s}}
\lesssim
(1+\|B\|_{H^1})^{\kappa+2(s -1)}
\triple{\bA-\bar\bA}_{-\kappa}
\|u\|_{H_B^{(1-\theta)s+\theta\kappa}}.
\end{aligned}
\]
\end{lemma}

\begin{remark}\label{rem:J_operator_order}
Writing $d=1+3\kappa$, Lemma \ref{lem:bounds_JAB} with $s=1+2\kappa$ implies that $\|J_{\bA,B}-J_{\bar\bA,B}\|_{H_B^p\to H_B^{p-d}}\lesssim
(1+\|B\|_{H^1})^{5\kappa}
\triple{\bA-\bar\bA}_{-\kappa}$
for any $p\in[\kappa,1+2\kappa]$.
This shows we can treat $J_{\bA,B}-J_{\bar\bA,B}$
as an operator of order $d$.
\end{remark}

\begin{proof}
By interpolation (Lemma \ref{lem:interpolation_hilbert_scale}), it suffices to prove the estimate for $\theta=0$ and $\theta=1$.
Since each $J_{\bA,B},J_{
\bar\bA,B}$ is symmetric and $H^s_B$ is canonically dual to $H^{-s}_B$, the case $\theta=1$ follows immediately from the case $\theta=0$ by duality.
By monotonicity in $s$, it suffices to consider $s\in(1+\kappa,2]$.
Using Lemma~\ref{lem:Sobolev_comparison}
and the assumption $s>1+\kappa$, we obtain
\[
\begin{aligned}
\|(A-\bar A)^*\diff_Bu\|_{H_B^{-\kappa}}
&\lesssim
(1+\|B\|_{H^1})^\kappa
\|(A-\bar A)^*\diff_Bu\|_{H^{-\kappa}}
\\
&\lesssim
(1+\|B\|_{H^1})^\kappa
\|A-\bar A\|_{C^{-\kappa/2}}
\|\diff_Bu\|_{H^{s-1}}.
\end{aligned}
\]
Applying Lemma~\ref{lem:Sobolev_comparison} and
Proposition~\ref{prop:derivative_regularity_reduction},
\[
\|\diff_Bu\|_{H^{s-1}}
\lesssim
(1+\|B\|_{H^1})^{s-1}
\|\diff_Bu\|_{H_B^{s-1}}
\lesssim
(1+\|B\|_{H^1})^{2(s-1)}
\|u\|_{H_B^s}.
\]
Hence
\begin{equation}\label{eq:bound_Ad_B}
\|(A-\bar A)^*\diff_Bu\|_{H_B^{-\kappa}}
\lesssim
(1+\|B\|_{H^1})^{\kappa+2(s-1)}
\triple{\bA-\bar\bA}_{-\kappa}
\|u\|_{H_B^s}.
\end{equation}

For the zeroth-order term, choose $\delta\in(0,1)$ sufficiently small.
Then
\[
\begin{aligned}
\|(A^2-\bar A^2)u\|_{H_B^{-\kappa}}
&\lesssim
(1+\|B\|_{H^1})^\kappa
\|(A^2-\bar A^2)u\|_{H^{-\kappa}}
\\
&\lesssim
(1+\|B\|_{H^1})^\kappa
\|A^2-\bar A^2\|_{C^{-\kappa/2}}
\|u\|_{H^{(1+\delta)\kappa}}
\\
&\lesssim
(1+\|B\|_{H^1})^{(2+\delta)\kappa}
\|A^2-\bar A^2\|_{C^{-\kappa/2}}
\|u\|_{H_B^{(1+\delta)\kappa}}
\end{aligned}
\]
where we used Lemma~\ref{lem:Sobolev_comparison} in the last line.
Since $s>1+\kappa$, we may choose $\delta$ so that $(1+\delta)\kappa<s$ and $(2+\delta)\kappa\leq \kappa+2(s-1)$. Therefore
\begin{equation}\label{eq:bound_dA_A2}
\|(A^2-\bar A^2)u\|_{H_B^{-\kappa}}
\lesssim
(1+\|B\|_{H^1})^{\kappa+2(s-1)}
\triple{\bA-\bar\bA}_{-\kappa}
\|u\|_{H_B^s}.
\end{equation}
Combining~\eqref{eq:bound_Ad_B} and~\eqref{eq:bound_dA_A2} gives
the claimed estimate for $\theta=0$ and $s\leq 2$.
\end{proof}

\begin{remark}\label{rem:map_block_A}
    Note that with the same way we proved~\eqref{eq:bound_Ad_B} using Lemma~\ref{lem:Sobolev_comparison_1forms}, we can prove
        \[
        \|A^*K\|_{ H^{-\kappa}_{\bar B}}\lesssim (1+\|\bar B\|_{H^1})^{4\kappa}\|A\|_{C^{-\kappa/2}}\|K\|_{\Omega^1 H^{2\kappa}_B}, \qquad K\in\Omega^1 H^{2\kappa}_B.
        \]
\end{remark}

For $\bA\in\cX^{-\kappa}$ and $B\in\Omega^1H^1$, with $\kappa\in (0,\frac13)$, the distributional operator $L_{\bA,B}$ need not map smooth functions into $L^2$. To realise it as an unbounded operator on $L^2$, set
\[
\scD(L_{\bA,B}):=\{u\in H^{1+2\kappa}_B\colon u+R_B(1)J_{\bA,B}u\in H^2_B\}.
\]
Indeed, for $u\in\scD(L_{\bA,B})$,
\begin{gather*}
R_B(1)(L_Bu+J_{\bA,B}u)=-R_B(1)u+u+R_B(1)J_{\bA,B}u \in H^2_B, \\
L_{\bA,B}u=L_Bu+J_{\bA,B}u=(L_B+1)R_B(1)(L_Bu+J_{\bA,B}u)\in L^2.
\end{gather*}
Thus $(L_{\bA,B},\scD(L_{\bA,B}))$ is an unbounded operator on $L^2$.
For $B\in \Omega^1_\Coul H^1$, Lemma~\ref{lem:self_adjoint} gives $(L_{\bA,B},\scD(L_{\bA,B}))=(L_{\bA\bplus B,0},\scD(L_{\bA\bplus B,0}))$.

\subsection{Resolvent estimates for renormalised covariant Laplacian}\label{sec:resolvent}
For $z\in \rho(L_{\bA,B})$, write $R_{\bA,B}(z):=R(L_{\bA,B},z)$. The main resolvent estimate is as follows.
\begin{theorem}[Resolvent estimate]\label{thm:resolvent_estimate}
Let $\kappa\in (0,\frac 14)$, $\bA\in\cX^{-\kappa} $ and $B\in\Omega^1H^1$.
Denote $d=1+3\kappa\in (1,7/4)$.
There exists a constant $C>0$ depending on $\kappa$, such that
\begin{align}\label{eq:resolvent_set_distinguished}
    \rho_\star(L_{\bA,B}):=\Big\{z\in \bC\colon \Re(z)>0, \, \max\big\{1, C(1+\|B\|_{H^1})^{5\kappa} \triple{\bA}_{-\kappa}\big\}\leq  |z|^{1-\frac{d}{2}}\Big\},
\end{align}
satisfies $\rho_\star(L_{\bA,B})\subset \rho(L_{\bA,B})$. Moreover, for any $z\in\rho_\star(L_{\bA,B})$,
$s\in [-1-2\kappa,2-\kappa]$ and $\eta\in [0,2]\cap[s-1-2\kappa, s+2-\kappa]$, one has the resolvent estimate
\begin{equation}\label{eq:resolvent_estimate}
\|R_{\bA,B}(z)f\|_{H^{s}_B}\lesssim |z|^{\frac \eta2-1}\|f\|_{H^{s-\eta}_B}.
\end{equation}
Finally, given another $\bar\bA\in\cX^{-\kappa}$, we have for any $z\in \rho_\star(L_{\bA,B})\cap \rho_\star(L_{\bar\bA,B})$, the following continuity estimate:
\[
\|(R_{\bA,B}(z)-R_{\bar\bA,B}(z))f\|_{H^{s}_B}\lesssim
|z|^{\frac \eta2-1}
|z|^{\frac{d}{2}-1}
(1+\|B\|_{H^1})^{5\kappa}
\triple{\bA-\bar\bA}_{-\kappa}
\|f\|_{H^{s-\eta}_B},
\]
for any parameters $s$ and $\eta$ satisfying the conditions above.
\end{theorem}

\begin{proof}
Let $z\in\bC$ satisfy $\Re(z)>0$ and $|z|\geq 1$. The factorisation
    \[
    (L_{\bA,B}+z)=(L_B+z)+J_{\bA,B}=(L_B+z)(\id+R_B(z)J_{\bA,B}),
    \]
    reduces invertibility to that of $\id+R_B(z)J_{\bA,B}$, for which we use a Neumann series. For $p\in[\kappa,1+2\kappa]$ and $d=1+3\kappa$, Remark~\ref{rem:J_operator_order} gives
    \begin{equation}\label{eq:resolvent_JAB}
        \begin{aligned}
    \|R_B(z)J_{\bA,B}f\|_{H^{p}_B}&\lesssim (1+\|B\|_{H^1})^{5\kappa}\triple{\bA}_{-\kappa}\|R_B(z)\|_{H^{p-d}_B\to H^{p}_B} \|f\|_{H^{p}_B} \\
    &\leq C (1+\|B\|_{H^1})^{5\kappa}\triple{\bA}_{-\kappa} |z|^{\frac{d}{2} -1} \|f\|_{H^{p}_B}.
    \end{aligned}
    \end{equation}
    where the last bound uses Lemma~\ref{lem:resolvent_estimate_classical}.
    Since $\kappa<1/4$ implies $d<2$, we may choose $|z|$ large enough that
    \begin{align}\label{eq:choice_z}
    C (1+\|B\|_{H^1})^{5\kappa}\triple{\bA}_{-\kappa} |z|^{\frac{d}{2}-1}\leq \frac 1 2.
    \end{align}
    This is the condition defining $\rho_\star(L_{\bA,B})$ in~\eqref{eq:resolvent_set_distinguished}. The operator $(L_B+z)(\id+R_B(z)J_{\bA,B})$ therefore has inverse
    \[
    T_z:=\sum_{k=0}^\infty (-1)^k (R_B(z)J_{\bA,B})^k R_B(z).
    \]
    The map $T_z:L^2\to H^p_B\subset L^2$ is continuous. To identify its range, write
     \begin{align*}
     T_zf+R_B(z)J_{\bA,B}T_zf&=\sum_{k=0}^\infty (-1)^k (R_B(z)J_{\bA,B})^k R_B(z)f+R_B(z)J_{\bA,B}\sum_{k=0}^\infty (-1)^k (R_B(z)J_{\bA,B})^k R_B(z)f\\
     &=\sum_{k=0}^\infty (-1)^k (R_B(z)J_{\bA,B})^k R_B(z)f-\sum_{k=0}^\infty (-1)^{k+1} (R_B(z)J_{\bA,B})^{k+1} R_B(z)f\\
     &=R_B(z)f,
     \end{align*}
     and use the resolvent identity for $R_B(z)$ to obtain
     \begin{align*}
     T_z f+R_B(1)J_{\bA,B}T_z f&=T_z f+R_B(z) J_{\bA,B}T_z-(z-1) R_B(z) R_B(1) J_{\bA,B} T_z f\\
     &=R_B(z)f-(z-1) R_B(z) R_B(1)J_{\bA,B}T_z f.
     \end{align*}
    Taking $p=1+2\kappa$ gives $T_z f\in H^{1+2\kappa}_B$ and $T_z f+R_B(1)J_{\bA,B}T_z f\in H^2_B$, so $T_z f\in\scD(L_{\bA,B})$.

     For injectivity, let $u\in\scD(L_{\bA,B})$ satisfy $(L_{\bA,B}+z)u=0$. Invertibility of $L_B+z$ and $\id+R_B(z)J_{\bA,B}$ gives
     \[
     (L_B+z)(\id+R_B(z)J_{\bA,B})u=0
     \quad\Longrightarrow\quad (\id+R_B(z)J_{\bA,B})u=0
     \quad\Longrightarrow\quad u=0.
     \]
     Thus $z\in\rho(L_{\bA,B})$, so $\rho_\star(L_{\bA,B})\subset\rho(L_{\bA,B})$ and $T_z=R_{\bA,B}(z)$.

     For the resolvent estimate, take $z\in\rho_\star(L_{\bA,B})$ and write
     \begin{align}\label{eq:proof_resolvent_estimate_init_inequality}
     \|R_{\bA,B}(z)f\|_{H^s_B}=\|T_zf\|_{H^s_B}\leq \|R_B(z)f\|_{H^s_B}+ \sum_{k=1}^\infty \|(R_B(z)J_{\bA,B})^k R_B(z)f\|_{H^s_B}.
     \end{align}
     The first term is bounded by Lemma~\ref{lem:resolvent_estimate_classical}, since $\eta\in[0,2]$. For $k\geq1$, \eqref{eq:resolvent_JAB}-\eqref{eq:choice_z} give
     \begin{align*}
      \|(R_B(z)J_{\bA,B})^k &R_B(z)f\|_{H^s_B}\\
      &\leq \|R_B(z) J_{\bA,B}\|_{H^{p}_B \to H^s_B} \|(R_B(z) J_{\bA,B})^{k-1}\|_{H^{p}_B \to H^{p}_B}  \|R_B(z)f\|_{H^{p}_B} \\
      &\leq 2^{-(k-1)}\|R_B(z) J_{\bA,B}\|_{H^{p}_B \to H^s_B}   \|R_B(z)f\|_{H^{p}_B}.
     \end{align*}
    Applying Lemma~\ref{lem:resolvent_estimate_classical}, provided $p-s+\eta \in [0,2]$,
    we obtain
  \begin{equation}\label{eq:s_eta_p}
     \|R_B(z)\|_{H^{s-\eta}_B\to H^{p}_B}\lesssim |z|^{\frac{p-s+\eta}{2}-1}
  \end{equation}
     Moreover, $\|J_{\bA,B}\|_{H^{p}_B \to H^{p-d}_B}\lesssim (1+\|B\|_{H^1})^{5\kappa}\triple{\bA}_{-\kappa}$ by Remark~\ref{rem:J_operator_order}, thus,
provided $s-p+d\in [0,2]$, we have by Lemma ~\ref{lem:resolvent_estimate_classical} again,
     \begin{equation}\label{eq:p_s}
     \|R_B(z) J_{\bA,B}\|_{H^{p}_B \to H^s_B}\lesssim (1+\|B\|_{H^1})^{5\kappa}\triple{\bA}_{-\kappa} |z|^{\frac{s-p+d}{2}-1}.
     \end{equation}
     Therefore
    \begin{align*}
       \|(R_B(z)J_{\bA,B})^k R_B(z)f\|_{H^s_B}
       &\lesssim 2^{-(k-1)}
       |z|^{\frac{p-s+\eta}{2}-1}
       |z|^{\frac{s-p+d}{2}-1}
       (1+\|B\|_{H^1})^{5\kappa}\triple{\bA}_{-\kappa} \|f\|_{H^{s-\eta}_B}
       \\
       &=
       2^{-(k-1)}
       |z|^{\frac{\eta}{2}-1}|z|^{\frac{d}{2}-1}
       (1+\|B\|_{H^1})^{5\kappa}\triple{\bA}_{-\kappa} \|f\|_{H^{s-\eta}_B}.
     \end{align*}
    The conditions $s\in[-1-2\kappa,2-\kappa]$
    and $\eta\in[0,2]\cap[s-1-2\kappa, s+2-\kappa]$
    ensure that we can find $p\in [\kappa,1+2\kappa]$ such that $s-p+d\in [0,2]$ and
    $p-s+\eta\in [0,2]$.

    Using~\eqref{eq:choice_z} to estimate
    $|z|^{\frac{d}{2}-1}(1+\|B\|_{H^1})^{5\kappa}\triple{\bA}_{-\kappa}\lesssim 1$
    leads to
     \[
      \|(R_B(z)J_{\bA,B})^k R_B(z)f\|_{H^s_B}\lesssim 2^{-(k-1)} |z|^{\frac\eta 2-1}\|f\|_{H^{s-\eta}_B}.
     \]
     Substituting into~\eqref{eq:proof_resolvent_estimate_init_inequality} yields
     \[
     \|R_{\bA,B}(z)f\|_{H^s_B}\lesssim |z|^{\frac\eta 2-1}\|f\|_{H^{s-\eta}_B}.
     \]

     For continuity, take $z\in \rho_\star(L_{\bA,B})\cap \rho_\star(L_{\bar\bA,B})$ and write
    \begin{align*}
      R_{\bA,B}(z)-R_{\bar\bA,B}(z)&=((\id+R_B(z)J_{\bA,B})^{-1}-(\id+R_B(z)J_{\bar\bA,B})^{-1})R_B(z) \\
    &=-(\id+R_B(z)J_{\bar\bA,B})^{-1}(R_B(z)J_{\bA,B}-R_B(z)J_{\bar\bA,B})(\id+R_B(z)J_{\bA,B})^{-1} R_B(z)\\
    &=-R_{\bar\bA,B}(z) (J_{\bA,B}-J_{\bar\bA,B})R_{\bA,B}(z).
    \end{align*}
The resolvent estimate just proved gives
    \begin{align*}
    \|(R_{\bA,B}(z)-R_{\bar\bA,B}(z))f\|_{H^s_B}
    &\lesssim (1+\|B\|_{H^1})^{5\kappa}\triple{\bA-\bar\bA}_{-\kappa}
    |z|^{\frac{s-p+d}{2}-1}
    |z|^{\frac{p-s+\eta}{2}-1}  \|f\|_{H^{s-\eta}_B}\\
    &=
    |z|^{\frac\eta 2-1}(1+\|B\|_{H^1})^{5\kappa}\triple{\bA-\bar\bA}_{-\kappa}
    |z|^{\frac{d}{2}-1}  \|f\|_{H^{s-\eta}_B}.\qedhere
    \end{align*}
\end{proof}

\begin{remark}[Resolvent identity]\label{rem:resolvent_identity}
Actually on the way, we have proved the resolvent identity
\[
R_{\bA,B}(z)-R_B(z)=-R_B(z)J_{\bA,B}R_{\bA,B}(z).
\]
\end{remark}

Using the resolvent estimate, we can now actually prove
crucial properties of  $(L_{\bA,B},\scD(L_{\bA,B}))$.
\begin{lemma}\label{lem:self_adjoint}
Let $\kappa\in (0,\frac 1 4)$,  $A\in \cX^{-\kappa}$ and $B\in \Omega^1H^1$. Then the unbounded operator $(L_{\bA,B},\scD(L_{\bA,B}))$ is self-adjoint (and in particular densely defined and closed). Moreover, $B\in\Omega^1_\Coul H^1$, then  $(L_{\bA,B},\scD(L_{\bA,B}))=(L_{\bA\bplus B,0},\scD(L_{\bA\bplus B,0}))$.
\end{lemma}

\begin{proof}
Both claims follow by approximating $\bA\in\cX^{-\kappa}$ by smooth $\bA_k\to\bA$ in $\cX^{-\kappa-\delta}$, with $\delta>0$ sufficiently small. For self-adjointness, note that
\[
\bigcap_{k\geq 0}\rho_\star(L_{\bA_k,B})\cap \rho_\star(L_{\bA,B})\cap\R\neq \varnothing
\]
by~\eqref{eq:resolvent_set_distinguished}. Fix such a $z$. Self-adjointness of $(L_{\bA_k,B},\scD(L_{\bA_k,B}))$ and operator-norm convergence $R_{\bA_k,B}(z)\to R_{\bA,B}(z)$ imply that $R_{\bA,B}(z)$, and hence $(L_{\bA,B},\scD(L_{\bA,B}))$, is self-adjoint.

For $B\in\Omega^1_\Coul H^1$, the identity $(L_{\bA_k,B},\scD(L_{\bA_k,B}))=(L_{\bA_k\bplus B,0},\scD(L_{\bA_k\bplus B,0}))$ gives equal resolvents for each $k$. Their limits are $R_{\bA,B}(z)$ and $R_{\bA\bplus B,0}(z)$, respectively, so $R_{\bA,B}(z)=R_{\bA\bplus B,0}(z)$ and the limiting operators agree.

\end{proof}

\begin{remark}[Spectral theorem]\label{rem:spectral_theorem}
    The resolvent $R_{\bA,B}(z)$ is compact by $H^1\Subset L^2$ and $H^1_B\subset H^1$ (Lemma~\ref{lem:Sobolev_comparison}). Thus the spectrum of $(L_{\bA,B},\scD(L_{\bA,B})$ consists of discrete real eigenvalues with finite-dimensional eigenspaces \cite[Prop.~2.11]{schmudgen2012}. The spectral theorem for self-adjoint unbounded operators gives an orthonormal eigenbasis $(e_i)_{i\geq0}$ and the representation
    \[
    L_{\bA,B}u=\sum_{i\geq 0} \lambda_i\av{u,e_i}e_i,
    \]
     by \cite[Prop.~5.12]{schmudgen2012}, allowing us to use functional calculus.
\end{remark}

For $(L_{\bA,B},\scD(L_{\bA,B}))$, define
\begin{equation}\label{eq:m_def}
m_{\bA,B}:=1+\max\{0,- \min \sigma(L_{\bA,B})\},
\end{equation}
Then $m_{\bA,B}\in\rho(L_{\bA,B})$ and $L_{\bA,B}+m_{\bA,B}$ is strictly positive. The next lemma gives continuity of $\bA\mapsto m_{\bA,B}$. It is also convenient, for $\bA,\bar\bA\in\cX^{-\kappa}$ and $B\in\Omega^1H^1$, to introduce
\begin{align}\label{eq:fC_def}
\fC_\kappa(\bA,\bar\bA;B):=(1+\|B\|_{H^1})^{\frac{20\kappa}{1-3\kappa}} (1+\triple{\bA}_{-\kappa}+\triple{\bar \bA}_{-\kappa})^{\frac{4}{1-3\kappa}},
\end{align}
together with the abbreviation $\fC_\kappa(\bA;B):=\fC_\kappa (\bA,\0;B)$.

\begin{remark}\label{rem:fC_bound}
    Note that
\[
\fC_\kappa(\bA,\bar\bA;B)=(1+\|B\|_{H^1})^{\tau(\kappa)} (1+\triple{\bA}_{-\kappa}+\triple{\bar \bA}_{-\kappa})^{\rmn(\kappa)},
\]
with $\tau(\kappa)\to 0$ as $\kappa\to 0$, and the corresponding equality with $\bar\bA=\0$ holds
\[
\fC_\kappa(\bA;B)=(1+\|B\|_{H^1})^{\tau(\kappa)} (1+\triple{\bA}_{-\kappa})^{\rmn(\kappa)}.
\]
\end{remark}

\begin{lemma}\label{lem:difference_masses}
    Let $\kappa\in (0,\frac 1 4)$, $\bA,\bar\bA\in\cX^{-\kappa}$ and $B\in\Omega^1 H^1$.
    Denote $d=1+3\kappa$. Then
    \[
    |m_{\bA,B}-m_{\bar\bA,B}|
    \lesssim
   \fC_\kappa(\bA,\bar\bA;B)^{1/2}
    \triple{\bA-\bar\bA}_{-\kappa}.
    \]
    Note that $m_{\0,B}=1$, since $L_{\0,B}=L_B\geq0$.
\end{lemma}

\begin{proof}
We set
\begin{align}\label{eq:m_star}
m_\star :=\min \rho_\star(L_{\bA,B})\cap \R
=\max\big\{1, C(1+\|B\|_{H^1})^{\frac{10\kappa}{1-3\kappa}} \triple{\bA}_{-\kappa}^{\frac{2}{1-3\kappa}}\big\}.
\end{align}
where we recall $\rho_\star(L_{\bA,B})$ from Theorem \ref{thm:resolvent_estimate}.
For another $\bar\bA\in\cX^{-\kappa}$,
we define $\bar m_\star$ similarly as $m_\star$, but with $\bar\bA$. Set
\begin{equation}\label{eq:m_max_def}
m:=\max\{m_\star,\bar m_\star\},
\end{equation}
and note that
\begin{equation}\label{eq:m_bounded_by_fC}
m\lesssim
(1+\|B\|_{H^1})^{\frac{10\kappa}{1-3\kappa}}
(1+\triple{\bA}_{-\kappa}+\triple{\bar\bA}_{-\kappa})^{\frac{2}{1-3\kappa}}\lesssim \fC_\kappa(\bA,\bar\bA;B)^{1/2}.
\end{equation}
Since $m\in \rho_\star(L_{\bA,B})\cap \rho_\star(L_{\bar\bA,B})$, Theorem~\ref{thm:resolvent_estimate} gives $m+L_{\bA,B}>0$. The spectral relation and the min-max theorem for self-adjoint compact operators yield
\[
\max\sigma(R_{\bA,B}(m))
=\frac{1}{m+\min\sigma(L_{\bA,B})}
=
\sup_{\|f\|_{L^2}=1}
\av{R_{\bA,B}(m)f,f}.
\]
Denote $J = J_{\bA,B} - J_{\bar\bA,B}$.
Taking $p=\frac{d}{2}$ in Remark~\ref{rem:J_operator_order}, we obtain
\begin{align}\label{eq:J_symmetric_mass}
\|Ju\|_{H^{-d/2}_B}
\lesssim
(1+\|B\|_{H^1})^{5\kappa}
\triple{\bA-\bar\bA}_{-\kappa}
\|u\|_{H^{d/2}_B}.
\end{align}

Using the resolvent identity,
\[
R_{\bA,B}(m)-R_{\bar\bA,B}(m)
=
-R_{\bar\bA,B}(m)
(J_{\bA,B}-J_{\bar\bA,B})
R_{\bA,B}(m),
\]
self-adjointness of the resolvents, and~\eqref{eq:J_symmetric_mass}, we obtain for $\|f\|_{L^2}=1$
\begin{align*}
&\left|
\av{(R_{\bA,B}(m)-R_{\bar\bA,B}(m))f,f}
\right|
\\
&\qquad\lesssim
(1+\|B\|_{H^1})^{5\kappa}
\triple{\bA-\bar\bA}_{-\kappa}
\|R_{\bA,B}(m)f\|_{H^{d/2}_B}
\|R_{\bar\bA,B}(m)f\|_{H^{d/2}_B}.
\end{align*}
Since $m\in\rho_\star(L_{\bA,B})\cap\rho_\star(L_{\bar\bA,B})$, Theorem~\ref{thm:resolvent_estimate} (taking therein $s=\eta=d/2$) implies
\[
\|R_{\bA,B}(m)f\|_{H^{d/2}_B}
+
\|R_{\bar\bA,B}(m)f\|_{H^{d/2}_B}
\lesssim
m^{-1+\frac d4}\|f\|_{L^2}.
\]
Consequently,
\[
\left|
\max\sigma(R_{\bA,B}(m))
-
\max\sigma(R_{\bar\bA,B}(m))
\right|
\lesssim
m^{-2+d/2}
(1+\|B\|_{H^1})^{5\kappa}
\triple{\bA-\bar\bA}_{-\kappa}.
\]

Set $f(x)=1+\max\{0,m-1/x\}$. By definition, $m_{\bA,B}=f(\max\sigma(R_{\bA,B}(m)))$. Since $f$ is $m^2$-Lipschitz, i.e. $|f(x)-f(y)|\leq m^2|x-y|$, we obtain
\begin{align*}
|m_{\bA,B}-m_{\bar\bA,B}|
&\lesssim
m^{d/2}
(1+\|B\|_{H^1})^{5\kappa}
\triple{\bA-\bar\bA}_{-\kappa}.
\qedhere
\end{align*}
By~\eqref{eq:m_bounded_by_fC} together with the fact that $d=1+3\kappa\leq 2$, we get the desired estimate.
\end{proof}

We next strengthen the resolvent estimate of Theorem~\ref{thm:resolvent_estimate}.

\begin{theorem}[Pushed resolvent estimate]
\label{thm:resolvent_estimate_pushed}
Let $\kappa\in(0,\frac14)$, $\bA\in\cX^{-\kappa}$ and
$B\in\Omega^1H^1$. Then, for every $\lambda\ge0$,
$s\in[-1-2\kappa,2-\kappa]$,
$\eta\in[0,2]\cap[s-1-2\kappa,s+2-\kappa]$,
one has
\begin{equation}\label{eq:pushed_resolvent}
\|R_{\bA,B}(\lambda+m_{\bA,B})\|_
{H_B^{s-\eta}\to H_B^s}
\lesssim
(1+\lambda)^{\frac\eta2-1}
\fC_\kappa(\bA;B).
\end{equation}
Moreover, denote $d=1+3\kappa$.
Then for $\bar\bA\in\cX^{-\kappa}$,
one has for every $\lambda\geq 0$ and $s,\eta$ as above
\begin{align}
\|R_{\bA,B}(\lambda+m_{\bA,B})
-R_{\bar\bA,B}(\lambda+m_{\bar\bA,B})\|_
{H_B^{s-\eta}\to H_B^s}
\lesssim
(1+\lambda)^{-2+\frac{\eta+d}{2}}
\fC_\kappa(\bA,\bar\bA;B)^2
\triple{\bA-\bar\bA}_{-\kappa}.
\label{eq:pushed_resolvent_difference}
\end{align}
\end{theorem}

\begin{proof}
Recall $m_\star$ from~\eqref{eq:m_star} and note that
\begin{equation}\label{eq:m_star_local_bound}
m_\star
\lesssim
(1+\|B\|_{H^1})^{\frac{10\kappa}{1-3\kappa}}
(1+\triple{\bA}_{-\kappa})^{\frac{2}{1-3\kappa}}
=
\fC_\kappa(\bA;B)^{1/2}.
\end{equation}
Denote
\[
R_\lambda=R_{\bA,B}(\lambda+m_{\bA,B}),
\qquad
R_{\lambda,\star}=R_{\bA,B}(\lambda+m_\star).
\]
By \eqref{eq:resolvent_set_distinguished},
$\lambda+m_\star\in\rho_\star(L_{\bA,B})$, while the definition of $m_{\bA,B}$ in
\eqref{eq:m_def} implies
\begin{equation}\label{eq:m_def_bound}
\|R_\lambda\|_{L^2\to L^2}\le(1+\lambda)^{-1}.
\end{equation}
Moreover, by Theorem \ref{thm:resolvent_estimate},
$1\leq m_{\bA,B}\le m_\star+1$
and thus
$|m_\star-m_{\bA,B}|\le m_\star$.

Recall the resolvent identity
\begin{equation}\label{eq:fist_resolv}
R_\lambda
=
R_{\lambda,\star}
+(m_\star-m_{\bA,B})R_{\lambda,\star}R_\lambda.
\end{equation}
By \eqref{eq:resolvent_estimate} with $(s,\eta)=(2-\kappa,0)$,
we have $\|R_{\lambda,\star}\|_{H_B^{2-\kappa}\to H_B^{2-\kappa}}\lesssim(\lambda+m_\star)^{-1}$.
On the other hand,
\eqref{eq:m_def_bound} and
\eqref{eq:resolvent_estimate} with
$(s,\eta)=(2-\kappa,2-\kappa)$ imply
\begin{align*}
\|(m_\star-m_{\bA,B})R_{\lambda,\star}R_\lambda\|_{H_B^{2-\kappa}\to H_B^{2-\kappa}}
&\leq
m_\star \|R_{\lambda,\star}\|_{L^2\to H^{2-\kappa}_B}
\|R_\lambda\|_{H_B^{2-\kappa}\to L^2}
\\
&\lesssim
m_\star(\lambda+m_\star)^{-\kappa/2}(1+\lambda)^{-1}
\leq
(1+\lambda)^{-1}m_\star^{(2-\kappa)/2}.
\end{align*}
Consequently $\|R_\lambda\|_{H_B^{2-\kappa}\to H_B^{2-\kappa}}\lesssim(1+\lambda)^{-1}m_\star^{(2-\kappa)/2}$.
By self-adjointness, duality and interpolation
(Lemma~\ref{lem:interpolation_hilbert_scale}),
for any $|q|\le2-\kappa$,
\begin{equation}\label{eq:pushed_same_space_sharp}
\|R_\lambda\|_{H_B^q\to H_B^q}
\lesssim
(1+\lambda)^{-1}m_\star^{|q|/2}.
\end{equation}
Now observe, by \eqref{eq:resolvent_estimate}, $\|R_{\lambda,\star}\|_{H_B^{s-\eta}\to H_B^s} \lesssim (\lambda+m_\star)^{\frac{\eta}{2}-1}$,
and by \eqref{eq:pushed_same_space_sharp}
\[
\|R_{\lambda,\star}R_\lambda\|_{H_B^{s-\eta}\to H_B^{s}}
\leq
\|R_{\lambda,\star}\|_{H_B^{s-\eta}\to H_B^s}
\|R_{\lambda}\|_{H_B^{s-\eta}\to H_B^{s-\eta}}
\lesssim
(\lambda+m_\star)^{\frac{\eta}{2}-1}
(1+\lambda)^{-1}m_\star^{|s-\eta|/2}.
\]
Using again the resolvent identity \eqref{eq:fist_resolv} and $|m_\star-m_{\bA,B}|\le m_\star$, we obtain
\begin{align*}
\|R_\lambda\|_{H_B^{s-\eta}\to H_B^s}
&\lesssim
(\lambda+m_\star)^{\frac{\eta}{2}-1}
\{1+(1+\lambda)^{-1}m_\star^{1+|s-\eta|/2}\}
\\
&\lesssim
(1+\lambda)^{\frac\eta2-1}
+
(1+\lambda)^{-1}
m_\star^{(\eta+|s-\eta|)/2}
\end{align*}
where we used $m_\star(\lambda+m_\star)^{\frac{\eta}{2}-1}\leq m_\star^{\eta/2}$.
The same estimate with $s'=\eta-s$ and $\eta'=\eta$ gives
\[
\|R_\lambda\|_{H_B^{-s}\to H_B^{\eta-s}}
=
\|R_\lambda\|_{H_B^{s'-\eta'}\to H_B^{s'}}
\lesssim
(1+\lambda)^{\frac\eta2-1}
+
(1+\lambda)^{-1}
m_\star^{(\eta+|s'-\eta|)/2}.
\]
The admissible range of $s,\eta$ is invariant under this duality transformation, so $\|R_\lambda\|_{H_B^{s-\eta}\to H_B^s}$ satisfies the same bound. Since $|s'-\eta|=|s|$, taking the better bound gives
\[
\|R_\lambda\|_{H_B^{s-\eta}\to H_B^s}
\lesssim
(1+\lambda)^{\frac\eta2-1}
+
(1+\lambda)^{-1}
m_\star^{\frac{\eta}{2}+\alpha(s,\eta)},
\]
where
$
\alpha(s,\eta)
=
\min\{|s|,|s-\eta|\}/2.
$
Since $\frac{\eta}{2}+\alpha(s,\eta)\leq3/2$ by considering the signs of $s$ and $s-\eta$, together with~\eqref{eq:m_star_local_bound} this then  proves~\eqref{eq:pushed_resolvent}.

For the difference estimate, recall $m$ defined in~\eqref{eq:m_max_def}.
Write $\bar R_\lambda = R_{\bar\bA,B}(\lambda+m_{\bar\bA,B})$.
By the resolvent identity
\begin{equation}\label{eq:resolv_diff}
R_\lambda-\bar R_\lambda
=
-\bar R_\lambda
\bigl(
J_{\bA,B}-J_{\bar\bA,B}
+m_{\bA,B}-m_{\bar\bA,B}
\bigr)
R_\lambda.
\end{equation}
By the analogous estimates to
\eqref{eq:s_eta_p}-\eqref{eq:p_s}
with Lemma \ref{lem:resolvent_estimate_classical}
replaced by the intermediate bound
\[
\|R_{\bA,B}(\lambda+m_{\bA,B})\|_
{H_B^{s-\eta}\to H_B^s}
\lesssim
(1+\lambda)^{\frac\eta2-1}m_\star^{3/2},
\]
and the corresponding bound for $\bar\bA$, we obtain
\begin{align}\label{eq:proof_pushed_resolvent}
\|\bar R_\lambda
(J_{\bA,B}-J_{\bar\bA,B})
R_\lambda\|_{H_B^{s-\eta}\to H_B^s}
\lesssim
(1+\lambda)^{-2+\frac{\eta+d}{2}}
(1+\|B\|_{H^1})^{5\kappa}
m^3
\triple{\bA-\bar\bA}_{-\kappa}.
\end{align}
Using~\eqref{eq:m_bounded_by_fC}, we have
$
(1+\|B\|_{H^1})^{5\kappa}m^3
\lesssim
\fC_\kappa(\bA,\bar\bA;B)^2,
$
and hence the right-hand side of~\eqref{eq:proof_pushed_resolvent} is bounded by the right-hand side of \eqref{eq:pushed_resolvent_difference}.

Finally, Lemma~\ref{lem:difference_masses}, combined with the intermediate pushed resolvent estimate above with
\eqref{eq:pushed_same_space_sharp} shows
\[
\|\bar R_\lambda R_\lambda\|_{H_B^{s-\eta}\to H_B^s}
\lesssim
(1+\lambda)^{-2+\eta/2}
m^{\frac{\eta}{2}+\alpha(s,\eta)}m^{\alpha(s,\eta)}.
\]
Note that $\frac{\eta}{2}+2\alpha(s,\eta)\leq 2$, and therefore, by~\eqref{eq:m_bounded_by_fC},
\[
\|\bar R_\lambda R_\lambda\|_{H_B^{s-\eta}\to H_B^s}
\lesssim
(1+\lambda)^{-2+\eta/2}
\fC_\kappa(\bA,\bar\bA;B).
\]
Thus the term with the mass difference in \eqref{eq:resolv_diff} is bounded by
\[
(1+\lambda)^{-2+\eta/2}
\fC_\kappa(\bA,\bar\bA;B)^{3/2}
\triple{\bA-\bar\bA}_{-\kappa},
\]
which is also bounded by a multiple of the right-hand side of \eqref{eq:pushed_resolvent_difference}.
\end{proof}

The following consequence will be used in the variational problem.

\begin{corollary}\label{cor:fractional_power_comparison_LAB}
Let $\kappa\in (0,\frac 14)$, $\bA\in\cX^{-\kappa} $ and $B\in\Omega^1H^1$.
Then for all $\theta\in[0,1]$ and  $u\in\scD((L_{\bA,B}+m_{\bA,B})^\theta)$, one has
\[
\|u\|_{H_B^{\theta(2-\kappa)}}\lesssim \fC_\kappa(\bA;B)^\theta \|(L_{\bA,B}+m_{\bA,B})^\theta u\|_{L^2}.
\]
\end{corollary}

\begin{proof}
The resolvent estimate at $s=2-\kappa$ gives the first bound for every $f\in L^2$, and taking $f=(L_{\bA,B}+m_{\bA,B})u$ afterwards, gives the second:
\begin{gather*}
\|R_{\bA,B}(m_{\bA,B})f\|_{H_B^{2-\kappa}}\lesssim \fC_\kappa(\bA;B)\|f\|_{L^2}, \\
\|u\|_{H_B^{2-\kappa}}\lesssim \fC_\kappa(\bA;B) \|(L_{\bA,B}+m_{\bA,B})u\|_{L^2}.
\end{gather*}
Interpolation by Lemma~\ref{lem:norm_comparison_pos_op} with the positive self-adjoint operators $(L_B+1)$ and $(L_{\bA,B}+m_{\bA,B})$ gives, for every $\theta\in[0,1]$,
\[
\|u\|_{H_B^{\theta(2-\kappa)}}=\|(L_B+1)^{\theta(2-\kappa)/2}u\|_{L^2}
\lesssim \fC_\kappa(\bA;B)^\theta\|(L_{\bA,B}+m_{\bA,B})^\theta u\|_{L^2}.
\]
\end{proof}

\subsection{Power counting for traces}\label{sec:power_counting}

We develop a power counting argument to bound traces of compositions
of resolvents, derivatives and multiplication operators. This yields,
for example, the remainder bound~\eqref{eq:rem_bound} for the ratio of
determinants in Section~\ref{sec:determinants}. Sufficient regularity
gain makes these compositions trace class, with trace norms controlled
by the corresponding operator norms.

For the rest of this section, fix $\kappa\in (0,\frac 14)$, $\lambda\geq 0$, $\bA,\bar\bA\in\cX^{-\kappa}$ and $B\in\Omega^1_\Coul H^1$. Define
\begin{gather*}
\Xi_{-1}:=\{\diff_B\},\qquad
\Xi_0:= C^{-\kappa/2} \cup \Omega^1 C^{-\kappa/2}, \\
\Xi_2
:= \{ R_B(1+\lambda), R_{\bA,B}(\lambda+m_{\bA,B}), R_{\bA,B}(\lambda+m_{\bA,B})-R_{\bar\bA,B}(\lambda+m_{\bar\bA,B})\}.
\end{gather*}
We treat elements in $\Xi_0$ either $f\in C^{-\kappa/2}$ or $K\in\Omega^1 C^{-\kappa/2}$ as the operators
\[
f:\Omega^0 H^{2\kappa}_B\to \Omega^0 H^{-\kappa}_B,\qquad
K^*:\Omega^1 H^{2\kappa}_B\to \Omega^0 H^{-\kappa}_B,
\]
and therefore $\Xi:=\Xi_{-1}\cup \Xi_0\cup \Xi_2$ is a set of operators.

\begin{remark}
Note that $i\in\{-1,0,2\}$ in $\Xi_i$ can be viewed as the amount the operator $X\in\Xi_i$ regularises.  For example, a derivative regularises by $-1$ and a resolvent by $2$.
\end{remark}

Furthermore,  on $\Xi_{-1}$, we define
\[
\fav{\diff_B}:=(1+\lambda)^{1/2} (1+\|B\|_{H^1})^{2\kappa},
\]
and on $\Xi_2$
\begin{align*}
\fav{R_B(1+\lambda)}
&:=(1+\lambda)^{-1+2\kappa},\\
\fav{R_{\bA,B}(m_{\bA,B}+\lambda)}
&:=(1+\lambda)^{-1+2\kappa}
\fC_\kappa(\bA;B),\\
\fav{R_{\bA,B}(m_{\bA,B}+\lambda)-R_{\bar\bA,B}(m_{\bar\bA,B}+\lambda)}
&:=(1+\lambda)^{-1+2\kappa}
\fC_\kappa(\bA,\bar\bA;B)^2
\triple{\bA-\bar\bA}_{-\kappa}.
\end{align*}
Finally on $\Xi_0$, for $f\in C^{-\kappa/2}$ or $K\in \Omega^1C^{-\kappa/2}$, we set
\[
\fav{f}:=(1+\|B\|_{H^1})^{4\kappa} \|f\|_{C^{-\kappa/2}},\qquad \fav{K}:=(1+\|B\|_{H^1})^{4\kappa}\|K\|_{\Omega^1 C^{-\kappa/2}}.
\]
\begin{remark}\label{rem:fC_bound_power_counting}
    By Remark~\ref{rem:fC_bound}, we have that $\fav{R_{\bA,B}(m_{\bA,B}+\lambda)}$ are bounded by $(1+\|B\|_{H^1})^{\tau(\kappa)}(1+\triple{\bA}_{-\kappa})^{\rmn(\kappa)}$ and similar statement for the difference of resolvents.
\end{remark}

The choice for the latter makes sense as we have
\begin{lemma}\label{lem:maps_0_order}
    For any $f\in C^{-\kappa/2}$ or $K\in\Omega^1 C^{-\kappa/2}$, we have
    \[
   \|f\|_{\Omega^0 H^{2\kappa}_B\to \Omega^0 H^{-\kappa}_B }\lesssim \fav{f}, \qquad \|K^*\|_{\Omega^1 H^{2\kappa}_B\to \Omega^0 H^{-\kappa}_B} \lesssim \fav{K}.
    \]
\end{lemma}

\begin{proof}
One should just note that the this is Young multiplication theorem where we move between the Sobolev spaces $H^s$ and $H^s_B$ (or $\Omega^1H^s$ and $\Omega^1H^s_B$) which yields the corresponding powers in $(1+\|B\|_{H^1})$ (note that for the $1$-form Sobolev comparison Lemma~\ref{lem:Sobolev_comparison_1forms} we lose an additional $\kappa$).
\end{proof}

We also recall the resolvent estimate

\begin{lemma}[Power counting]\label{lem:power_counting}
Let $k\geq2$ be an integer and $(X_j)_{j=1}^k\subset\Xi$ have well-defined composition $\prod_{j=1}^k X_j$. Assume
\[
X_1,X_k\in\Xi_2,\qquad
X_j=\diff_B\implies X_{j+1}\in\Xi_2,\qquad
X_j\in\Xi_0\implies X_{j-1},X_{j+1}\not\in\Xi_0.
\]
Then $\prod_{j=1}^k X_j$ is trace class and
\[
\Tr\Big|\prod_{j=1}^k X_j\Big|\lesssim (1+\lambda)^{1-2\kappa}\prod_{j=1}^k \fav{X_j}.
\]
\end{lemma}

\begin{proof}
If cyclicity were true, then
    \begin{align*}
\Tr\Big(\prod_{j=1}^k X_j\Big)
&=\Tr\Big((L_B+1)^{-1-\kappa/2} (L_B+1)^{1-\kappa/2}  \Big(\prod_{j=1}^k X_j\Big) (L_B+1)^{\kappa} \Big),
\end{align*}
which means that the composition of operators is trace class if
\[
\Big\|\prod_{j=1}^k X_j\Big\|_{H^{-2\kappa}_B\to H^{2-\kappa}_B}<\infty,
\]
in which case we get
\begin{align*}
\Tr\Big|\prod_{j=1}^k X_j\Big|&\leq \Tr((L_B+1)^{-1-\kappa/2}) \Big\|\prod_{j=1}^k X_j\Big\|_{H^{-2\kappa}_B\to H^{2-\kappa}_B}.
\end{align*}
A resolvent $R\in\Xi_2$ allows any neighbouring operators. By the hypotheses, each derivative is followed by a resolvent and, to return from $1$-forms to $0$-forms, preceded by $K^*$ with $K\in\Omega^1C^{-\kappa/2}$. Thus derivatives and $1$-forms occur in blocks $K^*\diff_B R$, while $f\in C^{-\kappa/2}$ occurs as $fR$. The available blocks are
\[
R, \qquad fR, \qquad K^*\diff_B R , \qquad f,K\in\Xi_0, R\in\Xi_2.
\]
Write $\prod_{j=1}^k X_j=\prod_{j=1}^\ell Y_j$, where the $Y_j$ are these blocks and $Y_1=X_1\in\Xi_2$.

For $R\in\Xi_2$, $f\in C^{-\kappa/2}$ and $K\in\Omega^1C^{-\kappa/2}$, the first two bounds follow from Lemma~\ref{lem:resolvent_estimate_classical}, Theorem~\ref{thm:resolvent_estimate_pushed} and Young's multiplication theorem. For the third, we also use Proposition~\ref{prop:derivative_regularity_reduction}:
\begin{gather*}
\|R\|_{H^{-2\kappa}_B\to H^{-\kappa}_B}\leq \|R\|_{H^{-2\kappa}_B\to H^{2\kappa}_B}\lesssim \fav{R},  \qquad \|fR\|_{H^{-2\kappa}_B\to H^{-\kappa}_B}\lesssim \fav{f} \fav{R}, \\
\|K^*\diff_BR\|_{H^{-2\kappa}_B\to H^{-\kappa}_B}\lesssim \fav{K}\fav{\diff_B}\fav{R}.
\end{gather*}
In particular
\[
\Big\| \prod_{j=1}^\ell Y_j \Big\|_{H^{-2\kappa}_B\to H^{2-\kappa}_B}\lesssim \|X_1\|_{H^{-\kappa}_B\to H^{2-\kappa}_B}\prod_{j=2}^\ell \|Y_j\|_{H^{-2\kappa}_B\to H^{-\kappa}_B}\lesssim \|X_1\|_{H^{-\kappa}_B\to H^{2-\kappa}_B}\prod_{j=2}^k \fav{X_j},
\]
by the block decomposition. Lemma~\ref{lem:resolvent_estimate_classical} and Theorem~\ref{thm:resolvent_estimate_pushed} give
\[
\|X_1\|_{H^{-\kappa}_B\to H^{2-\kappa}_B}\lesssim (1+\lambda)^{1-2\kappa}\fav{X_1},
\]
which concludes the proof.
\end{proof}

\section{Ratio of determinants}\label{sec:determinants}

We analyse the ratio of determinants $Q_{N,T}(A_T+B_T)$
in the variational approach, where $A_T$ is the Gaussian part
and $B_T=J_T(v)$ the drift defined in Section~\ref{sec:gauge_variational_setup}.
We write
\begin{align}\label{eq:determinant_computation}
\begin{split}
Q_{N,T}(A_T&+B_T)
=-\frac{\lambda_{A_T+B_T,T}-1}{2}\rmp_N+\rmo_{N,T}\\
& +\int_0^N\! \Tr \Big((L_{B_T}\!+\!1\!+\!\lambda)^{-1}\!-\!(L_{A_T+B_T}\!-\!\rmc_T\!+\!\lambda_{A_T+B_T,T}\!+\!\lambda)^{-1} \Big)\diff \lambda \\
& +\int_0^N \Tr \Big((L_0\!+\!1\!+\!\lambda)^{-1}\!-\!(L_{B_T}\!+\!1\!+\!\lambda)^{-1} \Big)\dif \lambda.
\end{split}
\end{align}
For the variational upper bound, take $B_T\equiv0$, so the
last integral vanishes. For the lower bound, it is nonnegative
by the diamagnetic inequality. Thus only the first integral
requires an estimate.

For the pathwise analysis, let $\kappa\in (0,\frac14)$, $\bA\in\cX^{-\kappa}$ and $B\in\Omega^1_\Coul H^1$, and define
\[
 \scq_{\bA,B}(\lambda):=R_B(\lambda+1)- R_{\bA,B}(\lambda+m_{\bA,B}).
\]
Note that for $\bA_T=(A_T,A_T^*A_T-\rmc_T)$ and $B_T$, we have
\[
\scq_{\bA_T,B_T}(\lambda)=(L_{B_T}+1+\lambda)^{-1}-(L_{A_T+B_T}-\rmc_T+\lambda_{A_T+B_T,T}+\lambda)^{-1},
\]
which is indeed the integrand (before taking the trace) of the integral we need to prove suitable bounds for.

To see that $\scq_{\bfA,B}(\lambda)$ is trace class, we apply the resolvent identity
\begin{align*}
\scq_{\bA,B}(\lambda)&=R_B(1+\lambda)( J_{\bA,B}+m_{\bA,B}-1) R_{\bA,B}(m_{\bA,B}+\lambda),
\end{align*}
The two resolvents and the single derivative in $J_{\bA,B}$ make $\scq_{\bA,B}(\lambda)$ trace class.
Since the suggested decay $(1+\lambda)^{-\frac12+\delta}$ is not integrable on $\R_+$, we still need an argument to pass to the limit $N\to\infty$.
However, each further application of the resolvent identity to $R_{\bA,B}$ gains half a power of decay at large $\lambda$.
After two further applications, we obtain an integrable remainder along with explicit terms requiring separate analysis.

With $J^\rmm_{\bA,B}:=J_{\bA,B}+m_{\bA,B}-1$, two further applications of the resolvent identity give
\begin{align*}
\scq_{\bA,B}(\lambda)
&=R_B(1+\lambda)J^\rmm_{\bA,B} R_B(1+\lambda)\\
&\quad - 4R_B(1+\lambda)A^*\diff_BR_B(1+\lambda)A^*\diff_B R_B(1+\lambda)\\
&\quad +4R_B(1+\lambda)A^*\diff_BR_B(1+\lambda)A^*\diff_B R_B(1+\lambda)J^\rmm_{\bA,B}R_{\bA,B}(m_{\bA,B}+\lambda)\\
&\quad - 2R_B(1+\lambda)(A^2+m_{\bA,B}-1) R_B(1+\lambda)A^*\diff_B R_{\bA,B}(m_{\bA,B}+\lambda)\\
&\quad - R_B(1+\lambda)J^\rmm_{\bA,B} R_B(1+\lambda)(A^2+m_{\bA,B}-1) R_{\bA,B}(m_{\bA,B}+\lambda).
\end{align*}
The last three terms turn out the easiest to treat and therefore we call them the remainder \begin{align*}
\Rem_{\bA,B}(\lambda)&:=4R_B(1+\lambda)A^*\diff_BR_B(1+\lambda)A^*\diff_B R_B(1+\lambda)J^\rmm_{\bA,B}R_{\bA,B}(m_{\bA,B}+\lambda)\\
&\qquad - 2R_B(1+\lambda)(A^2+m_{\bA,B}-1) R_B(1+\lambda)A^*\diff_B R_{\bA,B}(m_{\bA,B}+\lambda)\\
&\qquad - R_B(1+\lambda)J^\rmm_{\bA,B} R_B(1+\lambda)(A^2+m_{\bA}-1) R_{\bA,B}(m_{\bA,B}+\lambda).
\end{align*}
We write
\begin{align*}
&R_B(1+\lambda) J_{\bA,B}^\rmm R_B(1+\lambda)\\=&2R_B(1+\lambda) A^*\diff_B R_B(1+\lambda)+R_B(1+\lambda) A^2 R_B(1+\lambda)+(m_{\bA,B}-1)R_B(1+\lambda)^2.
\end{align*}

We then define
\begin{align}
\scq^1_{A,B}(\lambda)&:=2R_B(1+\lambda) A^*\diff_B R_B(1+\lambda) \label{eq:def_q1}  \\
\nonumber \scq^2_{\bA,B}(\lambda)&:=R_B(1+\lambda) A^2 R_B(1+\lambda)- 4R_B(1+\lambda)A^*\diff_BR_B(1+\lambda)A^*\diff_B R_B(1+\lambda)\\
\nonumber\scq^3_{\bA,B}(\lambda)&:=(m_{\bA,B}-1)R_B(1+\lambda)^2,
\end{align}
so that\footnote{We also define $\scq^1_{\bA,B}:=\scq^1_{A,B}$. }
\begin{align}\label{eq:decomp_scq}
\scq_{\bA,B}(\lambda)=\sum_{i=1}^3\scq^i_{\bA,B}(\lambda)+\Rem_{\bA,B}(\lambda).
\end{align}

Reflecting back to our original computation~\eqref{eq:determinant_computation}, we are interested in the existence of the limit
\[
\lim_{N\to\infty}\int_0^N \Tr(\scq_{\bA,B}(\lambda))\dif\lambda.
\]
The remainder has enough resolvents for its trace to be integrable; see Lemma~\ref{lem:Rem_det}.
The traces of $\scq^i_{\bA,B}$ need more care.
For $i=2$, the assumption $\bA\in\cX^{-\kappa}$ is not enough to pass to the limit $N\to\infty$ in the trace integral.
This is the only term that requires a further enhancement of the gauge field.

\subsection{Enhanced state space for ratio of determinants}\label{sec:enhancement_determinants}
The divergence in the integral of $\Tr\scq^2_{\bA,B}(\lambda)$ over $\R_+$ comes from $\Tr\scq^2_{\bA,0}(\lambda)$, as suggested by Taylor expansion around $B=0$.
To see what enhancement is needed, consider the Gaussian case $\bA_T=(A_T,A_T^2)=(A_T,A_T^*A_T-\rmc_T)$, for which
\begin{align*}
\Tr \scq^2_{\bA_T,0}(\lambda)
&=\Tr\big( R_0(1+\lambda) (A^2_T-A_T^*A_T) R_0(1+\lambda)\big)\\
&\qquad \qquad +\int_{\mathbb{T}^2\times \mathbb{T}^2} A^*_T(x)\Pi^\lambda(x-y)A_T(y)\,\diff x\,\diff y,
\end{align*}
where the equality is understood as defining the bilinear form
\phantomsection\label{def:vacuum_polarisation_bilinear}\[
\boldsymbol{\Pi}^\lambda(f,g) \coloneqq \int_{\mathbb{T}^2} f^*(x)\Pi^\lambda(x-y)g(y)\,\diff x\,\diff y\, \textup{ for }f,g\in\Omega^1 C^\infty.
\]
It can be readily verified that the operator $\Pi^\lambda$, as a Fourier multiplier, is characterised by the symbol
\phantomsection\label{def:vacuum_polarisation_multiplier}\[
\Pi^{\lambda}_{j,k}(p) = \sum_{q\in 2\pi\Z^2}(1+\lambda+|q|^2)^{-2}\Big(\delta_{j,k}-\frac{(2q+p)_j(2q+p)_k}{1+\lambda+|q+p|^2}\Big),\quad j,k=1,2.
\]
It is easy to see from this expression that
\[
\Pi^{\lambda}_{j,k}(p) \approx \delta_{j,k}\sum_{q\in 2\pi\Z^2}(1+\lambda+|q|^2)^{-2} \, \textup{ as } \, |p_1|,|p_2|\to\infty,
\]
and hence that $\Pi^\lambda$ is not regularising, even at fixed $\lambda$.
For fixed $T$, $A_T$ is smooth and $\lambda\mapsto\boldsymbol{\Pi}^\lambda(A_T,A_T)$ is well-defined and integrable.
However, we cannot expect this integrability to be uniform as $T\to\infty$, because of the high frequency behaviour of $\Pi^\lambda$.

The divergent part lies in the zeroeth Wiener chaos and hence is deterministic. We remove it with the counterterm $\rmo_{N,T}$ defined by
\[
\rmo_{N,T} = -\E^\Theta\E^\rmA\Big[\int_0^N\; \Tr \scq^2_{(A_T,A^*_T A_T-\rmc_T),0}(\lambda)\,\diff \lambda \Big].
\]
Now letting
\begin{align}\label{eq:def_vacuum_A_T}
\bV_{A_T}(\lambda) = \boldsymbol{\Pi}^\lambda(A_T,A_T)-\E^\Theta\E^\rmA[\boldsymbol{\Pi}^\lambda(A_T,A_T)],
\end{align}
in combination with the fact that $A_T^2-A_T^*A_T=-\rmc_T$, we immediately have the identity
\[
\int_0^N\; \Tr \scq^2_{(A_T,A^*_T A_T-\rmc_T),0}(\lambda)\,\diff \lambda + \rmo_{N,T} = \int_0^N\; \bV_{A_T}(\lambda)\,\diff \lambda.
\]
We show later that $\bV_{A_T}\in L^1(\R_+;\R)$ almost surely, with estimates that allow us to remove the cut-off in $T$.
For the pathwise analysis, we therefore add an $L^1(\R_+;\R)$ component to the enhanced data in place of $\bV_{A_T}$.

For $\kappa\in(0,\frac14)$, define the enlarged state space and write its elements as
\phantomsection\label{def:full_enhancement}\[
\cX^{-\kappa}_\rmV:=\cX^{-\kappa}\times L^1(\R_+;\R), \qquad \bfA=(\bA,\bV)=(A,A^2,\bV).
\]
The first two components form an element of $\cX^{-\kappa}$, so we define the metric by
\phantomsection\label{def:full_enhancement_distance}\[
\bfnorm{\bfA-\bar\bfA}_{-\kappa}:=\triple{\bA-\bar\bA}_{-\kappa}+\|\bV-\bar\bV\|_{L^1}, \qquad \bfA,\bar\bfA\in\cX^{-\kappa}_\rmV.
\]

We define from these quantities a rough variant of $Q_{N,T}(A_T+B_T)$ as follows:
\begin{align}\label{eq:def_cQ_N}
\begin{split}
\cQ_N(\bfA,B)&:=-(m_{\bA,B}-1)\int_0^N \Tr( R_0(\lambda+1)^2)\dif \lambda +   \int_0^N \big( \bV(\lambda)-\Tr\scq^2_{\bA,0}(\lambda) \big)\dif\lambda  \\
&\qquad+  \int_0^N  \Tr(\scq_{\bA,B}(\lambda))\dif\lambda + \int_0^N \Tr(R_0(1+\lambda)-R_B(1+\lambda))\dif\lambda.
\end{split}
\end{align}
For $\bfA_T:=(A_T,A_T^*A_T-\rmc_T,\bV_{A_T})$, the definition gives $\cQ_N(\bfA_T,B_T)=Q_{N,T}(A_T+B_T)$.

We need compatibility with addition of a regular drift after removing the cut-offs, particularly for the variational formula in the large-deviation analysis.
For $\bfA\in\cX^{-\kappa}_\rmV$ and $B\in\Omega^1_\Coul H^1$, define the translation by
\begin{align}\label{eq:translation_bfA}
\bfA\bplus B:=(A+B,A^2+2A^*B+B^*B, \bV+2\bfPi(A,B)+\bfPi(B,B)),
\end{align}
\phantomsection\label{def:vacuum_polarisation_family}where we define $\bfPi(A,B):=(\lambda\mapsto \bfPi^\lambda(A,B))$.
The following proposition shows that this translation is well-defined.
\begin{proposition}\label{prop:Pi_quadratic_form}
  For $\kappa\in [0,1/4)$ the following bound and its $L^1$ consequence hold:
  \[
| \bfPi^\lambda(A,B)| \lesssim (1+\lambda)^{-\frac 3 2 +\kappa}\norm{B}_{H^1} \|A\|_{C^{-\kappa/2}}, \qquad
\| \bfPi(A,B)\|_{L^1}\lesssim  \norm{B}_{H^1} \|A\|_{C^{-\kappa/2}}.
  \]
\end{proposition}
The proof uses the following properties of $\Pi^\lambda$ and its Fourier symbol. These computations are a continuum analogue of the treatment of vacuum polarisation in \cite[App.~A]{BFS80}, with operator bounds corresponding to Lemma~A.1 there. We include a detailed proof for clarity and because our regularisation in $\lambda$ differs.

\begin{lemma} \label{lem:vac_pol_bound}
    The operator $\Pi^\lambda$ is transverse, meaning that
    \[
    \sum_{k=1}^2 \Pi^\lambda_{j,k}(p)p_k = 0 \quad\textup{ for all }\, j\in\{1,2\},
    \]
    and satisfies the spectral bound
    \[
    |\Pi^\lambda_{j,k}(p)| \lesssim (1+\lambda)^{-2} + \begin{cases}
        |p|^2(1+\lambda)^{-2} &\textup{ if }\quad |p|\le \sqrt{1+\lambda} \\
        (1+\lambda)^{-1} &\textup{ if }\quad |p|> \sqrt{1+\lambda}
    \end{cases}
    \]
    for all $j,k=1,2$ and $p\in 2\pi\Z^2$.
\end{lemma}
\begin{proof}
   We first verify transversality. With the help of the identity
   \[
   (2q+p)\cdot p = (1+\lambda+|p+q|^2)-(1+\lambda+|q|^2),
   \]
   we find that
   \[
   \sum_{k=1}^2 \Pi^\lambda_{j,k}(p)p_k = \sum_{q\in 2\pi\Z^2} \frac{p_j}{(1+\lambda+|q|^2)^2}-\frac{(2q+p)_j}{(1+\lambda+|q|^2)^2}+\frac{(2q+p)_j}{(1+\lambda+|q|^2)(1+\lambda+|p+q|^2)}.
   \]
   The first summand cancels part of the second; the remaining sums vanish by antisymmetry under $q\mapsto -q$ and $q\mapsto -q-p$, respectively. For $p\ne0$, $\Pi^\lambda(p)$ acts as a scalar multiple of the rank one projection onto the transverse component. It therefore has the form
   \begin{align*}
   \Pi^\lambda_{j,k}(p) &= \Tr\Pi^\lambda(p) \Big(\delta_{j,k}-\frac{p_j p_k}{|p|^2}\Big) \,\textup{ for }\, p\neq 0,\\
   \Pi^\lambda_{j,k}(0) &= \delta_{j,k}\sum_{q\in 2\pi\Z^2}\frac{1+\lambda-|q|^2}{(1+\lambda+|q|^2)^3},
   \end{align*}
   where the trace is over the internal indices.
   We estimate the zero mode first, since we will use this bound again below. The sum converges absolutely for each $\lambda$. We show that cancellation between the terms on either side of $|q|^2=1+\lambda$ gives an integrable bound in $\lambda$. To see this, define $h_\lambda$ and observe that
   \[
   h_\lambda(x) = \frac{1+\lambda-|x|^2}{(1+\lambda+|x|^2)^3}, \qquad \int_{\R^2}\; h_\lambda(x)\,\diff x = 0.
   \]
   Since $h_\lambda$ is smooth and decays as $O(|x|^{-4})$ on $\R^2$, Poisson summation gives the identity
   \[
   \sum_{q\in 2\pi\Z^2}\frac{1+\lambda-|q|^2}{(1+\lambda+|q|^2)^3} = \frac{1}{4\pi^2}\sum_{\substack{\xi\in\Z^2 \\ \xi\neq 0}} \hat h_\lambda(\xi) = \frac{1}{4\pi^2(1+\lambda)}\sum_{\substack{\xi\in\Z^2 \\ \xi\neq 0}} \hat h_0\big(\sqrt{1+\lambda}\,\xi\big),
   \]
   where the last equality follows by rescaling the Fourier transform. Since $\hat h_0(\xi)$ decays as $O\big((1+|\xi|)^{-n}\big)$ for every $n$, we obtain $|\Pi^\lambda_{j,k}(0)| \lesssim (1+\lambda)^{-2}$.
   We address $\Tr\Pi^\lambda(p)$ for $|p|^2>1+\lambda$ next. Split the lattice sum according to $D_1 \coloneqq \{|q|\le |p|/2\}$, $D_2 \coloneqq \{|q+p|\le |p|/2\}$ and $D_3 \coloneqq 2\pi\Z^2\setminus(D_1\cup D_2)$. On $D_1$, it holds that $1+\lambda + |p+q|^2 \ge |p|^2/2$, which implies the inequality
   \begin{align*}
   \sum_{q\in D_1} \frac{2(1+\lambda -|q|^2) + |p|^2}{(1+\lambda+|q|^2)^2(1+\lambda+|q+p|^2)} &\lesssim \sum_{q\in 2\pi\Z^2} \frac{1}{(1+\lambda +|q|^2)^2} + \frac{1}{|p|^2}\sum_{q\in D_1} \frac{1}{1+\lambda +|q|^2} \\
   &\lesssim \frac{1}{1+\lambda} + \frac{1}{1+\lambda}\frac{\log\big(1+\frac{|p|^2}{1+\lambda}\big)}{\frac{|p|^2}{1+\lambda}} \lesssim \frac{1}{1+\lambda}.
   \end{align*}
   The estimate for $D_2$ runs along similar lines after the change of lattice variables $q\mapsto p+q$, and for the complement piece $D_3$ we have the bound
   \begin{multline*}
   \sum_{q\in D_3} \frac{2(1+\lambda -|q|^2) + |p|^2}{(1+\lambda+|q|^2)^2(1+\lambda+|q+p|^2)} \lesssim \sum_{q\in 2\pi\Z^2\setminus D_1} \frac{1}{(1+\lambda+|q|^2)^2} \\
   + \sum_{q\in 2\pi\Z^2\setminus D_2} \frac{1}{(1+\lambda+|p+q|^2)^2} \lesssim \sum_{|r|\ge |p|/2} \frac{1}{(1+\lambda+|r|^2)^2} \lesssim \frac{1}{1+\lambda + |p|^2} \sim \frac{1}{1+\lambda}.
   \end{multline*}
   For $|p|^2\le 1+\lambda$, the zero-mode estimate reduces the proof to bounding $\Tr\big(\Pi^\lambda(p)-\Pi^\lambda(0)\big)$. With $a\coloneqq 1+\lambda+|q|^2$, $b\coloneqq 1+\lambda+|p+q|^2$ and $c\coloneqq 1+\lambda-|q|^2$, we have
   \begin{align*}
   \Tr\big(\Pi^\lambda(p)-\Pi^\lambda(0)\big) &= \sum_{q\in 2\pi\Z^2} \frac{2(1+\lambda -|q|^2) + |p|^2}{(1+\lambda+|q|^2)^2(1+\lambda+|q+p|^2)} -  \frac{2(1+\lambda -|q|^2)}{(1+\lambda+|q|^2)^3} \\
   &= \sum_{q\in 2\pi\Z^2} \frac{2(2p\cdot q+|p|^2)^2}{a^4b} + \frac{|p|^2}{a^2b} - \frac{2c(2p\cdot q+|p|^2)}{a^4}.
   \end{align*}
   Part of the third term above is odd in $q$ and thus vanishes, and the remaining terms can immediately be simplified by the inequality $|c|\le a$. Using $|p|^2 \le 1+\lambda$, the estimates
   \[
   \frac{2(2p\cdot q+|p|^2)^2}{a^3 b} + \frac{|p|^2}{a^2b} + \frac{2|p|^2}{a^3} \le \frac{16|p|^2|q|^2 + 4|p|^4}{a^3b} + \frac{3|p|^2}{(1+\lambda) a^2} \lesssim \frac{|p|^2}{(1+\lambda) a^2}
   \]
   are straightforward and yield the desired $|p|^2/(1+\lambda)^2$ bound after summing in $q$.
\end{proof}

\begin{proof}[Proof of Proposition~\ref{prop:Pi_quadratic_form}]
 Using Lemma~\ref{lem:vac_pol_bound} and repeated applications of the Cauchy-Schwarz inequality yields
    \begin{align*}
        |\boldsymbol{\Pi}^\lambda (A,B)| &\le \sum_{j,k=1}^2 \Big[ \Big((1+\lambda)^{-2}\Big(\sum_{p\in 2\pi\Z^2} + \sum_{|p|^2\le 1+\lambda} |p|^2\Big) + (1+\lambda)^{-1}\sum_{|p|^2>1+\lambda} \Big] |\hat A_j(p)||\hat B_k(p)| \\
        &\le \sum_{j,k=1}^2(1+\lambda)^{-2}\Big(\sum_{p\in 2\pi\Z^2}\langle p\rangle^{-4\kappa} |\hat A_j(p)|^2\sum_{p\in 2\pi\Z^2}\langle p\rangle^{4\kappa} |\hat B_k(p)|^2\Big)^{1/2} \\
        &\quad\quad\quad\quad + (1+\lambda)^{-2}\Big(\sum_{|p|^2\le 1+\lambda}\langle p\rangle^{-2+6\kappa} \sum_{p\in 2\pi\Z^2}\langle p\rangle^{2-6\kappa} |\hat A_j(p)|^2|\hat B_k(p)|^2\Big)^{1/2} \\
        &\quad\quad\quad\quad + (1+\lambda)^{-3/2+\kappa} \sum_{|p|^2> 1+\lambda}\langle p\rangle^{1-2\kappa} |\hat A_j(p)||\hat B_k(p)|.
        \end{align*}
        By standard embeddings on $\mathbb{T}^2$, we thus obtain the bound
        \begin{align*}
        |\boldsymbol{\Pi}^\lambda (A,B)| &\lesssim (1+\lambda)^{-2}\norm{A}_{\Omega^1 H^{-2\kappa}} \norm{B}_{\Omega^1 H^{2\kappa}} + (1+\lambda)^{-2+3\kappa}\norm{A}_{\Omega^1 C^{-\kappa/2}} \norm{B}_{\Omega^1 H^{1-\kappa}} \\
        &\quad\quad\quad\quad + (1+\lambda)^{-3/2+\kappa}\norm{A}_{\Omega^1 H^{-2\kappa}} \norm{B}_{\Omega^1 H^1} \\
        &\lesssim (1+\lambda)^{-3/2+\kappa}\norm{A}_{\Omega^1 C^{-\kappa/2}} \norm{B}_{\Omega^1 H^1},
     \end{align*}
     for all $\kappa<1/4$, which establishes the claim.
\end{proof}

As expected we also have
\begin{proposition}\label{prop:translation_determinant}
   Let $\kappa\in(0,\frac 14)$, $\bfA\in\cX^{-\kappa}_\rmV$ and $B\in\Omega^1_\Coul H^1$. It holds that
    \[
    \cQ_N(\bfA\bplus B,0)=\cQ_N(\bfA,B).
    \]
\end{proposition}

\begin{proof}
By definition,
\begin{align*}
\begin{split}
\cQ_N(\bfA\bplus B,0)&:=-(m_{\bA\bplus B,0}-1)\int_0^N \Tr( R_0(\lambda+1)^2)\dif \lambda \\
&\qquad +   \int_0^N \big( \bV(\lambda)+\bfPi^\lambda(2A+B,B)-\Tr\scq^2_{\bA\bplus B,0}(\lambda) \big)\dif\lambda  \\
&\qquad+  \int_0^N  \Tr(\scq_{\bA\bplus B,0}(\lambda))\dif\lambda.
\end{split}
\end{align*}
By Lemma~\ref{lem:self_adjoint}, $m_{\bA\bplus B,0}=m_{\bA,B}$. Also,
\[
\scq_{\bA\bplus B,0}(\lambda)=\scq_{\bA,B}(\lambda)+R_0(\lambda+1)-R_B(\lambda+1).
\]
Writing the trace in Fourier variables gives
\[
\bfPi^\lambda(2A+B,B)-\Tr\scq^2_{\bA\bplus B,0}(\lambda)=-\Tr\scq^2_{\bA,0}(\lambda).
\]
Substituting these identities into~\eqref{eq:def_cQ_N} proves the claim.
\end{proof}

Now that we have set up the new state space, we can prove the two main theorems in this section. The first proves convergence of $\cQ_N(\bfA,0)$ as $N\to\infty $ and continuity of the limit with respect to $\bfA$, thereby establishing the existence of a ``rough'' ratio of determinants. The second gives a lower bound for $\cQ_N(\bfA,B)$ for use in the variational problem.

\subsection{Convergence of ratio of determinants}
The following theorem shows the convergence of the approximate ratio of determinants:
\begin{theorem} \label{thm:convergence_cQ_N}
    Let $\kappa\in (0,\frac 1{12})$ and $\bfA\in\cX^{-\kappa}_\rmV$. The limit of $\cQ_N(\bfA,0)$ as $N\to\infty$ exists and is denoted by $\cQ_\infty(\bA,0)$. Furthermore, one has the identity
    \[
\cQ_N(\bfA,0)= \int_0^N \bV(\lambda)\dif\lambda +  \int_0^N \Tr(\Rem_{\bA,0}(\lambda))\dif\lambda, \qquad N\in [0,\infty].
    \]
   Finally, for any $N\in [0,\infty]$, one has
   \[
   |\cQ_N(\bfA,0)-\cQ_N(\bar\bfA,0)|\lesssim (1+\triple{\bA}_{-\kappa}+\triple{\bar\bA}_{-\kappa})^{\rmn(\kappa)} \bfnorm{\bfA-\bar\bfA}_{-\kappa}.
   \]
\end{theorem}

\begin{proof}
    By the definition of $\cQ_N$ in~\eqref{eq:def_cQ_N} and the decomposition of $\scq_{\bA,0}$ from~\eqref{eq:def_cQ_N}, we get
    \begin{align*}
\cQ_N(\bfA,0)&:=-(m_{\bA,0}-1)\int_0^N \Tr( R_0(\lambda+1)^2)\dif \lambda +   \int_0^N \big( \bV(\lambda)-\Tr\scq^2_{\bA,0}(\lambda) \big)\dif\lambda  \\
&\qquad + \sum_{i=1}^3\int_0^N  \Tr(\scq_{\bA,0}^i(\lambda))\dif\lambda+ \int_0^N \Tr(\Rem_{\bA,0}(\lambda))\dif\lambda,
\end{align*}
where we have used that each term in the decomposition of $\scq_{\bA,0}$ is trace class on its own and integrable when $N<\infty$. We know that $\scq^2$ cancels with the corresponding term next to $\bV$ and the $\scq^3$ term cancels with the mass term in the beginning of the expression. Finally,
\begin{align}\label{eq:tr_q1_vanishes}
\Tr(\scq^1_{\bA,0}(\lambda))=2\Tr(A^*\diff R_0(1+\lambda)^2)=0,
\end{align}
 due to the fact that $R_0(1+\lambda)^2$  is translation invariant and even, so its derivative vanishes on the diagonal.   Hence for $N<\infty$
\[
\cQ_N(\bfA,0)= \int_0^N \bV(\lambda)\dif\lambda +  \int_0^N \Tr(\Rem_{\bA,0}(\lambda))\dif\lambda.
\]
Now Lemma~\ref{lem:Rem_det} below (with $B=0$) implies the integrability required for the remainder term to take $N\to\infty$ and the continuity estimate.
\end{proof}

\begin{lemma}\label{lem:Rem_det}
Let $\kappa\in(0,\frac 14)$, $\bA\in \cX^{-\kappa}$ and $B\in \Omega^1 H^1$ in the Coulomb gauge. Then we have
\begin{align}\label{eq:rem_bound}
|\Tr (\Rem_{\bA,B}(\lambda))| \lesssim (1+\lambda)^{-\frac 32+6\kappa } (1+\|B\|_{H^1})^{\tau(\kappa)} (1+\triple{\bA}_{-\kappa})^{\rmn(\kappa)}.
\end{align}
Moreover, given another $\bar\bA\in\cX^{-\kappa}$, we have
\[
|\Tr (\Rem_{\bA,B}(\lambda)-\Rem_{\bar\bA,B}(\lambda))| \lesssim (1+\lambda)^{-\frac 32+6\kappa }(1+\|B\|_{H^1})^{\tau(\kappa)} (1+\triple{\bA}_{-\kappa}+\triple{\bar\bA}_{-\kappa})^{\rmn(\kappa)} \triple{\bA-\bar\bA}_{-\kappa}.
\]
\end{lemma}

\begin{proof}
Write $R_B:=R_B(1+\lambda)$ and $R_{\bA,B}:=R_{\bA,B}(m_{\bA,B}+\lambda)$. Define
\begin{gather*}
C_{\bA}:=1+\triple{\bA}_{-\kappa}, \qquad C_B:=1+\|B\|_{H^1}, \qquad \mu:=1+\lambda,\\
J^{m;0}_{\bA,B}:=(A^2+m_{\bA,B}-1).
\end{gather*}
Here $J^{m;0}_{\bA,B}$ is the zero order part of $J^m_{\bA,B}$.
We write  $\Rem_{\bA,B}:=\Rem_{\bA,B}(\lambda)$ and we split it as follows:
\[
\Rem_{\bA,B}(\lambda)=8\Rem^1+4\Rem^2-2\Rem^3-2\Rem^4-\Rem^5,
\]
where
\begin{alignat*}{2}
\Rem^1&:= R_BA^*\diff_B R_B A^*\diff_B R_B A^*\diff_B R_{\bA,B},
&\qquad \Rem^2&:= R_BA^*\diff_B R_B A^*\diff_B R_B J^{m;0}_{\bA,B}R_{\bA,B},\\
\Rem^3&:=R_BJ_{\bA,B}^{m;0}R_B A^*\diff_B R_{\bA,B},
&\qquad \Rem^4&:=R_B A^*\diff_B R_B J_{\bA,B}^{m;0} R_{\bA,B},\\
\Rem^5&:=R_B J^{m;0}_{\bA,B} R_B J_{\bA,B}^{m;0} R_{\bA,B}.
\end{alignat*}

Using the power counting setting from Section~\ref{sec:power_counting}, we apply Lemma~\ref{lem:power_counting} to bound the traces of $\Rem^i$. Its assumptions are easily verified. Define $\tilde\Xi_0:=\{J^{m;0}_{\bA,B}\}\cup\{A\}$. Lemma~\ref{lem:difference_masses} gives
\begin{align}\label{eq:regularity_J_0_order}
\|J^{m;0}_{\bA,B}\|_{C^{-\kappa/2}}\lesssim (1+\|B\|_{H^1})^{\tau(\kappa)} (1+\triple{\bA}_{-\kappa})^{\rmn(\kappa)}.
\end{align}
To treat $\Rem^1$, we note that we have four resolvents and three derivatives, and as such we get by  Lemmas~\ref{lem:maps_0_order} and~\ref{lem:power_counting}
\begin{align*}
\Tr|\Rem^1|&\lesssim (1+\lambda)^{1-2\kappa}(1+\lambda)^{4\cdot(-1+2\kappa)} (1+\lambda)^{3\cdot\frac 1 2} (1+\|B\|_{H^1})^{\tau(\kappa)} (1+\triple{\bA}_{-\kappa})^{\rmn(\kappa)}\\
&\lesssim (1+\lambda)^{-\frac 3 2+6\kappa} (1+\|B\|_{H^1})^{\tau(\kappa)} (1+\triple{\bA}_{-\kappa})^{\rmn(\kappa)}.
\end{align*}
Regarding $\Rem^2$, we count four resolvents and two derivatives and such we get together with~\eqref{eq:regularity_J_0_order}
\begin{align*}
\Tr|\Rem^2|&\lesssim (1+\lambda)^{1-2\kappa}(1+\lambda)^{4\cdot(-1+2\kappa)} (1+\lambda)^{2\cdot\frac 1 2} (1+\|B\|_{H^1})^{\tau(\kappa)} (1+\triple{\bA}_{-\kappa})^{\rmn(\kappa)}\\
&\lesssim (1+\lambda)^{-2+6\kappa} (1+\|B\|_{H^1})^{\tau(\kappa)} (1+\triple{\bA}_{-\kappa})^{\rmn(\kappa)}.
\end{align*}
For the remaining terms, the same power counting argument gives
\begin{align*}
\Tr|\Rem^3|&\lesssim   (1+\lambda)^{-\frac 32+4\kappa} (1+\|B\|_{H^1})^{\tau(\kappa)} (1+\triple{\bA}_{-\kappa})^{\rmn(\kappa)}\\
\Tr|\Rem^4|&\lesssim   (1+\lambda)^{-\frac 3 2+4\kappa} (1+\|B\|_{H^1})^{\tau(\kappa)} (1+\triple{\bA}_{-\kappa})^{\rmn(\kappa)}\\
\Tr|\Rem^5|&\lesssim   (1+\lambda)^{-2+4\kappa} (1+\|B\|_{H^1})^{\tau(\kappa)} (1+\triple{\bA}_{-\kappa})^{\rmn(\kappa)}.
\end{align*}
Combining these bounds yields~\eqref{eq:rem_bound}.

For the continuity estimate, we repeat the bounds, inserting the differences $A-\bar A$, $J^{m;0}_{\bA,B}-J^{m;0}_{\bar \bA,B}$ and $R_{\bA,B}-R_{\bar\bA,B}$ in each term. The same bounds then hold with the additional factor $\triple{\bA-\bar\bA}_{-\kappa}$.
\end{proof}

\subsection{Lower bound for ratio of determinants}
The main result of this subsection gives the lower bound for the determinant ratio:
\begin{theorem}[Lower bound determinant ratio]\label{thm:lower_bound_det}
    Let $\kappa\in (0,\frac 1{12})$, $\bA\in\cX^{-\kappa}_\rmV$ and $B\in\Omega^1 H^1$ in the Coulomb gauge. We have the following lower bound uniformly in $N\in [0,N]$:
    \[
\cQ_N(\bfA,B) \geq -C(1+\bfnorm{\bfA}_{-\kappa})^{\rmn(\kappa)}-\delta \|B\|_{H^1}^{1+\tau(\kappa)}.
    \]

\end{theorem}

The proof requires some machinery to bound traces efficiently and is postponed to page~\pageref{proof:lower_bound_det}. Inspired by \cite[Sec~3.4]{BC24_YM}, we interpolate $B$ smoothly from $0$ to $B$ by setting $B_\theta:=\theta B$ for $\theta\in [0,1]$. This compares terms involving $B$ with those for $B=0$, up to a Taylor remainder. When $B=0$, the explicit flat resolvent $R_0(1+\lambda)$ usually allows direct computation. Each derivative comes from a resolvent, so the Taylor remainder has an additional resolvent, making it easier to treat.

The differentiability of the resolvent is established in the following lemma:

\begin{lemma}\label{lem:differentiating_resolvents}
For any $\lambda>0$, the map $\theta\mapsto R_{B_\theta}(\lambda)$ is twice differentiable as a map $(0,1)\to \rmL(H^{-\kappa},H^{2-\kappa})$ for any $\kappa\in [0,\frac14)$, and we have
\begin{align*}
\partial_\theta R_{B_\theta}(\lambda)&=-2R_{B_\theta}(\lambda)B^*\diff_{B_\theta}R_{B_\theta}(\lambda)\\
\partial_\theta^2 R_{B_\theta}(\lambda)&=8R_{B_\theta}(\lambda) B^* \diff_{B_\theta}R_{B_\theta}(\lambda)B^*\diff_{B_\theta}R_{B_\theta}(\lambda)-2R_{B_\theta}(\lambda)B^*BR_{B_\theta}(\lambda).
\end{align*}

\end{lemma}

\begin{proof}
It is clear that $R_{B_\theta}(\lambda)$ maps $L^2$ to $H^2$. Let $\theta\in (0,1)$ and $h\in\R$ be close enough to zero, and write
\[
L_{B_{\theta+h}}-L_{B_\theta}=2h B^*\diff_{B_\theta}+h^2B^*B.
\]
Note that uniformly in $\theta\in [0,1]$
\begin{gather*}
\|B^*\diff_{B_\theta}u\|_{L^2}\lesssim \|B\|_{L^4}\|\diff_{B_\theta}u\|_{L^4}\lesssim \|B\|_{H^1}(\|Bu\|_{L^4}+\| u\|_{W^{1,4}})\lesssim (1+\|B\|_{H^1})^2 \|u\|_{H^{2-\kappa}},\\
\|B^*B u\|_{L^2}\leq \|B\|_{L^4}^2\|u\|_{L^\infty}\lesssim \|B\|_{H^1}^2\|u\|_{H^{2-\kappa}},\\
\text{so that}\qquad \|L_{B_{\theta+h}}-L_{B_\theta}\|_{H^2\to L^2}\lesssim |h| (1+\|B\|_{H^1})^2.
\end{gather*}
Now we apply the resolvent identity
    \begin{align*}
    R_{B_{\theta+h}}(\lambda)-R_{B_\theta}(\lambda) &=-R_{B_{\theta+h}}(\lambda) (L_{B_{\theta+h}}-L_{B_\theta})R_{B_\theta}(\lambda)\\ &= -R_{B_{\theta+h}}(\lambda) (2h B^*\diff_{B_\theta}+h^2B^*B)R_{B_\theta}(\lambda),
        \end{align*}

    from where we see continuity of $\theta\mapsto R_{B_\theta}(\lambda)$ and that
    \[
    \lim_{h\to 0}\Big\|\frac{R_{B_{\theta+h}}(\lambda)-R_{B_\theta}(\lambda)}{h}+2R_{B_\theta}(\lambda)B^*\diff_{B_\theta}R_{B_\theta}(\lambda)\Big\|_{L^2\to H^2}=0,
    \]

    yielding differentiability and the identity of the first derivative.

    The second derivative follows from the product rule and differentiability of $\theta\mapsto \diff_{B_\theta}$ as a map $H^2\to H^1$.
\end{proof}

The next ingredient is a Taylor remainder theorem for traces, which follows from finite rank approximations:
\begin{lemma}[Differentiating under the trace]\label{lem:differentiation_trace}
Assume an operator $[0,1]\ni\theta\mapsto X_\theta$ is differentiable as a map $(0,1)\to \rmL(L^2)$ with derivative $\dot X_\theta$. If $X_\theta$ and $\dot X_\theta$ are both trace class with trace class norms bounded uniformly in $\theta$, then
\[
\Tr(X_\theta)-\Tr(X_0)=\int_0^\theta \Tr(\dot X_\rho)\dif\rho\]
Furthermore, if $X_\theta$ is twice differentiable with trace class derivatives, we get
\[
\Tr(X_\theta)-\Tr(X_0)-\Tr(\dot X_0)=\int_0^\theta \int_0^\rho \Tr(\ddot X_\eta) \dif \eta\dif \rho=\int_0^\theta (\theta-\rho) \Tr(\ddot X_\rho)\dif \rho.
\]
\end{lemma}

Now we have gathered all materials needed to prove Theorem~\ref{thm:lower_bound_det}:
\phantomsection
\label{proof:lower_bound_det}
\begin{proof}[Proof of~Theorem~\ref{thm:lower_bound_det}]
    We start by writing
    \begin{align*}
   \scq_{\bA,B}&=\scq_{A,B}^1+\scq_{\bA,B}^2 +\scq_{\bA,B}^3+\Rem_{\bA,B}\\
   &=\scq^2_{\bA,0}+\tilde\scq^3_{\bA,B}+\scq^1_{A,B}+ (\scq^2_{\bA,B}-\scq^2_{\bA,0})+(\scq^3_{\bA,B}-\tilde\scq^3_{\bA,B})+\Rem_{\bA,B},
    \end{align*}
    where $\tilde\scq_{\bA,B}^3(\lambda)=(m_{\bA,B}-1)R_0(1+\lambda)^2$. We obtain after cancelling the terms $\scq_{\bA,B}^2$ and $\tilde\scq_{\bA,B}^3$ (the mass term) that
   \begin{align}\label{eq:Q_N_rewritten}
   \begin{split}
\cQ_N(\bfA,B)
&=  \int_0^N  \bV(\lambda)\dif\lambda
 +  \int_0^N  \Tr\scq^1_{A,B}(\lambda) \dif\lambda \\
&\qquad + \int_0^N
\Tr(\scq^2_{\bA,B}(\lambda)-\scq^2_{\bA,0}(\lambda))\dif\lambda \\
&\qquad + \int_0^N
\Tr (\scq^3_{\bA,B}(\lambda)-\tilde\scq^3_{\bA,B}(\lambda))\dif\lambda
+ \int_0^N \Tr(\Rem_{\bA,B}(\lambda))\dif\lambda\\
&\qquad + \int_0^N
\Tr(R_0(1+\lambda)-R_B(1+\lambda))\dif\lambda.
\end{split}
\end{align}
The first integral is bounded by $\|\bV\|_{L^1}$. By Lemma~\ref{lem:Rem_det}, the remainder is integrable in $\lambda$ and satisfies
\[
\Big| \int_0^N \Tr (\Rem_{\bA,B}(\lambda))\dif\lambda\Big|\lesssim (1+\|B\|_{H^1})^{\tau(\kappa)} (1+\triple{\bA}_{-\kappa})^{\rmn(\kappa)}.
\]
Finally, by diamagnetic inequality,
\[
\int_0^N \Tr(R_0(1+\lambda)-R_B(1+\lambda))\dif\lambda\geq 0.
\]
It remains to estimate the three terms involving $\scq^i_{\bA,B}$ using power counting Lemma~\ref{lem:power_counting}.

\paragraph{Preparations for power counting.}
We introduce the following abbreviations (for $\theta\in [0,1]$):
\begin{align*}
   B_\theta=\theta B,\qquad \diff_\theta:=\diff_{B_\theta},  \qquad R_\theta:=R_{B_\theta}(1+\lambda), \qquad H_\theta^s:=H_{B_\theta}^s.
\end{align*}
We also define $\tilde \Xi_0:= \{A,B,B^*B,A^*B,A^2\}$.
All these terms belong to $C^{-\kappa/2}$ or $\Omega^1 C^{-\kappa/2}$ by Sobolev embedding and Young's multiplication theorem.\footnote{Note that we are losing a lot of regularity for $B$ and $B^*B$ in doing so, but our argument does not need the better regularities.}

\paragraph{Treatment $\scq^1_{A,B}$.} We apply the power counting Lemma~\ref{lem:power_counting} directly to $\scq^1_{\bfA,B}(\lambda)=2R_B(1+\lambda) A^*\diff_B R_B(1+\lambda)$. We count two resolvents and one derivative, giving
\begin{align}\label{eq:good_bound}
\nonumber |\Tr(\scq_{A,B}^1(\lambda))|&\lesssim (1+\lambda)^{1-2\kappa} (1+\lambda)^{2(-1+2\kappa)} (1+\lambda)^{\frac 1 2} (1+\|B\|_{H^1})^{2\kappa} (1+\|B\|)^{4\kappa} \|A\|_{C^{-\kappa/2}}\\
&\lesssim (1+\lambda)^{-\frac 1 2+2\kappa}(1+\|B\|_{H^1})^{6\kappa}\|A\|_{C^{-\kappa/2}}.
\end{align}
This bound does not have enough decay in $\lambda$ for integrability. We next prove a bound with enough decay in $\lambda$ but super-quadratic in $B$, then interpolate between the two bounds.

To that end, define $X_\theta (\lambda):=\scq_{\bfA,B_\theta}(\lambda)=2R_\theta A^*\diff_\theta R_\theta$. The map $\theta\mapsto X_\theta (\lambda)$ is twice differentiable with continuous derivatives by Lemma~\ref{lem:differentiating_resolvents} and differentiability of $\theta\mapsto \diff_\theta$. Lemma~\ref{lem:differentiation_trace} then gives
\begin{align}\label{eq:differentiating_q1}
\begin{split}
\Tr(\scq_{A,B}^1(\lambda))&=\Tr(\scq^1_{A,0}(\lambda))+\Tr(\dot X_0(\lambda)) + \int_0^1 (1-\theta)\Tr(\ddot X_\theta(\lambda))\dif\theta\\
&=\Tr(\dot X_0(\lambda)) + \int_0^1 (1-\theta)\Tr(\ddot X_\theta(\lambda))\dif\theta,
\end{split}
\end{align}
where we have used $\Tr(\scq^1_{\bA,0}(\lambda))=0$ from~\eqref{eq:tr_q1_vanishes}.
We first compute
 \begin{equation} \label{eq:Xdot_definition}
    \dot X_\theta(\lambda)=-4R_\theta B^*\diff_\theta R_\theta A^*\diff_\theta R_\theta+2R_\theta A^*BR_\theta-4R_\theta A^*\diff_\theta R_\theta B^*\diff_\theta R_\theta.
    \end{equation}
    At $\theta=0$, taking the trace and using Proposition~\ref{prop:Pi_quadratic_form} gives
    \begin{align}
    \Tr(\dot X_0(\lambda))&=2\bfPi^\lambda(A,B),\nonumber\\
    |\Tr(\dot X_0(\lambda))|&\lesssim (1+\lambda)^{-\frac 3 2 +\kappa}\|B\|_{H^1}\|A\|_{C^{-\kappa/2}}.\label{eq:Xdot_bound}
    \end{align}

    Differentiating~\eqref{eq:Xdot_definition} once more gives
    \begin{align*}
  \ddot X_\theta(\lambda)&=8 R_\theta B^*\diff_\theta B^*\diff_\theta R_\theta A^*\diff_\theta R_\theta-4R_\theta B^*B R_\theta A^*\diff_\theta R_\theta+8R_\theta B^*\diff_\theta R_\theta B^*\diff_\theta R_\theta A^*\diff_\theta R_\theta \\
    &\quad -4R_\theta B^*\diff_\theta R_\theta A^*BR_\theta +8R_\theta B^*\diff_\theta R_\theta A^*\diff_\theta R_\theta B^*\diff_\theta R_\theta \\
    &\quad -4R_\theta B^*\diff_\theta R_\theta A^*BR_\theta-4R_\theta A^*BR_\theta B^*\diff_\theta R_\theta \\
    &\quad +8R_\theta B^*\diff_\theta R_\theta A^*\diff_\theta R_\theta B^*\diff_\theta R_\theta -4R_\theta A^*BR_\theta B^*\diff_\theta R_\theta +8 R_\theta A^*\diff_\theta R_\theta B^*\diff_\theta R_\theta B^*\diff_\theta R_\theta \\
    &\quad -4R_\theta A^*\diff_\theta R_\theta B^*BR_\theta +8 R_\theta A^*\diff_\theta R_\theta B^*\diff_\theta R_\theta B^*\diff_\theta R_\theta.
    \end{align*}
    This simplifies to
    \begin{align*}
    \ddot X_\theta(\lambda)    &=
16R_\theta B^*\diff_\theta R_\theta B^*\diff_\theta R_\theta A^*\diff_\theta R_\theta
-4R_\theta B^*BR_\theta A^*\diff_\theta R_\theta
-8R_\theta B^*\diff_\theta R_\theta A^*BR_\theta
\\
&\quad
+16R_\theta B^*\diff_\theta R_\theta A^*\diff_\theta R_\theta B^*\diff_\theta R_\theta
-8R_\theta A^*BR_\theta B^*\diff_\theta R_\theta
+16R_\theta A^*\diff_\theta R_\theta B^*\diff_\theta R_\theta B^*\diff_\theta R_\theta
\\
&\quad
-4R_\theta A^*\diff_\theta R_\theta B^*BR_\theta.
    \end{align*}
The three terms with coefficient $16$ have three derivatives and $4$ resolvents, so power counting Lemma~\ref{lem:power_counting} gives the trace bound
\[
(1+\lambda)^{1-2\kappa} (1+\lambda)^{4(-1+2\kappa)}(1+\lambda)^{3\cdot\frac 12} (1+\|B\|_{H^1})^{3\cdot 2\kappa} (1+\|B\|_{H^1})^{3\cdot 4\kappa} \|A\|_{C^{-\kappa/2}}\|B\|_{H^1}^2.
\]
All other terms have $1$ derivative  and $3$ resolvents,  so the trace is bounded by
\[
(1+\lambda)^{1-2\kappa} (1+\lambda)^{3(-1+2\kappa)} (1+\lambda)^{\frac 1 2} (1+\|B\|_{H^1})^{2\kappa} (1+\|B\|_{H^1})^{2\cdot 4\kappa} \|A\|_{C^{-\kappa/2}}\|B\|_{H^1}^2.
\]
Hence we combine both estimates to get
\begin{align*}
\int_0^1 (1-\theta)|\Tr(\ddot X_\theta(\lambda))|\dif \theta
\lesssim (1+\lambda)^{-\frac32+6\kappa}\|A\|_{C^{-\kappa/2}}(1+\|B\|_{H^1})^{18\kappa}\|B\|_{H^1}^2.
\end{align*}
In particular, this bound together with~\eqref{eq:differentiating_q1} and~\eqref{eq:Xdot_bound} yields
\begin{align}\label{eq:bad_bound}
|\Tr(\scq^1_{A,B}(\lambda))|\lesssim (1+\lambda)^{-\frac32+6\kappa} (1+\|B\|_{H^1})^{2+18\kappa} \|A\|_{C^{-\kappa/2}}.
\end{align}

Finally, we interpolate~\eqref{eq:good_bound} and~\eqref{eq:bad_bound} to obtain \begin{align}\label{eq:Tr_q1_bound}
|\Tr(\scq^1_{A,B}(\lambda))|\lesssim (1+\lambda)^{-1-2\kappa+24\kappa^2} (1+\|B\|_{H^1})^{1+24\kappa+72\kappa^2} \|A\|_{C^{-\kappa/2}}.
\end{align}
For $\kappa<\frac 1 {12}$, the power of $\lambda$ is less than $-1$, so this bound is integrable in $\lambda$.  \paragraph{Treatment of $\scq^2_{\bA,B}$.} With the same abbreviations as above, we write using~Lemma~\ref{lem:differentiation_trace}
\begin{align}\label{eq:differentiating_q2}
\nonumber\Tr(\scq^2_{\bA,B}(\lambda))-\Tr(\scq^2_{\bA,0})&=\int_0^1 \Tr(\partial_\theta \scq^2_{\bA,\theta B}(\lambda))\dif \theta\\
&=\int_0^1 \Tr\big(\partial_\theta(R_\theta A^2 R_\theta)- 4\partial_\theta(R_\theta A^*\diff\theta BR_\theta A^*\diff_\theta R_\theta)\big)\dif\theta.
\end{align}
 The first term is equal to
    \[
   \partial_{\theta}(R_{\theta}A^2 R_\theta)=-2R_\theta B^*\diff_\theta R_\theta A^2 R_\theta-2R_\theta A^2R_\theta B^*\diff_\theta R_\theta,
    \]
where each term on the right has $1$ derivative and $3$ resolvents, so Lemma~\ref{lem:power_counting} gives the trace bound
\[
(1+\lambda)^{1-2\kappa}(1+\lambda)^{3(-1+2\kappa)}(1+\lambda)^{\frac 1 2}(1+\|B\|_{H^1})^{2\kappa} (1+\|B\|_{H^1})^{2\cdot 4\kappa} \|B\|\|A^2\|_{C^{-\kappa/2}}.
\]

Regarding the second term in~\eqref{eq:differentiating_q2}, we find
\begin{align*}
    \partial_\theta (R_\theta A^*\diff_\theta R_\theta A^*\diff_\theta R_\theta) &=-2R_\theta B^*\diff_\theta R_\theta A^*\diff_\theta R_\theta A^*\diff_\theta R_\theta +R_\theta A^* BR_\theta A^*\diff_\theta R_\theta \\
    &\qquad -2R_\theta A^*\diff_\theta R_\theta B^*\diff_\theta R_\theta A^*\diff_\theta R_\theta +R_\theta A^*\diff_\theta R_\theta A^* BR_\theta\\
    &\qquad -2R_\theta A^*\diff_\theta R_\theta A^*\diff_\theta R_\theta B^*\diff_\theta R_\theta.
    \end{align*}
   The three terms with coefficient $-2$ have $3$ derivatives and $4$ resolvents, so their trace is bounded by
   \[
   (1+\lambda)^{1-2\kappa}(1+\lambda)^{4(-1+2\kappa)}(1+\lambda)^{3\cdot\frac 1 2}(1+\|B\|_{H^1})^{3\cdot 2\kappa} (1+\|B\|_{H^1})^{3\cdot 4\kappa} \|B\|\|A\|_{C^{-\kappa/2}}^2.
   \]
   The remaining terms have $1$ derivative and $3$ resolvents, so their trace is bounded by
    \[
   (1+\lambda)^{1-2\kappa}(1+\lambda)^{3(-1+2\kappa)}(1+\lambda)^{\frac 1 2}(1+\|B\|_{H^1})^{2\kappa} (1+\|B\|_{H^1})^{2\cdot 4\kappa} \|B\|\|A\|_{C^{-\kappa/2}}^2.
   \]
Combining all these estimates then yields
\begin{align}\label{eq:Tr_q2_bound}
|\Tr(\scq^2_{\bA,B}(\lambda))-\Tr(\scq^2_{\bA,0})|\lesssim (1+\lambda)^{-\frac 3 2 +6\kappa} (1+\|B\|_{H^1})^{1+18\kappa}(1+\triple{\bA}_{-\kappa})^2.
\end{align}
This is integrable in $\lambda$ for $\kappa<\frac 1{12}$.

\paragraph{Treatment of $\scq^3_{\bA,B}$.}

We have, with the second identity following from Lemma~\ref{lem:differentiation_trace},
\begin{gather*}
\scq^3_{\bA,B}(\lambda)-\tilde\scq^3_{\bA,B}(\lambda)
= (m_{\bA,B}-1)\big(R_B(1+\lambda)^2-R_0(1+\lambda)^2\big),\\
\Tr(\scq^3_{\bA,B}(\lambda))-\Tr(\tilde\scq^3_{\bA,B}(\lambda))
= (m_{\bA,B}-1)\int_0^1\Tr(\partial_\theta R_\theta^2)\dif\theta.
\end{gather*}
We also have, by differentiation and cyclicity of the trace,
\begin{gather*}
\partial_\theta R_\theta^2
= -2R_\theta B^*\diff_\theta R_\theta^2-2R_\theta^2B^*\diff_\theta R_\theta,\\
\Tr(\partial_\theta R_\theta^2)=-4\Tr(R_\theta B^*\diff_\theta R_\theta^2).
\end{gather*}
The power counting Lemma~\ref{lem:power_counting} gives
\begin{align}\label{eq:difference_resolvents_trace}
\nonumber |\Tr(R_B(1+\lambda)^2-R_0(1+\lambda)^2)|&\leq \int_0^1|\Tr(\partial_\theta R_\theta^2)|\dif\theta \\
&\lesssim
(1+\lambda)^{-\frac32+4\kappa}
(1+\|B\|_{H^1})^{6\kappa}\|B\|_{H^1}.
\end{align}
Moreover, Lemma~\ref{lem:difference_masses},  using $m_{\0,B}=1$ gives
\[
|m_{\bA,B}-1|
\lesssim
(1+\|B\|_{H^1})^{\frac{15\kappa}{1-3\kappa}}
(1+\triple{\bA}_{-\kappa})^{\frac{2}{1-3\kappa}}
\triple{\bA}_{-\kappa}.
\]
Hence,
\begin{align}\label{eq:Tr_q3_bound}
&|\Tr(\scq^3_{\bA,B}(\lambda)-\tilde\scq^3_{\bA,B}(\lambda))|\lesssim
(1+\lambda)^{-\frac32+4\kappa}
(1+\|B\|_{H^1})^{1+6\kappa+\frac{15\kappa}{1-3\kappa}}
(1+\triple{\bA}_{-\kappa})^{1+\frac{2}{1-3\kappa}}.
\end{align}
This is integrable in $\lambda$ for $\kappa<\frac1{8}$, hence in particular for $\kappa<\frac 1{12}$.

\paragraph{Concluding the lower bound.}
For~\eqref{eq:Q_N_rewritten}, it remains to treat the terms with $\scq^i_{\bA,B}$ for $i=1,2,3$. The bounds on their integrands in~\eqref{eq:Tr_q1_bound},\eqref{eq:Tr_q2_bound} and~\eqref{eq:Tr_q3_bound} show that each integral with $\scq^i_{\bA,B}$ is bounded in absolute value by a multiple of $(1+\bfnorm{\bfA}_{-\kappa})^{\rmn(\kappa)}(1+\|B\|_{H^1})^{1+\tau(\kappa)}$, uniformly in $N$. This finishes the proof.
\end{proof}

\section{Conditional Higgs measure}\label{sec:conditional_Higgs}

\subsection{Set-up of the variational problem}
Let $(\Omega,\mathcal F,(\mathcal F_t)_{t\ge0},\bP)$ carry independent standard real cylindrical Brownian motions
$W_1,W_2$ in $L^2(\T^2;\R)$, and set $X^{\Phi}:=W_1+\bfi W_2$.
Define
\[
  \bH^{\Phi}:=L^2(\R_+;L^2),
  \qquad
  \bH_a^{\Phi}:=L^2_{\mathrm{prog}}(\Omega\times\R_+;L^2).
\]
Given a $\bA\in\cX^{-\kappa}$, recall the renormalised covariant Laplacian $L_{\bA,B}$ from Section~\ref{sec:domain_renorm_laplacian} and the mass $m_{\bA,B}$ from Section~\ref{sec:resolvent}, and introduce the resolvent
\phantomsection\label{def:covariant_GFF_resolvent}\[
  \Sigma_t^{\bA,B}:=(L_{\bA,B}+m_{\bA,B}+t)^{-1},
\]
and define
\phantomsection\label{def:covariant_GFF}\[
  \phi_{\bA,B,N}:=\int_0^N \Sigma_t^{\bA,B}\dif X_t^{\Phi}.
\]
For each $u\in\bH^{\Phi}$ we define a corresponding  shift of our field by setting
\[
  I_N^{\bA,B}(u):=\int_0^N \Sigma_t^{\bA,B}u_t\dif t.
\]
Define also the norm
\[
\|u\|_{\bH^\Phi}^2:=\int_0^N \|u_s\|_{L^2}^2\dif s.
\]
Since the Higgs field is complex-valued and the self-interaction must be invariant under the gauge action $\phi\mapsto g\phi$, with $g$ taking values in the unit circle of $\bC$, we take the interaction to be a polynomial in $|\phi|^2$. We therefore fix for the whole section a polynomial
\begin{equation}\label{eq:def_V}
\mathcal{V}(x):=\sum_{k=0}^{n}c_kx^{k},\qquad n\ge 2,\quad c_n>0,\quad c_0,\dots,c_{n-1}\in\R\;,
\end{equation}
and take $\mathcal{V}(|\phi|^2)$ as the Higgs self-interaction before renormalization.

For any $m\in \R$, define
\begin{equation}\label{eq:def_V_N_m}
\scV_N^{m}(\phi):=\int_{\T^2} \sum_{k=0}^{n}c_k\wick{|\phi|^{2k}}_{N}-\frac{m}2\int_{\T^2} (|\phi|^2-\rmp_N),
\end{equation}
where $\rmp_N \geq 0 $ is a Wick renormalization constant introduced below. For $k\ge 0$, define the Wick powers recursively by $\wick{|\phi|^0}_N =1$ and
\begin{equation}\label{eq:wick_power_recursive}
\wick{|\phi|^{2k}}_{N} = |\phi|^{2k} - \sum_{j=1}^{k} \binom{k}{j}^2 j!\, \rmp_N^j \wick{|\phi|^{2(k-j)}}_N.
\end{equation}

Our goal is to obtain good control, for suitable observables $h$, of the quantity\footnote{The exponent $q\ge1$ is introduced here because the convergence argument (Theorem \ref{thm:higgs_convergence}) requires exponential moments at some $q>1$ for uniform integrability, and the outer variational problem of Section \ref{sec:construction_measure} uses the $q$-th power of the inner Higgs integral.}
\begin{equation}\label{eq:def_cW}
\cW^{h,q,m}_N(\bA,B):=-\log \E[e^{q(h(A,\phi_{\bA,B,N})-\scV_N^m(\phi_{\bA,B,N}))}],
\end{equation}
In particular, we will want to show convergence as $N \uparrow \infty$, of the above quantity and our fundamental tool for doing so will be the Boue-Dupuis variational approach.
As a preliminary step, we observe\footnote{See Corollary~\ref{cor:cov_GFF_finite_cutoff} below.} that $\phi_{\bA,B,N}$ is a well-defined function and moreover one can check the tameness condition from \cite{BG18}\cite{Barash} to conclude that the Boue-Dupuis variational formula applies:
\[
\cW^{h,q,m}_N(\bA,B)=\inf_{u\in \bH^{\Phi}} \E\Big[ -qh(A,\phi_{\bA,B,N}+I^{\bA,B}_N(u))+q\scV_N^m(\phi_{\bA,B,N}+I^{\bA,B}_N(u))+\frac 1 2 \int_0^N\|u_s\|_{L^2}^2\dif s\Big].
\]

\subsection{Covariant Gaussian free field decomposition}
We have the following result:

\begin{theorem}[Covariant GFF decomposition]\label{thm:cov_GFF_decomposition}
    Let $\kappa\in(0,\frac14)$, $\bA,\bar\bA\in\cX^{-\kappa}$ and $B\in\Omega^1H^1$.
    Denote $d=1+3\kappa$.
Then, for every $s<1-3\kappa$,
\begin{equation}\label{eq:cov_GFF_decomposition_difference}
\E\|\phi_{\bA,B,N}-\phi_{\bar\bA,B,N}\|_{H_B^s}^2
\lesssim
\fC_\kappa(\bA,\bar\bA;B)^4
\triple{\bA-\bar\bA}_{-\kappa}^2.
\end{equation}
\end{theorem}

\begin{proof}

By It\^o isometry,
\begin{align} \label{eq:cov_GFF_dec_proof_Ito}
\E\|\phi_{\bA,B,N} &- \phi_{\bar\bA,B,N}\|_{H_B^s}^2= \nonumber\\
&2\int_0^N \| (L_B+1)^{s/2} (R_{\bA,B}(m_{\bA,B}+\lambda)-R_{\bar\bA,B}(m_{\bar\bA,B}+\lambda))\|_{\HS}^2\dif\lambda.
\end{align}
Choose $\gamma>1$ so that
\[
\max\{0,s+1\}<s+\gamma<2-3\kappa.
\]
Then $\| (L_B+1)^{-\gamma/2}\|_{\HS} \lesssim 1$.
Writing $(L_B+1)^{s/2} = (L_B+1)^{-\gamma/2}(L_B+1)^{(s+\gamma)/2}$, we obtain
\begin{align*}
&\| (L_B+1)^{s/2} (R_{\bA,B}(m_{\bA,B}+\lambda)-R_{\bar\bA,B}(m_{\bar\bA,B}+\lambda))\|_{\HS}
\\
&\leq \|(L_B+1)^{-\gamma /2}\|_{\HS}
\|(R_{\bA,B}(m_{\bA,B}+\lambda)-R_{\bar\bA,B}(m_{\bar\bA,B}+\lambda))\|_{L^2 \to H^{s+\gamma}_B}.
\end{align*}
Recall $d=1+3\kappa$. Since $s+\gamma<2-\kappa$, \eqref{eq:pushed_resolvent_difference} gives
\[
\|(R_{\bA,B}(m_{\bA,B}+\lambda)-R_{\bar\bA,B}(m_{\bar\bA,B}+\lambda))\|_{L^2 \to H^{s+\gamma}_B}
\lesssim
(\lambda+1)^{-\frac32 + \frac{3\kappa}{2} + \frac{s+\gamma}{2}}
\fC_\kappa(\bA,\bar\bA;B)^2
\triple{\bA-\bar\bA}_{-\kappa},
\]
from which \eqref{eq:cov_GFF_decomposition_difference} follows
provided that $-3 + 3\kappa + s+\gamma<-1
\Leftrightarrow s+\gamma<2-3\kappa$.
\end{proof}

\begin{corollary}\label{cor:cov_GFF_finite_cutoff}
    In the setting of Theorem~\ref{thm:cov_GFF_decomposition}, one has for any $0<s<1-3\kappa$
    \[
    \E\|\phi_{\bA,B,N}\|_{H^s_B}^2
    \lesssim
    N^{1-3\kappa}
    +
    \fC_\kappa(\bA;B)^4
    \triple{\bA}_{-\kappa}^2.
    \]
\end{corollary}

\begin{proof}
Set $\gamma:=2-3\kappa-s>1$, and apply Lemma~\ref{lem:neg_power_trace} and Lemma~\ref{lem:resolvent_estimate_classical}, to obtain
\begin{align*}
\E \|\phi_{\0,B,N}\|_{H^s_B}^2
&=
2\int_0^N \|(L_B+1)^{s/2}R_B(1+\lambda)\|_{\HS}^2\dif\lambda
\\
&\leq 2\int_0^N \|(L_B+1)^{-\gamma/2}\|_{\HS}^2 \|(L_B+1)^{(s+\gamma)/2}R_B(1+\lambda)\|_{L^2\to L^2}^2\dif\lambda
\\
&\lesssim \int_0^N (1+\lambda)^{-2+s+\gamma}\dif\lambda \; \lesssim N^{1-3\kappa}.
\end{align*}
Combining this estimate with Theorem~\ref{thm:cov_GFF_decomposition} yields the result.
\end{proof}

\subsection{Renormalisation constants and Wick ordering} \label{sec:Higgs_renormalisation}
We quickly review Wick ordering formulae for centred circular complex Gaussians.

Recall that for a real jointly Gaussian family $(X_1,\dots,X_M)$, not necessarily centred, the Wick product is given by the partial pairing formula
\begin{equation}\label{eq:partial_pairing}
\wick{X_1\cdots X_M}=\sum_{P}(-1)^{|P|}\prod_{\{a,b\}\in P}\mathrm{Cov}(X_a,X_b)\prod_{a\notin\mathrm{supp}\,P}X_a,
\end{equation}
the sum running over the partial pairings $P$ of $\{1,\dots,M\}$, with $|P|$ the number of pairs of $P$ and $\mathrm{Cov}(X,Y):=\E[XY]-\E[X]\E[Y]$.
We extend $\wick{\bullet}$ to complex Gaussians by treating them as pairs of real Gaussians corresponding to the real and imaginary parts and extending the Wick product $\bC$-multilinearly.

Recall that a complex Gaussian is said to be circular if its real and imaginary parts are independent and identically distributed with mean zero and equal variance.
Restricting to the circular case gives recursions of the form~\eqref{eq:wick_power_recursive} with the variance playing the role of $\rmp_N$.

We now state an algebraic lemma concerning the Wick products of centred, circular complex Gaussians.
For $\rmz\ge0$ and integers $0\le l<k$ we set
\begin{equation}\label{eq:def_single_counterterm}
\rmv^{(k)}_l(\rmz):=(-1)^{k-l-1}\binom kl^2 (k-l)!\;\rmz^{\,k-l},
\end{equation}

\begin{lemma}\label{lem:Wick_prods}
Let $\zeta$ be a centred circular complex Gaussian, and define $\rmz:=\E[|\zeta|^2]$.
Then, for any $\eta\in\bC$ and any integer $k\ge0$, the Wick power with respect to the base $\rmz$ is
\begin{equation}\label{eq:def_Wick_power}
\wick{|\eta|^{2k}}=\sum_{j=0}^{k}(-1)^j\binom kj^2 j!\,\rmz^{\,j}\,|\eta|^{2(k-j)}=|\eta|^{2k}-\sum_{l=0}^{k-1}\rmv^{(k)}_l(\rmz)\,|\eta|^{2l},
\end{equation}

Moreover, we have the following.
\begin{enumerate}[label=(\roman*)]
\item\label{item:Wick_monomials} For all integers $i,j\ge0$,
\begin{align*}
\wick{\zeta^{i}(\zeta^*)^{j}}=\sum_{r=0}^{i\wedge j}(-1)^r\binom ir\binom jr\,r!\,\rmz^{\,r}\,\zeta^{\,i-r}(\zeta^*)^{\,j-r}.
\end{align*}
In particular $\wick{\zeta^{i}}=\zeta^{i}$ and $\wick{(\zeta^*)^{j}}=(\zeta^*)^{j}$, moreover $\big(\wick{\zeta^{i}(\zeta^*)^{j}}\big)^*=\wick{\zeta^{j}(\zeta^*)^{i}}$, and $\E\big[\wick{\zeta^{i}(\zeta^*)^{j}}\big]=0$ unless $i=j=0$.
\item\label{item:Wick_shift} For every integer $k\ge0$ and every $w\in\bC$,
\begin{align*}
|\zeta+w|^{2k}-\sum_{l=0}^{k-1}\rmv^{(k)}_l(\rmz)\,|\zeta+w|^{2l}=\sum_{i=0}^{k}\sum_{j=0}^{k}\binom ki\binom kj\,\wick{\zeta^{i}(\zeta^*)^{j}}\,w^{\,k-i}(w^*)^{\,k-j},
\end{align*}
equivalently, in real form,
\begin{align*}
\wick{|\zeta+w|^{2k}}=\sum_{i=0}^{k}\binom ki^2\wick{|\zeta|^{2i}}\,|w|^{2(k-i)}+2\Re\sum_{0\le j<i\le k}\binom ki\binom kj\,\wick{\zeta^{i}(\zeta^*)^{j}}\,w^{\,k-i}(w^*)^{\,k-j}.
\end{align*}
\item\label{item:Wick_V} Consequently, for $\mathcal{V}(x)=\sum_{k=0}^{n}c_kx^k$ and with $\rmv_l$ as in~\eqref{eq:def_counterterm_family} below,
\begin{align*}
\mathcal{V}(|\zeta+w|^2)-\sum_{l=0}^{n-1}\rmv_l(\rmz)\,|\zeta+w|^{2l}
=\sum_{k=0}^{n}c_k\sum_{i,j=0}^{k}\binom ki\binom kj\,\wick{\zeta^{i}(\zeta^*)^{j}}\,w^{\,k-i}(w^*)^{\,k-j}.
\end{align*}
\end{enumerate}
\end{lemma}
Collecting the Wick ordering of our potential by degree gives
\[
\wick{\mathcal{V}(|\phi|^2)} =
\sum_{k=0}^{n}c_k\Big[|\phi|^{2k} - \sum_{l=0}^{k-1}\rmv^{(k)}_l(\rmz)\,|\phi|^{2l}\;\Big],
\]
where
\begin{equation}\label{eq:def_counterterm_family}
\rmv_l(\rmz):=\sum_{k=l+1}^{n}c_k\,\rmv^{(k)}_l(\rmz),\qquad l=0,\dots,n-1.
\end{equation}

The two standard real components give the Wick constant $\rmp_N$ in $\scV_N^m$ from~\eqref{eq:def_V_N_m} as
\phantomsection\label{def:flat_Wick_constant}\[
\rmp_N:=2\int_0^N \Tr( R_0(1+\lambda)^2)\dif\lambda=\E[|\phi_{\0,0,N}(0)|^2].
\]
Defining
\[
\rmv_{l,N}:=\rmv_l(\rmp_N)=\sum_{k=l+1}^{n}c_k(-1)^{k-l-1}\binom kl^2(k-l)!\;\rmp_N^{\,k-l},\qquad l=0,\dots,n-1.
\]
we have
\begin{equation}\label{eq:Wick_vs_counterterms}
\wick{\mathcal{V}(|\phi|^2)} = \mathcal{V}(|\phi|^2) - \sum_{l=0}^{n-1}\rmv_{l,N}\,|\phi|^{2l}.
\end{equation}
As mentioned earlier, the Wick constant $\rmp_N$ is not tuned to the variance of $\phi_{\bA,B,N}$ but rather $\phi_{0,0,N}$.
We will also want to use other Wick constants. Since $\phi_{\bA,B,N}-\phi_{\0,B,N}$ is expected to be more regular than either of its two constiutent terms by~Theorem~\ref{thm:cov_GFF_decomposition}, it makes sense to also define the renormalisation constants for $\phi_{\0,B,N}$.
We set
\[
\rmp_{B,N}(x):=\E[|\phi_{\0,B,N}(x)|^2], \qquad
\hat\rmp_{B,N}:=\rmp_{B,N}-\rmp_{N}.
\]
Accordingly, we distinguish the two Wick orderings by subscripts. For $i,j\geq 0$, we set
\[
\wick{\phi^i_{\0,B,N}(\phi^*_{\0,B,N})^j}_N:=\sum_{r=0}^{i\wedge j} (-1)^r \binom ir \binom jr r!\rmp_N^r \phi^{i-r}_{\0,B,N} (\phi^*_{\0,B,N})^{j-r},
\]
and
\[
\wick{\phi^i_{\0,B,N}(\phi^*_{\0,B,N})^j}_{B,N}:=\sum_{r=0}^{i\wedge j} (-1)^r \binom ir \binom jr r!\rmp_{B,N}^r \phi^{i-r}_{\0,B,N} (\phi^*_{\0,B,N})^{j-r},
\]
The renormalisation appearing $\scV_N^m$ is defined using the flat Wick constant $\rmp_N$, whereas the Wick ordering naturally associated with the Gaussian field $\phi_{\0,B,N}$ is the ones based on its variance $\rmp_{B,N}$. We therefore need to relate the different Wick powers through a change-of-base formula which is
\begin{align}\label{eq:change_base}
  \llbracket\phi_{\0,B,N}^i(\phi_{\0,B,N}^*)^j\rrbracket_N
    =\sum_{a=0}^{i\wedge j}\binom ia\binom ja a!\,
      \widehat\rmp_{B,N}^{a}
      \llbracket\phi_{\0,B,N}^{i-a}(\phi_{\0,B,N}^*)^{j-a}\rrbracket_{B,N}.
  \end{align}
We get the following result:

\begin{lemma}\label{lem:wick-powers}
  Let $i,j\geq0$, $p,q\in[1,\infty)$ and $\varepsilon,\eta>0$.  Uniformly in
  $N$ and in Coulomb-gauge $B\in H^1$,
  \[
    \mathbb E\!\left[
      \bigl\|\llbracket\phi_{0,B,N}^{i}
      (\phi_{0,B,N}^{\ast})^{j}\rrbracket_N
      \bigr\|_{W^{-\varepsilon,p}}^q
    \right]^{1/q}
    \leq C(1+\|B\|_{H^1})^{\eta}.
  \]
\end{lemma}

The proof relies on the regularity of the Wick powers $\wick{\Cdot}_{B,N}$ and of $\hat\rmp_{B,N}$.
Below, $W^{-\delta,p}$ carries the flat Bessel-potential norm $\|g\|_{W^{-\delta,p}}=\|(1-\Delta)^{-\delta/2}g\|_{L^p}$.

Since $\phi_{\0,B,N}=\int_0^N\Sigma^{\0,B}_t\dif X^{\Phi}_t$ with $\Sigma^{\0,B}_t=(L_B+1+t)^{-1}$ and $X^\Phi$ a cylindrical Brownian motion in $L^2(\T^2;\bC)$, the field $\phi_{\0,B,N}$ is a centred circular complex Gaussian field: by It\^o's isometry
\begin{equation}\label{eq:cov_kernel_phi}
C_{B,N}(x,y):=\E\big[\phi_{\0,B,N}(x)\phi_{\0,B,N}(y)^*\big]=2\int_0^N R_B(1+\lambda)^2(x,y)\dif\lambda,
\end{equation}
and $\E[\phi_{\0,B,N}(x)\phi_{\0,B,N}(y)]=0$, with $\rmp_{B,N}(x)=C_{B,N}(x,x)$.

\begin{lemma}\label{lem:wickmoment}
For every $\delta>0$, $p\in[1,\infty)$ and $q\in[1,\infty)$, and $i,j \in \N$,
\[
\sup_{N>0}\ \sup_{B\in\Omega^1H^1}\ \E\Big[\big\|\wick{\phi_{\0,B,N}^{\,i}(\phi_{\0,B,N}^{*})^{j}}_{B,N}\big\|^{q}_{W^{-\delta,p}}\Big]<\infty\;.
\]
\end{lemma}

\begin{proof}
Let $B\in\Omega^1_\Coul H^1$ and $N>0$. For almost every $(x,y)\in\T^2\times\T^2$, Lemma~\ref{lem:diamagnetic} gives $|R_B(1+\lambda)(x,y)|\le R_0(1+\lambda)(x-y)$; since $R_0(1+\lambda)\ge0$, composing kernels preserves the domination, so $|R_B(1+\lambda)^2(x,y)|\le R_0(1+\lambda)^2(x-y)$, and~\eqref{eq:cov_kernel_phi} gives
\[
\big|C_{B,N}(x,y)\big|\le2\int_0^N R_0(1+\lambda)^2(x-y)\dif\lambda.
\]
Lemma~\ref{lem:semi-group} with $\gamma=2$ gives $R_0(1+\lambda)^2=\int_0^\infty te^{-(1+\lambda)t}p_t\dif t$, and $\int_0^N e^{-\lambda t}\dif\lambda\le t^{-1}$, so that
\[
\int_0^N R_0(1+\lambda)^2(z)\dif\lambda\le\int_0^\infty te^{-t}t^{-1}p_t(z)\dif t=R_0(1)(z),
\]
uniformly in $N$. Hence
\[
\big|C_{B,N}(x,y)\big|\le 2R_0(1)(x-y)\lesssim 1+\log\frac1{|x-y|}.
\]
Moreover, the standard two-dimensional massive Green's function estimates give $R_0(1)\in L^\rho(\T^2)$ for every $\rho<\infty$, with
$
\|R_0(1)\|_{L^\rho}\lesssim\rho.
$
All these bounds are uniform in $B$ and $N$.

The moment bounds then follow from standard Gaussian estimates combined with the kernel domination established above.
\end{proof}

The next lemma shows that the difference of Wick constants is a suitably regular function whose norm  can be controlled by an arbitrarily small power of $\|B\|_{\Omega^1 H^1}$.
\begin{lemma} \label{lem:comp_renorm_constants}
Let $B\in\Omega^1_\Coul H^1$ and  $j\ge1$ be an integer. For any $s\in[0,1)$ and $q\in(1,\infty)$ with $sq<1$, and any $\delta_0\in(s,1)$,
\[
\big\|\hat\rmp_{B,N}^{\,j}\big\|_{W^{s,q}}\lesssim_{j,s,q,\delta_0}(1+\|B\|_{\Omega^1H^1})^{j\delta_0},
\]
where the implicit constant is independent of $N$ and of $B$.
\end{lemma}

\begin{proof}
We write
\[
\hat\rmp_{B,N}(x)=2\int_0^N a_\lambda(x) \dif \lambda, \qquad a_\lambda(x):=R_B(1+\lambda)^2(x,x)-R_0(1+\lambda)^2(x,x).
\]
We derive an $L^\infty$ bound for $a_\lambda$ and a $B^\eps_{r,1}$ bound for $\diff a_\lambda$, then interpolate.

For the $L^\infty$ bound, we use diamagnetic inequality Lemmas~\ref{lem:diamagnetic} and~\ref{lem:neg_power_trace}, to obtain
\begin{align}\label{eq:L_infty_bound_a_lambda}
\|a_\lambda\|_{L^\infty}\leq 2R_0(1+\lambda)^2(0) =2\Tr(R_0(1+\lambda)^2)\lesssim (1+\lambda)^{-1}.
\end{align}
Note that
\[
\av{\hat\rmp_{B,N},1}=2\int_0^N \Tr(R_B(1+\lambda)^2-R_0(1+\lambda)^2)\dif\lambda.
\]
Then interpolating the $L^\infty$ bound~\eqref{eq:L_infty_bound_a_lambda} and~\eqref{eq:difference_resolvents_trace}, we get
\[
|\av{\hat\rmp_{B,N},1}|\lesssim_\eps (1+\|B\|_{H^1})^\eps.
\]
Regarding the derivative, we use self-adjointness  $R_B(1+\lambda)(x,y)^*=R_B(1+\lambda)(y,x)$ to write
\begin{align*}
\partial_{x_i} a_\lambda(x)&=\partial_{x_i} \int R_B(1+\lambda)(x,y) R_B(1+\lambda)(y,x)\dif y\\
&=\partial_{x_i} \int |R_B(1+\lambda)(x,y)|^2\dif y\\
&=2\Re \int \partial_{x_i}  R_B(1+\lambda)(x,y) R_B(1+\lambda)(x,y)^* \dif y.
\end{align*}
Since $B$ is purely imaginary we can make the derivative a covariant derivative so that
\[
\partial_{x_i} a_\lambda(x)=2\Re \int (\partial_{x_i}+B_i(x)) R_B(1+\lambda)(x,y) R_B(1+\lambda)(x,y)^* \dif y.
\]
In particular for a real-valued $1$-form $K\in\Omega^1 C^\infty$
\begin{align*}
\av{K,\diff a_\lambda}&=2\sum_{i=1}^2 \int K_i(x) \Re \int (\partial_{x_i}+B_i(x)) R_B(1+\lambda)(x,y) R_B(1+\lambda)(x,y)^* \dif y \dif x\\
&=2 \Re\Tr(K^* \diff_B R_B(1+\lambda)^2).
\end{align*}
Thus
\[
\av{K,\diff a_\lambda}=\Re \Tr(\scq^1_{K,B}(\lambda)),
\]
where $\scq^1$ is defined in~\eqref{eq:def_q1}. Applying~\eqref{eq:Tr_q1_bound}, with $\kappa>0$ replaced by $\eps>0$, gives
\[
|\av{K,\diff a_\lambda}|\lesssim_{\eps} \|K\|_{C^{-\eps}} (1+\lambda)^{-1-\gamma_\eps} (1+\|B\|_{H^1})^{1+\beta_\eps},
\]
with $\beta_\eps,\gamma_\eps>0$ such that $\beta_\eps,\gamma_\eps\to 0$ as $\eps\to 0$.  Hence, by  Besov duality (see e.g.~\cite[Sec.~2.11.2]{Triebel})
\[
\|\diff a_\lambda\|_{B^\eps_{1,1}}\lesssim_\eps (1+\lambda)^{-1-\gamma_\eps} (1+\|B\|_{H^1})^{1+\beta_\eps}.
\]
By Poincar\'e inequality, we get
\[
\|a_\lambda-\bar a_\lambda\|_{B^{1+\eps}_{1,1}}\lesssim_{\eps}  (1+\lambda)^{-1-\gamma_\eps} (1+\|B\|_{H^1})^{1+\beta_\eps}.
\]
where $\bar a_\lambda$ is the spatial mean of $a_\lambda$. Interpolating this with~\eqref{eq:L_infty_bound_a_lambda} (see e.g.~\cite[Sec.~2.4.7]{Triebel}) we get
\[
\|a_\lambda-\bar a_\lambda\|_{B^{(1+\eps)/r}_{r,r}}\lesssim_{\eps,r}  (1+\lambda)^{-1-\gamma_\eps/r} (1+\|B\|_{H^1})^{(1+\beta_\eps)/r}, \qquad r\geq 1.
\]

Choose $\eps>0$ small enough such that we can find $r$ such that
\[
\max\Big\{q ,\frac{1+\beta_\eps}{\delta_0}\Big\} <r<\frac{1+\eps}{s}.
\]
Now set
\[
\alpha=\frac{1+\eps}{r},
\]
which by assumption on $r$ implies $s<\alpha$, and since $r>q $, the Besov embedding $B^\alpha_{r,r}\subset W^{s,q }$ (see e.g.~\cite[Sec.~2.3.2]{Triebel}), yields
\[
\|a_\lambda-\bar a_\lambda\|_{W^{s,q }}\lesssim_{s,q ,\eps,r}  (1+\lambda)^{-1-\gamma_\eps/r} (1+\|B\|_{H^1})^{(1+\beta_\eps)/r}.
\]

This bound gives a uniformly bounded integral in $\lambda$ for any sufficiently small $\eps>0$, and
\begin{align*}
\|\hat\rmp_{B,N}\|_{W^{s,q }}&\leq |\av{\hat\rmp_{B,N},1}|+\Big\| \int_0^N a_\lambda-\bar a_\lambda\dif\lambda  \Big\|_{W^{s,q }} \\
&\leq |\av{\hat\rmp_{B,N},1}|+\int_0^N \|a_\lambda-\bar a_\lambda\|_{W^{s,q }}\dif\lambda\\
&\lesssim_{s,q ,\eps,r} (1+\|B\|_{H^1})^\eps+  (1+\|B\|_{H^1})^{\frac{1+\beta_\eps}{r}}.
\end{align*}
Since $1+\beta_\eps<r\delta_0$ by the choice of $r$, taking $\eps>0$ smaller if necessary gives
\[
\|\hat\rmp_{B,N}\|_{W^{s,q }}\lesssim_{s,q ,\delta_0} (1+\|B\|_{H^1})^{\delta_0}.
\]

For $j\ge2$ choose $q _2\in(q ,\infty)$ with $sq _2<1$, which is possible since $sq <1$, and let $q _1\in(1,\infty)$ be defined by $\frac{j-1}{q _1}=\frac1q -\frac1{q _2}$. The fractional product rule (see e.g.~\cite{BookChemin}) applied $j-1$ times gives
\[
\big\|\hat\rmp^{\,j}_{B,N}\big\|_{W^{s,q }}\lesssim\big\|\hat\rmp_{B,N}\big\|^{j-1}_{L^{q _1}}\big\|\hat\rmp_{B,N}\big\|_{W^{s,q _2}},
\]
and we can conclude by applying the case $j=1$ twice: with $(s,q )=(0,q _1)$ to the first factor, and with $p=q _2$ to the second.

\end{proof}

\begin{proof}[Proof of Lemma~\ref{lem:wick-powers}]
By~\eqref{eq:change_base}
  \[
  \llbracket\phi_{\0,B,N}^i(\phi_{\0,B,N}^*)^j\rrbracket_N
    =\sum_{a=0}^{i\wedge j}\binom ia\binom ja a!\,
      \widehat\rmp_{B,N}^{a}
      \llbracket\phi_{\0,B,N}^{i-a}(\phi_{\0,B,N}^*)^{j-a}\rrbracket_{B,N}.
  \]
  The covariance-matched Wick products have moments bounded uniformly in
  $B$ and $N$ in every negative Sobolev space, by the diamagnetic covariance
  bound.  Moreover, for arbitrarily small $\eta_0>0$,
  Lemma~\ref{lem:comp_renorm_constants} gives
  \[
    \|\widehat\rmp_{B,N}^{a}\|_{W^{t,\sigma}}
      \lesssim (1+\|B\|_{H^1})^{a\eta_0}
  \]
  for any $\sigma\in(1,\infty)$ and $t\in(0,\eta_0)$ with $t\sigma<1$.  Applying the
  product rule to each summand and choosing $\eta_0<\eta/(i\wedge j\vee1)$
  proves the claim.
\end{proof}

\subsection{Bounds}

We now state our corresponding drift bound.

\begin{lemma}\label{lem:drift_estimate}
 Let $\kappa\in(0,\frac14)$, $\bA\in\cX^{-\kappa}$ and $B\in\Omega^1H^1$.  For any  $u\in\bH^{\Phi}$ one has
\[
 \| I_N^{\bA,B}(u)\|_{H^{1-\nfrac \kappa2}_B}
 \lesssim
 \fC_\kappa(\bA;B)^{1/2}
 \|u\|_{\bH^\Phi}.
\]
\end{lemma}

\begin{proof}
   Let $f\in L^2$ and $u\in\bH^\Phi$, and note that
    \begin{align*}
    \av{I_N^{\bA,B}(u),f}&=\int_0^N \av{\Sigma_t^{\bA,B}u_t,f}_{L^2}\dif t=\int_0^N \av{u_t,\Sigma_t^{\bA,B}f}\dif t.
    \end{align*}
    Define $U:L^2\to \bH^\Phi$ by
    \[
    (Uf)_t= \Sigma^{\bA,B}_t  f,
    \]
    and note that $U^*=I_N^{\bA,B}$.     Then we have that
    \[
    \int_0^N \|(Uf)_t\|_{L^2}^2\dif t=\int_0^N \|\Sigma_t^{\bA,B} f\|_{L^2}^2\dif t.
    \]
    Note that
    \begin{align*}
   \|\Sigma_t^{\bA,B} f\|_{L^2}&=\|(L_{\bA,B}+m_{\bA,B}+t)^{-1}f\|_{L^2}\leq (1+t)^{-1} \|f\|_{L^2}.
    \end{align*}
    Then
    \[
    \| Uf\|_{\bH^\Phi}^2=\|f\|_{L^2}^2 \int_0^N (1+t)^{-2}\dif t  \leq \|f\|_{L^2}^2.
    \]
    If we now take $f=(L_{\bA,B}+m_{\bA,B})^{1/2}v$, we get
    \[
    \|U (L_{\bA,B}+m_{\bA,B})^{1/2}v\|_{\bH^\Phi}\leq \|(L_{\bA,B}+m_{\bA,B})^{1/2}v\|_{L^2}.
    \]
    In particular, $U(L_{\bA,B}+m_{\bA,B})^{1/2}:\scD((L_{\bA,B}+m_{\bA,B})^{1/2})\to \bH^\Phi$ is bounded and so is its adjoint, namely
    \[
    (U(L_{\bA,B}+m_{\bA,B})^{1/2})^*=(L_{\bA,B}+m_{\bA,B})^{1/2} I_N^{\bA,B}: \bH^\Phi\to \scD((L_{\bA,B}+m_{\bA,B})^{1/2}),
    \]
    and
    \[
    \|(L_{\bA,B}+m_{\bA,B})^{1/2} I_N^{\bA,B}(u)\|_{L^2}\leq \|u\|_{\bH^\Phi}.
    \]
    Using~Corollary~\ref{cor:fractional_power_comparison_LAB} with $\theta=\frac12$, we obtain
    \[
    \| I_N^{\bA,B}(u)\|_{H^{1-\nfrac \kappa2}_B}
    \lesssim
    \fC_\kappa(\bA;B)^{1/2}
    \|(L_{\bA,B}+m_{\bA,B})^{1/2} I_N^{\bA,B}(u)\|_{L^2},
    \]
    and hence the result follows.
\end{proof}

The following lemma gives the key estimate for controlling the Boue-Dupuis formula. Recall that $\mathcal{V}(x)=\sum_{k=0}^n c_kx^k$ with $c_n>0$ and $n\geq 2$, and set
\[
s_n:=1-\frac 1n,
\]
so that $H^{s_n}\subset L^{2n}$ in two dimensions.

Set $s_n=1-1/n$.  The following estimate is the form of the coercive bound
that is compatible with the regularity available for the Gaussian remainder.

\begin{lemma}\label{lem:coercive-bound-higgs}
  Let $n\geq2$, $s\in(s_n,1)$, and let $P(z,z^*)$ be a homogeneous
  polynomial of total degree $1\leq r\leq2n-1$.  For every
  $\delta,\eta>0$ there are $\varepsilon>0$, $p\in(1,\infty)$, $K<\infty$
  and $C<\infty$ such that, for $B\in H^1$, $f\in W^{-\varepsilon,p}$ and
  $Z\in H_B^s\cap L^{2n}$,
  \[
    \left|\int_{\T^2} fP(Z,Z^*)\ \right|
    \leq C\bigl(1+\|f\|_{W^{-\varepsilon,p}}^K\bigr)
      +\delta\|B\|_{H^1}^{\eta}
      +\delta\|Z\|_{H_B^s}^{3/2}
      +\delta\|Z\|_{L^{2n}}^{2n}.
  \]
  The constants are uniform in $B$ and $Z$; $C$ may depend on the
  coefficients of $P$.
\end{lemma}

\begin{proof}
  It is enough to treat a monomial.  If $r=1$, Sobolev duality and comparison
  of the flat and covariant norms give
  \[
    \left|\int fP(Z,Z^*)\right|
      \lesssim \|f\|_{H^{-\varepsilon}}
      (1+\|B\|_{H^1})^\varepsilon\|Z\|_{H_B^s},
  \]
  and the claim follows from Young's inequality after taking
  $\varepsilon$ sufficiently small.

  Let $2\leq r\leq2n-1$ and put $\lambda=2\varepsilon/s$.  Interpolation
  between $L^{2n}$ and $H^s$, followed by the fractional product rule, yields
  \[
    \|P(Z,Z^*)\|_{W^{\varepsilon,p'}}
    \lesssim (1+\|B\|_{H^1})^{2\varepsilon}
      \|Z\|_{H_B^s}^{\lambda}\|Z\|_{L^{2n}}^{r-\lambda},
  \]
  where
  \[
    \frac1p=\frac{2n-r}{2n}-\frac{\varepsilon(n-1)}{sn}.
  \]
  For sufficiently small $\varepsilon$, $p\in(1,\infty)$ and
  \[
    1-\frac{2\lambda}{3}-\frac{r-\lambda}{2n}>0.
  \]
  Duality and Young's inequality now give the asserted estimate.  Since all
  resulting powers of $\|B\|_{H^1}$ are $O_n(\varepsilon)$, they are at most
  $\eta$ after decreasing $\varepsilon$.
\end{proof}

\begin{lemma}\label{lem:potential-bound-higgs}
  Let $\mathcal V(x)=\sum_{k=0}^n c_kx^k$, where $n\geq2$ and $c_n>0$.
  Let $s\in(s_n,1)$ and let $Z$ be an $H_B^s$-valued random variable, not
  necessarily independent of $\phi_{0,B,N}$.  For every $a_0,\delta,\eta>0$
  there are $C,K<\infty$ such that, uniformly in $N$, $a\geq a_0$ and
  $m\in\mathbb R$,
  \begin{align*}
    a\mathbb E \scV_N^m(\phi_{0,B,N}+Z)
    &\leq C(a^K+|m|^K+1)+\delta\|B\|_{H^1}^{\eta}\\
    &\quad+\delta\mathbb E\|Z\|_{H_B^s}^{3/2}
      +C(a+\delta)\mathbb E\|Z\|_{L^{2n}}^{2n},\\
    a\mathbb E \scV_N^m(\phi_{0,B,N}+Z)
    &\geq-C(a^K+|m|^K+1)-\delta\|B\|_{H^1}^{\eta}
      -\delta\mathbb E\|Z\|_{H_B^s}^{3/2}.
  \end{align*}
  Here $C$ depends on $a_0$, $\delta$, $\eta$ and the coefficients of
  $\mathcal V$, but not on $a,m,B,N$ or $Z$.
\end{lemma}

\begin{proof}
  Put $\phi=\phi_{0,B,N}$.  For every $k$ the exact shifted Wick expansion is
  \[
    \llbracket|\phi+Z|^{2k}\rrbracket_N
    =\sum_{i,j=0}^k\binom ki\binom kj
      \llbracket\phi^i(\phi^*)^j\rrbracket_N
      Z^{k-i}(Z^*)^{k-j}.
  \]
  Terms of $Z$-degree between $1$ and $2n-1$ are controlled by
  Lemmas~\ref{lem:coercive-bound-higgs} and~\ref{lem:wick-powers}.  The
  $Z$-free terms are controlled directly by Lemma~\ref{lem:wick-powers}; the
  pure powers of degree below $2n$ are controlled by H\"older and Young.
  The remaining pure term is $ac_n\|Z\|_{L^{2n}}^{2n}$.

  Finally,
  \[
    |\phi+Z|^2-\rmp_N
      =\llbracket|\phi|^2\rrbracket_N
       +2\operatorname{Re}(\phi^*Z)+|Z|^2,
  \]
  so the same estimates cover the linear and $Z$-free mass terms,
  while the quadratic term is bounded by a small multiple of
  $\|Z\|_{L^{2n}}^{2n}$ plus $C(a^K+|m|^K+1)$.  For the lower bound, all small
  $L^{2n}$ terms are absorbed by $ac_n\|Z\|_{L^{2n}}^{2n}$, uniformly
  for $a\geq a_0$.
\end{proof}

\begin{theorem}\label{thm:higgs-W-bound}
  Assume $|h(A,\psi)|\leq c(1+\|A\|^{2}_{C^{-\eta}}+\|\psi\|^{2n}_{C^{-\eta}})$  for some $\eta_0>0$ $c=c(n,\eta)>0$ sufficiently small and let $q\geq1$, $m\in\mathbb R$, and let
  $0<\kappa<1/(4n)$.  Then there are $K<\infty$ and $\tau(\kappa)>0$ (as usual $\tau(\kappa)\to0$ as $\kappa\to0$) such
  that, for every $\delta>0$, uniformly in $N$ and in Coulomb-gauge
  $B\in H^1$,
  \[
    \mathcal W_N^{h,q,m}(\mathbb A,B)
    \geq-C\bigl(q+\|\mathbb A\|_{-\kappa}+1+|m|\bigr)^K
      -\delta\|B\|_{H^1}^{\tau(\kappa)}-c\|A\|_{H^{-\eta}},
  \]
  and
  \[
    \mathcal W_N^{h,q,m}(\mathbb A,0)
    \leq C\bigl(q+\|\mathbb A\|_{-\kappa}+1+|m|\bigr)^K+\|A\|_{H^{-\eta}}.
  \]
\end{theorem}

\begin{proof}
  The Bou\'e--Dupuis formula reads
  \[
    \mathcal W_N^{h,q,m}(\mathbb A,B)
    =\inf_{u\in\mathbb H_a^{ \Phi  }}\mathbb E\!\left[
      -qh(A,\phi_{0,B,N}+Z)+q\scV_N^m(\phi_{0,B,N}+Z)
      +\frac12\|u\|_{\mathbb H^{ \Phi  }}^2\right],
  \]
  where
  \[
    Z=I_N^{\mathbb A,B}(u)+\phi_{\mathbb A,B,N}-\phi_{0,B,N}.
  \]
  Choose $s\in(s_n,1-4\kappa)$.  The drift estimate Lemma~\ref{lem:drift_estimate} and the covariant GFF
  decomposition Theorem~\ref{thm:cov_GFF_decomposition} imply, for some finite powers depending only on $\kappa$,
  \[
    \mathbb E\|Z\|_{H_B^s}^{3/2}
      \leq C((1+\|\mathbb A\|_{-\kappa})^K+\|B\|_{H^1}^{\tau(\kappa)})
        +\varepsilon\mathbb E\|u\|_{\mathbb H^{ \Phi  }}^2.
  \]
  Lemma~\ref{lem:potential-bound-higgs} with $a=q$, together with the growth
  bound on $h$, gives the lower estimate after choosing $\varepsilon$ small
  enough to absorb the drift cost into the entropy term.

  For the upper estimate set $B=0$ and $u=0$.  Then $Z$ is Gaussian, and
  Gaussian moment equivalence, the embedding $H^s\subset L^{2n}$, and the
  covariant GFF decomposition give
  \[
    \mathbb E\|Z\|_{H^s}^{3/2}
      +\mathbb E\|Z\|_{L^{2n}}^{2n}
      \leq C(1+\|\mathbb A\|_{-\kappa})^K.
  \]
  Applying the upper estimate of Lemma~\ref{lem:potential-bound-higgs} and
  the growth bound on $h$ completes the proof.
\end{proof}

\begin{remark}
The upper bound is stated only at $B=0$, which is all the construction requires: in the outer variational problem of \S\ref{sec:construction_measure} the upper bound on $\cW$ enters only through the choice of drift $v=0$ (equivalently $B_T=J_T(0)=0$), while a nonzero drift $B_T=J_T(v)$ is controlled by the lower bound above, valid for all $B\in\Omega^1H^1$ with $\diff^*B=0$.
\end{remark}

\subsection{Convergence}

We conclude this section by proving the convergence, in the $N\to\infty$ limit, of the conditional Higgs free energy $\cW_N^{h,1,m}(\bA,B)$. The (weak) convergence of the associated probability measure is an immediate consequence: we return to this in section~\ref{sec:densities}. We also show that the free energy depends continuously on the gauge field $\bA$, which is of importance for our argument in the next section.
\begin{theorem}
    \label{thm:higgs_convergence}
Let $\kappa\in (0,\kappa_0)$, $\bA,\bar\bA \in \cX^{-\kappa}$, and $B\in\Omega^1 H^1$. Further define $m\coloneqq m_{\bA,B}$ and $\bar m\coloneqq m_{\bar \bA,B}$. Then the following limit exists
    \[
    \lim_{N\to\infty} \cW_N^{h,1,m}(\bA,B) =: \cW_\infty^{h,1,m}(\bA,B).
    \]
Furthermore, there exist exponents $\tau(\kappa)$ and $\rmn(\kappa)$ such that, for all $N\in\N\cup\{\infty\}$,
	\[
	|\cW_N^{h,1,m}(\bA,B)-\cW_N^{h,1,\bar m}(\bar \bA,B)| \lesssim (1+\norm{B}_{\Omega^1 H^1})^{\tau(\kappa)}(1+\triple{\bA}_{-\kappa}+\triple{\bar \bA}_{-\kappa})^{\rmn(\kappa)} \triple{\bA-\bar \bA}_{-\kappa}.
	\]
\end{theorem}
\begin{proof}
For the first part of the theorem, we argue as in the proof of Lemma~\ref{lem:potential-bound-higgs}. Let $\phi_N := \phi_{0,B,N}$ and $Z_N := \phi_{\bA,B,N}-\phi_{0,B,N}$. $\scV_N^m(\phi_{\bA,B,N})=\scV_N^m(\phi_N + Z_N)$ is a continuous function of the complex Wick powers of $\phi_N$ (up to degree $2n$) in the $\wick{\cdot}_N$-ordering and $Z_N,Z_N^*$, each viewed as elements of their respective negative and positive regularity Sobolev spaces. We claim that the former converge in probability as $N\to\infty$ in the corresponding Wiener chaos generated by $X^\Higgs$, and the latter converge in $L^2(\bP^\Higgs)$ as random functions in $H^s_B$ for $0<s<1$.

Indeed, we see from the formula in \eqref{eq:change_base} that any monomial $\wick{\phi_N^i(\phi_N^*)}_N$ is a linear combination of monomials ordered in $\wick{\cdot}_{N,B}$ paired with integer powers of the deterministic function $\hat p_{B,N}$. The $\wick{\cdot}_{N,B}$-ordered Wick powers converge in probability by standard Wiener chaos results, and the deterministic functions converge in $W^{t,\sigma}$ for small enough (positive) $t$ and $\sigma$ by Lemma~\ref{lem:comp_renorm_constants} and the fractional product rule.

The $L^2$ convergence of $Z_N$ can be deduced from Theorem~\ref{thm:cov_GFF_decomposition} by observing that the representation in \eqref{eq:cov_GFF_dec_proof_Ito}, in combination with the $L^1(\diff \lambda)$-integrable bounds shown thereafter, easily yields that the sequence $(Z_N)_{N\ge 0}$ is Cauchy.

The test function $h$ is continuous by assumption and so is the (real) exponential function. We conclude that the random variables
\[
e^{-h(A,\phi_{\bA,B,N}) + \scV_N^m(\phi_{\bA,B,N})}
\]
converge in law as $N\to\infty$. Theorem~\ref{thm:higgs-W-bound} for any $q>1$ implies that these random variables are also uniformly integrable, and so it follows that the limit
\[
\lim_{N\to\infty} \E[e^{-h(A,\phi_{\bA,B,N}) + \scV_N^m(\phi_{\bA,B,N})}]=\E[\lim_{N\to\infty} e^{-h(\phi_{A,\bA,B,N}) + \scV_N^m(\phi_{\bA,B,N})}]
\]
exists. Theorem~\ref{thm:higgs-W-bound} further shows that it is bounded away from $0$ and infinity, and hence the limiting free energy $\cW_\infty^{h,1,m}(\bA,B)$ exists as a real number.

The second part of the theorem will follow if we prove local Lipschitz bounds uniformly in $N$. For simplicity, we set $h=0$. Extending the result to continuous test functions with suitably controlled growth is not difficult. We introduce the shorthand notation
\[
\tilde \scV_\beta := (1-\beta)\scV_N^m(\phi_{\bA,B,N}) + \beta \scV_N^{\bar m}(\phi_{\bar \bA,B,N}),
\]
and we write the difference of free energies as
\begin{equation} \label{eq:Higgs_convergence_proof_free_energy_diff}
\cW_N^{0,1,m}(\bA,B)-\cW_N^{0,1,\bar m}(\bar \bA,B) = \int_0^1\;\frac{\E\Big[\big(\scV_N^{\bar m}(\phi_{\bar \bA,B,N})-\scV_N^m(\phi_{\bA,B,N})\big)e^{-\tilde \scV_\beta}\Big]}{\E\big[e^{-\tilde \scV_\beta}\big]}\,\dif \beta.
\end{equation}
Differentiating under the expectation can be verified by adapting the next following arguments. We thus omit the details. Applying Lemma~\ref{lem:potential-bound-higgs} to
\[
\scV_N^m\big(\phi_{0,B,N}+(\phi_{\bA,B,N}-\phi_{0,B,N}+I_N^{\bA,B}(u))\big)
\]
and
\[
\scV_N^{\bar m}\big(\phi_{0,B,N}+(\phi_{\bar \bA,B,N}-\phi_{0,B,N}+I_N^{\bar \bA,B}(u))\big)
\]
separately, we infer upper and lower bounds for $-\log \E\big[e^{-q\tilde \scV_\beta}\big]$, for any $q\ge 1$, in terms of convex combinations of the bounds provided by Theorem~\ref{thm:higgs-W-bound}. Given any exponent $1<q\le 2$, convexity of the free energy gives
\[
\log \E\big[e^{-q\tilde \scV_\beta}\big] \le (2-q)\log \E\big[e^{-\tilde \scV_\beta}\big] + (q-1) \log \E\big[e^{-2\tilde \scV_\beta}\big]
\]
and hence (for $K$ and $\delta$ as in Theorem~\ref{thm:higgs-W-bound})
\begin{align*}
\frac{1}{q}\log \E\big[e^{-q\tilde \scV_\beta}\big]-\log \E\big[e^{-\tilde \scV_\beta}\big] &\le \frac{q-1}{q}\Big(\log \E\big[e^{-2\tilde \scV_\beta}\big]+2\log \E\big[e^{-\tilde \scV_\beta}\big]\Big) \\
&\lesssim (q-1)\big(6+\triple{\bA}_{-\kappa}+|m|+\triple{\bar \bA}_{-\kappa}+|\bar m|\big)^K \\
&\quad\quad +(q-1)\delta\norm{B}_{H^1}^{\tau(\kappa)} =: (q-1)M(\bA,\bar \bA,B).
\end{align*}
We recall the bounds on the masses $m$ and $\bar m$ implied by Lemma~\ref{lem:difference_masses}. $M(\bA,\bar \bA,B)$ raised to any power will therefore give an acceptable upper bound in view of our claim. Now apply H\"older's inequality with $q=1+M(\bA,\bar \bA,B)^{-1}$ in \eqref{eq:Higgs_convergence_proof_free_energy_diff} to find that
\[
\frac{\Big|\E\Big[\big(\scV_N^{\bar m}(\phi_{\bar \bA,B,N})-\scV_N^m(\phi_{\bA,B,N})\big)e^{-\tilde \scV_\beta}\Big]\Big|}{\E\big[e^{-\tilde \scV_\beta}\big]} \lesssim \E[|\scV_N^{\bar m}(\phi_{\bar \bA,B,N})-\scV_N^m(\phi_{\bA,B,N})|^p]^{1/p}
\]
with $p=M(\bA,\bar \bA,B)$. The proof therefore reduces to showing the next lemma.
\end{proof}
\begin{lemma}
For any $p\ge 2$, there exist $\tau(\kappa)$ and $K_1,K_2>0$ such that
\[
\norm{\scV_N^{\bar m}(\phi_{\bar \bA,B,N})-\scV_N^m(\phi_{\bA,B,N})}_{L^p(\bP^\Higgs)} \lesssim (1+\norm{B}_{H^1})^{\tau(\kappa)}p^{K_1}(1+\triple{\bA}_{-\kappa}+\triple{\bar \bA}_{-\kappa})^{K_2}\triple{\bA-\bar \bA}_{-\kappa}.
\]
\end{lemma}
\begin{proof}
As before, write $\phi:= \phi_{0,B,N}$, $Z_{\bA}:=\phi_{\bA,B,N}-\phi_{0,B,N}$ and likewise for $Z_{\bar \bA}$. A generic term in the difference $\scV_N^{\bar m}(\phi_{\bar \bA,B,N})-\scV_N^m(\phi_{\bA,B,N})$ takes the form
\[
\cT := \int_{\mathbb{T}^2}\; \wick{\phi^i(\phi^*)^j}_N \hat Z P(Z_{\bA},Z_{\bA}^*,Z_{\bar \bA},Z_{\bar \bA}^*)\,\diff x,
\]
where $0\le i,j\le n$, $\hat Z$ stands for $Z_{\bar \bA} - Z_{\bA}$ or $Z_{\bar \bA}^* - Z_{\bA}^*$, and $P$ is some homogeneous polynomial of degree $2n-i-j-1$. For the degree two terms, factors of $m$, $\bar m$ or $\bar m -m$ may also appear. We argue as in the proof of Lemma~\ref{lem:coercive-bound-higgs}, save that we are not concerned about exploiting the coercivity of the potential. We have for any $\eps>0$ and H\"older conjugates $q,q'$,
\[
|\cT|\le \norm{\wick{\phi^i(\phi^*)^j}_N}_{W^{-\eps,q}}\norm{\hat Z P(Z_{\bA},Z_{\bA}^*,Z_{\bar \bA},Z_{\bar \bA}^*)}_{W^{\eps,q'}},
\]
up to any additional factors involving the masses $\bar m$ and $m$. After applying the fractional product rule, we bound the factor with positive regularity by
\[
\norm{Z}_{W^{\eps,\tilde q}} \lesssim (1+\norm{B}_{H^1})^{\eps}\norm{Z}_{H^s_B}
\]
for $s>\eps +1-2/\tilde q$ (which follows by suitable interpolation as in the proof of Lemma~\ref{lem:coercive-bound-higgs}), and all remaining factors by the diamagnetic embedding
\[
\norm{Z}_{L^{2n}} \lesssim \norm{Z}_{H^s_B}
\]
for $s>1-1/n$. There is also a term where $\hat Z$ is replaced by $\bar m-m$, for which Lemma~\ref{lem:difference_masses} supplies an adequate estimate. We thus arrive at a bound of the form
\begin{align*}
|\cT| &\lesssim (1+\norm{B}_{H^1})^{\tau(\kappa)}(1+\triple{\bA}_{-\kappa}+\triple{\bar \bA}_{-\kappa})^K\norm{\wick{\phi^i(\phi^*)^j}_N}_{W^{-\eps,q}} \times \\
&\quad\quad\quad\quad (\norm{\hat Z}_{H^s_B}+\triple{\bA-\bar \bA}_{-\kappa})(\norm{Z_{\bA}}_{H^s_B}+\norm{Z_{\bar \bA}}_{H^s_B})^{2n-1}
\end{align*}
for small enough $\eps$ and large enough $q$, $K$ (such that choosing $s<1$ is permitted), in which any possible mass dependence is already accounted for. We now take expectation, apply H\"older's inequality (with any distribution of finite exponents) and use the estimates provided by Theorem~\ref{thm:cov_GFF_decomposition} and Lemma~\ref{lem:wick-powers}. By Gaussian hypercontractivity (for any $r>2$), we obtain the bounds
\begin{align*}
\norm{\norm{\hat Z}_{H^s_B}}_{L^r(\bP^\Higgs)} &\lesssim r^{1/2} \norm{\norm{\hat Z}_{H^s_B}}_{L^2(\bP^\Higgs)} \\
&\lesssim r^{1/2}(1+\norm{B}_{H^1})^{\tau(\kappa)}(1+\triple{\bA}_{-\kappa}+\triple{\bar \bA}_{-\kappa})^{\rmn(\kappa)}\triple{\bA-\bar \bA}_{-\kappa}
\end{align*}
and
\begin{align*}
\norm{\norm{Z_{\bA}}_{H^s_B}}_{L^r(\bP^\Higgs)} &, \norm{\norm{Z_{\bar \bA}}_{H^s_B}}_{L^r(\bP^\Higgs)} \\
&\lesssim r^{1/2}(1+\norm{B}_{H^1})^{\tau(\kappa)}(1+\triple{\bA}_{-\kappa}+\triple{\bar \bA}_{-\kappa})^{\rmn(\kappa)},
\end{align*}
from Theorem~\ref{thm:cov_GFF_decomposition}. The Wick powers of $\phi$ only contribute a small power of $(1+\norm{B}_{H^1})$. Indeed, tracking the constants behind the proof of Lemma~\ref{lem:wick-powers} yields
\[
\E\big[\norm{\wick{\phi^i(\phi^*)^j}_N}_{W^{\-eps,q}}^r\big]^{1/r} \lesssim r^{\frac{i+j}{2}}(1+\norm{B}_{H^1})^{\tau(\kappa)}.
\]
Combining all the factors produces a $\norm{B}_{H^1}$ dependence with an acceptable exponent, and this completes the proof.
\end{proof}

\section{Construction of the Abelian Yang--Mills--Higgs measure}\label{sec:construction_measure}

\subsection{Set up of the variational problem for the gauge field} \label{sec:gauge_variational_setup}

In the final two sections of this article, we combine all our preceding results to complete the program outlined in the Introduction. Let us recall briefly how we propose to construct the $2$d Abelian Yang-Mills-Higgs EQFT in this setting. As stated in the informal version of our main result (Theorem~\ref{thm:intro_main}), we characterise the measure we construct by its Laplace transform. We concentrate on a class of test functions $h=h(A,\phi)$ with suitable polynomial growth. To make our claims precise, let $C_{{\rm test}}\subset C(\Omega^1 \cS'\times\cS'(\mathbb{T}^2);\R)$ be the space of such functions satisfying the bound
\begin{equation}\label{eq:ymh_test_functions}
|h(A,\psi)|\leq c(1+\|A\|^{2}_{C^{-\kappa/2}}+\|\psi\|^{2n}_{C^{-\kappa}})
\end{equation}
for some positive (small) constant $c>0$. We will generally record any dependence on this constant as a dependence on $h$. Given any $h\in C_{{\rm test}}$ and $q\ge 1$, we define the partition function of our regularised measure $\mu_{N,T}^\YMH$ in the presence of a test function as follows:
\begin{align*}
&\cZ_{N,T}(h,q) := \\
&\int\limits_{\bfi (0,2\pi]^2\times\Omega^1 \cS'(\mathbb{T}^2)}\; e^{- q\cQ_{N,T}(A)} \Big(\int_{\cS'(\mathbb{T}^2)}\; e^{-h(A,\phi) -\cV_{A,N,T}(\phi)}\, \dif\mu_{A,N,T}^\GFF(\phi)\Big)^q \dif\mu_T^{\rmA}(A).
\end{align*}

We have already analysed the objects that appear inside this integral - the Laplace transform of the conditional Higgs measure and the regularised ratio of determinants $Q_{N,T}$ - depending on a generic rough (and suitably enhanced) gauge field $\bA$ or $\bfA$. To complete our construction, it therefore remains to take expectation with respect to the random gauge field $A_T$ and show stability as the UV cut-off represented by the parameter $T$ is removed. We now want to distinguish between the co-exact (Gaussian) and residual harmonic components of $A_T$. We thus write the regularisation $A_T$ of the field in \eqref{eq:intro_pure_YM_field} as
\phantomsection\label{def:gauge_field_decomposition}\[
A_T=\vartheta + A_T^0,
\qquad
A_T^0=\rmP_T^\rmA \diff^*(\diff\diff^*)^{-1}\int_0^1 \diff X_t^{\rmA}.
\]
\phantomsection\label{def:gauge_noise}The cylindrical Brownian motion $X^{\rmA}$ in $L^2(\mathbb{T}^2;\bfi \R)$ is defined on the filtered probability space $(\Omega^\rmA,\mathcal F^\rmA,(\mathcal F_t^\rmA)_{t\ge0},\bP^\rmA)$, and we write $\E^\rmA$ for expectation with respect to its law. The operators $(\rmP_T)_{T\ge0}$ are finite-rank Fourier projectors converging to the identity on $L^2$.

\phantomsection\label{def:gauge_reference_law}The harmonic component $\vartheta$ has normalised Lebesgue law $\bar\dif\vartheta$ on $\bfi (0,2\pi]^2$, with expectation denoted by $\E^\Theta$. This decomposition of $\mu_T^\rmA=\Law(A_T)$ lets us apply the variational formula to the Gaussian field $A_T^0$ for fixed $\vartheta$. The bounds obtained below are uniform in $\vartheta$, so the final integration over $\bfi (0,2\pi]^2$ is straightforward.

Since we purpose to rely on the Bou\'e-Dupuis formula (see the outline of our method in section~\ref{sec:intro_summary}) for our treatment of $A_T^0$, we also introduce the corresponding space of stochastic controls or drifts. Let
\[
  \bH^{\rmA}:=L^2(\R_+;L^2),
  \qquad
  \bH_a^{\rmA}:= L^2_{\mathrm{prog}}(\Omega^\rmA\times\R^+;L^2),
\]
and define
\[
J_T(v):= \rmP_T^\rmA \diff^*(\diff\diff^*)^{-1}\int_0^T v_t\dif t
\]
for $v\in \bH^\rmA$. As previously, we generally abbreviate $J_T(V)=B_T$ where there is no ambiguity in relation to the drift $v$.

From the perspective of the variational problem for the $A_T^0$ component of the random gauge field $A_T$, the spatial average $\vartheta$ is held fixed and thus forms part of the variational functional. Since $\vartheta$ is a random element of $\Omega^1_C H^1$, translation by $\vartheta$ is a well-defined operation in our enhanced state space $\cX_V^{-\kappa}$. By associativity, we may think of $(\bfA_T \bplus B_T)$ as $\bfA_T^0 \bplus (B_T+\vartheta)$, and we shall write ${}_\vartheta{B_T}\coloneqq B_T+\vartheta$ to lighten our notation below. We then have:
\begin{align*}
\bfnorm{\bfA_T}_{-\kappa} &\le (1+2|\vartheta|)\bfnorm{\bfA_T^0}_{-\kappa} + |\vartheta| + 2|\vartheta|^2 \lesssim 1+\bfnorm{\bfA_T^0}_{-\kappa}, \textup{ and } \\
\bfnorm{\bfA_T-\bar\bfA_T}_{-\kappa} &\le (1+2|\vartheta|)\bfnorm{\bfA_T^0-\bar\bfA_T^0}_{-\kappa} \lesssim \bfnorm{\bfA_T^0-\bar\bfA_T^0}_{-\kappa}.
\end{align*}

To set up the variational problem we want to analyse, we thus define \[
\cH^F_{N,T}(A^0,\vartheta):=Q_{N,T}(A) -\log \int_{\cS'(\mathbb{T}^2)}\; e^{- h(A,\phi) -\cV_{A,N,T}(\phi)} \dif\mu_{A,N,T}^\GFF(\phi),
\]
such that
\[
\cZ_{N,T}(h,q) = \E^\Theta\E^\rmA\Big[e^{-q\cH^F_{N,T}(A_T^0,\vartheta)}\Big].
\]
Before considering this functional further, we prove that the form of the renormalised vacuum polarisation anticipated in Section~\ref{sec:enhancement_determinants} is appropriate for the random gauge field $A_T$. This allows us to apply the results of Section~\ref{sec:determinants} to the variational problem above.

\subsection{Convergence of renormalised vacuum polarisation}

Most of the estimates and properties we require for the regularised ratio of determinants $Q_{N,T}$ follow from our treatment of the pathwise object $\cQ_N$ in Section~\ref{sec:determinants}. In this section, we address the vacuum polarisation of the random gauge field $A_T$. As discussed in Section~\ref{sec:enhancement_determinants}, this quantity requires renormalisation, which relies on the Gaussian structure of $A_T^0$. More precisely, we show by a probabilistic argument that our choice of $\rmo_{N,T}$ ensures that the renormalised (or Wick ordered) vacuum polarisation $\bV_{A_T}$ defined in~\eqref{eq:def_vacuum_A_T} has a meaningful interpretation as $N\to\infty$ and $T\to\infty$.

For ease of reference, we recall the definitions
\begin{align*}
\rmo_{N,T} &= -\E^\Theta\E^\rmA\Big[\int_0^N\; \Tr \scq^2_{(A_T,A^*_T A_T-\rmc_T),0}(\lambda)\,\diff \lambda \Big],\\
\bV_{A_T}(\lambda) &= \boldsymbol{\Pi}^\lambda(A_T,A_T)-\E^\Theta\E^\rmA[\boldsymbol{\Pi}^\lambda(A_T,A_T)].
\end{align*}
It is then a simple exercise to compute the above expectation and arrive at the following explicit formula for $\rmo_{N,T}$:
\begin{align*}
    \rmo_{N,T} = \int_0^N\; \rmc_T \Tr\big(R_0(1+\lambda)^2\big) - \Pi_0(\lambda) -\sum_{j,k=1}^2\sum_{p\in 2\pi\Z^2\setminus\{0\}}\chi_T(p)^2\frac{(-1)^{j+k}p_{j+1}p_{k+1}}{|p|^4} \Pi^\lambda_{j,k}(p) \,\diff \lambda,
\end{align*}
where the indices of the Fourier variable $p$ are understood modulo $2$ and
\[
\Pi_0(\lambda) \coloneqq \int_{\bfi (0,2\pi]^2}\; |\vartheta_1|^2\Pi^\lambda_{1,1}(0) + |\vartheta_2|^2\Pi^\lambda_{2,2}(0)\,\bar\dif \vartheta.
\]

It remains to verify our claim in Section~\ref{sec:enhancement_determinants} that this construction produces an element of the enhanced space $\cX_V^{-\kappa}$, so that we may use the pathwise results established in Section~\ref{sec:determinants}. This amounts to checking that $\bV_{A_T}$ is $L^1(\R^+;\R)$-integrable, which follows directly from Lemma~\ref{lem:vac_pol_bound} or, alternatively, Proposition~\ref{prop:Pi_quadratic_form}.

\begin{corollary} \label{cor:bV_well_defined}
    The renormalised vacuum polarisation $\bV_{A_T}$ is integrable, i.e., $\bV_{A_T}\in L^1(\R^+;\R)$, almost surely for all $T\ge 0$. It therefore holds that $Q_{N,T}(A_T+B_T) = \cQ_N(\bfA_T,B_T)$ and the limit
    \[
    \cQ_\infty(\bfA_T,0) = \lim_{N\to\infty} \cQ_N(\bfA_T,0)
    \]
    is well-defined.
\end{corollary}
\begin{proof}
    For a fixed $T$, Lemma \ref{lem:vac_pol_bound} gives the inequality
    \begin{align*}
        |\bV_{A_T}(\lambda)| &\lesssim (1+\lambda)^{-2}|\vartheta|^2 +(1+\lambda)^{-2} \sum_{|p|^2\le 1+\lambda} (1+|p|^2)|\hat A_T^0(p)|^2 \\
        &+ (1+\lambda)^{-1}\sum_{|p|^2> 1+\lambda} \frac{|p|^2}{|p|^2}|\hat A_T^0(p)|^2 \le (1+\lambda)^{-2} \Big(|\vartheta|^2+\sum_{|p|\le T} \langle p\rangle^2|\hat A_T^0(p)|^2\Big),
    \end{align*}
    which yields an integrable bound up to an a.s. finite constant. The same computation verifies that $\E^\Theta\E^\rmA[\bV_{A_T}]$ is integrable. The remaining claims now follow from our definitions and Theorem~\ref{thm:convergence_cQ_N} by dominated convergence.
\end{proof}

The main task of this section is to remove the gauge field regularisation. To use the continuity estimate in Theorem~\ref{thm:convergence_cQ_N}, we need convergence in the $\bfnorm{\cdot;\cdot}_{-\kappa}$ topology, and we establish the $T\to\infty$ limit of the vacuum polarisation probabilistically. The next result collects all the properties we need.
\begin{lemma} \label{lem:bV_properties}
    The sequence $(\bV_{A_T})_{T\ge 0}$ converges in $L^2(\bar\dif\vartheta\otimes\bP^\rmA;L^1(\R^+;\R))$ to a random limit denoted by $\bV_A$. Furthermore, it holds that
    \[
    \int_0^\infty\; \sup_T\E^\Theta\E^\rmA[|\bV_{A_T}(\lambda)|^2]^{1/2}\,\diff \lambda <\infty,
    \]
    and $\E^\Theta\E^\rmA[\norm{\bV_A}_{L^1(\R^+;\R)}^p]<\infty$ for all $p\ge 1$

\end{lemma}
\begin{proof}
    Since $\vartheta$ is bounded and constant in $T$, it is enough to consider the Gaussian part $A_T^0$ and show that $(\bV_{A_T^0})_{T\ge 0}$ is Cauchy in the stronger $L^1(\R^+;L^2(\bP^\rmA))$ norm, which will imply the first claim. To simplify the computation, we introduce the shorthand notation
    \[
    \hat{A^0}_{T,j,k}(p) \coloneqq \hat{A^0}_{T,j}(p)^* \hat{A^0}_{T,k}(p) \, \textup{ and } \, e_{T,j,k}(p) \coloneqq \frac{(-1)^{j+k}\chi_T(p)^2p_{j+1}p_{k+1}}{|p|^2}.
    \]
    Given any $S,T\ge 0$, we begin by evaluating the expectation
    \begin{align*}
        &\E^\rmA\big[\big(\hat{A^0}_{T,j,k}(p)-e_{T,j,k}(p) - (\hat{A^0}_{S,j,k}(p)-e_{S,j,k}(p))\big)^* \times \\
        &\quad\quad \big(\hat{A^0}_{T,j',k'}(p')-e_{T,j',k'}(p') - (\hat{A^0}_{S,j',k'}(p')-e_{S,j',k'}(p'))\big)\big] \\
        &\quad = \E^\rmA\big[\hat{A^0}_{T,j,k}(p) \hat{A^0}_{T,j',k'}(p') \big] - \E^\rmA\big[\hat{A^0}_{T,j,k}(p) \hat{A^0}_{S,j',k'}(p') \big]  \\
        &\quad\quad - \E^\rmA\big[\hat{A^0}_{S,j,k}(p) \hat{A^0}_{T,j',k'}(p') \big] + \E^\rmA\big[\hat{A^0}_{S,j,k}(p) \hat{A^0}_{S,j',k'}(p') \big] \\
        &\quad\quad - e_{T,j,k}(p)e_{T,j',k'}(p')  + e_{T,j,k}(p)e_{S,j',k'}(p') + e_{S,j,k}(p)e_{T,j',k'}(p') - e_{S,j,k}(p)e_{S,j',k'}(p') \\
        &\quad = 2e_{T,j,k}(p)e_{T,j',k'}(p) - 2e_{T,j,k}(p)e_{S,j',k'}(p) - 2e_{S,j,k}(p)e_{T,j',k'}(p) + 2e_{S,j,k}(p)e_{S,j',k'}(p)
    \end{align*}
    using Wick's theorem in combination with the covariance of $A_T^0$. This yields the expression
    \begin{align} \label{eq:vac_pol_variance}
        \E^\rmA[|\bV_{A_T^0}(\lambda)-\bV_{A_S^0}(\lambda)|^2] = 2\sum_{p\in 2\pi\Z^2}\sum_{\substack{j,j', \\ k,k'=1}}^2  & \frac{(-1)^{j+j'+k+k'}p_{j+1}p_{k+1}p_{j'+1}p_{k'+1}}{|p|^8}\times \nonumber\\
        &\Pi^\lambda_{j,k}(p)\Pi^\lambda_{j',k'}(p) \big(\chi_T(p)^2-\chi_S(p)^2\big)^2
    \end{align}
    for the variance of the vacuum polarisation. From Lemma \ref{lem:vac_pol_bound}, we obtain the estimates
    \begin{align*}
    |\Pi^\lambda_{j,k}(p)\Pi^\lambda_{j',k'}(p)|\,|p|^{-4}  &\lesssim (1+\lambda)^{-9/4}\,|p|^{-7/2} \, \textup{ if } \, |p| > \sqrt{1+\lambda}, \, \textup{ and } \\
    |\Pi^\lambda_{j,k}(p)\Pi^\lambda_{j',k'}(p)|\,|p|^{-4}  &\lesssim (1+\lambda)^{-4} \, \textup{ if } \, |p| \le \sqrt{1+\lambda}.
    \end{align*}
    Without loss of generality we now assume $S\le T$. The above then implies the inequality
    \[
    \E^\rmA[|\bV_{A_T^0}(\lambda)-\bV_{A_S^0}(\lambda)|^2] \lesssim (1+\lambda)^{-5/4}\sum_{\substack{p\in 2\pi\Z^2 \\ |p|^2 \ge 1+\lambda\vee S^2}}  |p|^{-7/2}  + (1+\lambda)^{-4}\sum_{\substack{p\in 2\pi\Z^2 \\ S^2\le |p|^2 \le 1+\lambda}} 1.
    \]
    We have thus shown that
    \begin{align*}
    ||\bV_{A_T^0}-\bV_{A_S^0}||_{L^1(\lambda)L^2(\bP^\rmA)} &\lesssim \Big(\sum_{\substack{p\in 2\pi\Z^2 \\ |p|^2 \ge S^2}} |p|^{-7/2}\Big)^{1/2}\int_0^\infty\; (1+\lambda)^{-9/8}\,\diff \lambda \\
    &\quad\quad\quad\quad + \int_{S^2-1}^\infty\; (1+\lambda)^{-2}\Big(\sum_{\substack{p\in 2\pi\Z^2 \\ |p|^2 \le 1+\lambda}} 1\Big)^{1/2} \,\diff \lambda \\
    &\lesssim \Big(\sum_{\substack{p\in 2\pi\Z^2 \\ |p|^2 \ge S^2}} |p|^{-7/2}\Big)^{1/2} + \int_{S^2-1}^\infty\; (1+\lambda)^{-3/2} \,\diff \lambda \xrightarrow{S\to\infty} 0,
    \end{align*}
    as desired. A similar reasoning starting from \eqref{eq:vac_pol_variance} with $S=T-1$ yields the bound
    \[
    ||\bV_{A_T^0}-\bV_{A_{T-1}^0}||_{L^1(\lambda)L^2(\bP^\rmA)} \lesssim \frac{1}{(T-1)^{5/4}},
    \]
    which is summable and thus implies an integrable bound for $\E^\rmA[|\bV_{A_T^0}|^2]^{1/2}$ that is uniform in the cut off $T$. We therefore deduce the second claim. We now conclude by noting that $\bV_{A_T^0} \xrightarrow{L^2(\bP^\rmA)} \bV_{A^0}$ is a renormalised quadratic functional and hence an element of the second Wiener chaos. It follows from Gaussian hypercontractivity that its higher moments are controlled by the second moment.
\end{proof}

\subsection{Construction of the two-dimensional Abelian YMH measure}

We are now in a position to formally state the main result of this work. We briefly recall the divergent counterterms introduced in this article to give rigorous meaning to various limiting objects. The Wick counterterm $\rmc_T=\E^\rmA[|A_T^0(0)|^2]$ was defined in Section~\ref{sec:def_regularised_measure}, \eqref{eq:gauge_wick_counterterm}. The Wick renormalisation of the Higgs interaction potential is constructed in terms of the constant $\rmp_N=\E^\Higgs[|\phi_{\0,0,N}(0)|^2]$ defined in section~\ref{sec:Higgs_renormalisation}. Finally, the renormalisation of the vacuum polarisation $\rmo_{N,T}$ is introduced in section~\ref{sec:enhancement_determinants}. We take these constants as part of our definition of the regularised measure $\mu^\YMH_{N,T}$ and the partition function $\cZ_{N,T}(h,q)$. We also note that the statements proved in previous sections generally require us to choose the regularity parameter $\kappa$ sufficiently small. To avoid repeating this requirement throughout this section, we will explicitly state it only once in our main theorem. With these conventions settled, we will now prove the following.

\begin{theorem}\label{thm:main}
    There exists a $\bar \kappa>0$ such that the following hold for all $\kappa\in (0,\bar \kappa)$. The limit (in any order of $T$ and $N$)
    \[
    \cZ(h) \coloneqq \lim_{N,T\to\infty} \cZ_{N,T}(h,1)
    \]
    exists for every $h\in C_{{\rm test}}$ for all sufficiently small $c>0$, and satisfies $0 < \cZ(h) < \infty$.
    In particular, $\mu^{\YMH}_{N,T}$ converges weakly to a probability measure on $\Omega^1 \cS'(\mathbb{T}^2)\times \cS'(\mathbb{T}^2)$.
\end{theorem}

We begin by showing uniform bounds for the quantity $\cZ_{N,T}(h,q)$. As advertised, for the main part these are obtained from a variational formulation of the (outer) expectation with respect to the Gaussian field $\tilde A_T$. In studying this variational problem, we rely in turn on the principal results established in Sections~\ref{sec:determinants} and~\ref{sec:conditional_Higgs} concerning, respectively, the conditional Higgs measure and the regularised ratio of determinants $Q_{N,T}$.

We first rewrite the variational functional $\cH^h$ using the functionals on the enhanced space $\cX_V^{-\kappa}$ from our pathwise results. The identities below follow from the definitions.
\begin{lemma} \label{lem:cH_identities}
    With $\bfA_T^0=(A_T^0,A_T^{0\,*} A_T^0-\rmc_T,\bV_{A_T^0})$ and $B_T\in\Omega_C^1 H^1$, the identities
    \begin{align*}
    \cH^F_{N,T}(A_T^0,\vartheta) &= \cQ_N(\bfA_T,0)+\cW_N^{h,1,m_{\bA_T,0}}(\bA_T,0), \textup{ and } \\
    \cH^F_{N,T}(A_T^0+B_T,\vartheta) &= \cQ_N(\bfA_T^0,{}_\vartheta{B_T})+\cW_N^{h,1,m_{\bA_T^0,{}_\vartheta{B_T}}}(\bA_T^0,{}_\vartheta{B_T})
    \end{align*}
    hold $\bP^\rmA\times \bar \dif\vartheta$-almost surely.
\end{lemma}
\begin{proof}
    This essentially follows by definition from the properties of the translation operation $\bfA_T\bplus B_T$ as defined in~\eqref{eq:translation_bfA}. For the ratio of determinants, the relevant inclusion in the enhanced state space $\cX_V^{-\kappa}$ is Corollary~\ref{cor:bV_well_defined}, and then Proposition~\ref{prop:translation_determinant} supplies the desired identity.
\end{proof}

The following is the main technical estimate we shall need.

\begin{lemma} \label{lem:BD_A_bounds}
    Fix $q\ge 1$ and $\eps>0$. For all $T,N\ge 0$, all test functions $h\in C_{{\rm test}}$ with sufficiently small growth parameter $c$, and all drifts $v\in \bH^\rmA$, there exist positive, measurable functions $\mathscr{Q}_T$ independent of $v$ (and $N$) and a $C_E>0$ such that
    \begin{align*}
    q\cH_{N,T}^h(A_T^0+B_T,\vartheta) &\ge -\mathscr{Q}_T - (4qC_Ec+\eps) \norm{v}_{\bH^\rmA}^2,\\
    q\cH_{N,T}^h(A_T^0,\vartheta) &\le \mathscr{Q}_T,\\
    \sup_T \E^\rmA[\mathscr{Q}_T] &\le C(h,\kappa,q,\eps),
    \end{align*}
    with the first two bounds uniform in $N$.
\end{lemma}
\begin{proof}
   We rely on the representation of $\cH_{N,T}^h$ in Lemma~\ref{lem:cH_identities} and consider separately each term that appears there. Note that $\diff^*({}_\vartheta{B_T})=0$: the drift $B_T=J_T(v)$ is co-exact by construction and $\vartheta$ is a constant form. Provided that $c$ is small enough, Theorem~\ref{thm:higgs-W-bound} ensures the existence of exponents $\tau(\kappa)$ and $\rmn(\kappa)$ such that, for all $\delta>0$, the inequality
    \begin{align*}
    \cW_N^{h,1,m_{\bA_T^0,{}_\vartheta{B_T}}}(\bA_T^0,{}_\vartheta{B_T}) &\ge -C(h,\kappa,\delta,\vartheta)\big(1+\triple{\bA_T^0}_{-\kappa}^{\rmn(\kappa)} + m_{\bA_T^0,{}_\vartheta{B_T}}^{\rmn(\kappa)}\big) - \delta\norm{B_T}_{\Omega^1 H^1}^{\tau(\kappa)} \\
    &\quad\quad -c \norm{A_T+{}_\vartheta{B_T}}_{\Omega^1 C^{-\kappa/2}}^2,
    \end{align*}
    holds uniformly in $N$ for some constant $C(h,\kappa,\delta,\vartheta)>0$ and all $\kappa$ small enough. We use here that $\norm{{}_\vartheta{B_T}}_{\Omega^1 H^1}^{\tau(\kappa)} \le 2|\vartheta| +  2\norm{B_T}_{\Omega^1 H^1}^{\tau(\kappa)}$ for $\tau(\kappa)\le 1$. For the mass $m_{\bA_T^0,{}_\vartheta{B_T}}$, Lemma~\ref{lem:difference_masses} implies the bound
    \begin{align*}
    m_{\bA_T^0,{}_\vartheta{B_T}} &\le 1 + C(\kappa)(1+\norm{{}_\vartheta{B_T}}_{\Omega^1 H^1})^{\tau(\kappa)}(1+\triple{\bA_T^0}_{-\kappa}^3),\\
    \shortintertext{which in turn yields}
    C(h,\kappa,\delta,\vartheta) m_{\bA_T^0,{}_\vartheta{B_T}}^{\rmn(\kappa)} &\le C(h,\kappa,\delta,\vartheta) + C(h,\kappa,\delta,\vartheta)(1+\norm{{}_\vartheta{B_T}}_{\Omega^1 H^1})^{\tau(\kappa)}(1+\triple{\bA_T^0}_{-\kappa}^3)^{\rmn(\kappa)} \\
    &\le C(h,\kappa,\delta,\vartheta)(1+\triple{\bA_T^0}_{-\kappa}^{\rmn(\kappa)}) + \delta\norm{B_T}_{\Omega^1 H^1}^{\tau(\kappa)}
    \end{align*}
    for a suitably redefined $C(h,\kappa,\delta,\vartheta)>0$. By Besov-H\"older embedding, there is an analytic constant $C_E$ (bounded above, for example, by the constant in Bernstein's inequality on $\mathbb{T}^2$) such that
    \[
    c \norm{A_T+{}_\vartheta{B_T}}_{\Omega^1 C^{-\kappa/2}}^2 \le 2c\norm{A_T}_{\Omega^1 C^{-\kappa/2}}^2 + 2^{2+\kappa} c |\vartheta| + 4c C_E \norm{B_T}_{\Omega^1 H^1}^2.
    \]
    Combining the above, we obtain
    \[
    q\cW_N^{h,q,m_{\bA_T^0,{}_\vartheta{B_T}}}(\bA_T^0,{}_\vartheta{B_T}) \ge -C(h,\kappa,\delta,q,\vartheta)\big(1+\triple{\bA^0}_{-\kappa}^{\rmn(\kappa)}\big) - 2q\delta\norm{B_T}_{\Omega^1 H^1}^{\tau(\kappa)} - 4qc C_E \norm{B_T}_{\Omega^1 H^1}^2
    \]
    valid for all $\kappa$ small enough. Arguing similarly starting from the upper bound in Theorem~\ref{thm:higgs-W-bound}, we also see that
    \[
    q\cW_N^{h,1,m_{\bA_T,0}}(\bA_T,0) \le C(h,\kappa,\delta,q,\vartheta)\big(1+\triple{\bA^0}_{-\kappa}^{\rmn(\kappa)}\big)
    \]
    for a possibly larger $C(h,\kappa,\delta,q,\vartheta)>0$. For the ratio of determinants, Theorems~\ref{thm:lower_bound_det} and~\ref{thm:convergence_cQ_N} give the lower and upper bounds, respectively,
    \begin{align*}
    q\cQ_{N}(\bfA_T^0,{}_\vartheta{B_T}) &\ge -C(\kappa,\delta,q,\vartheta)(1+\bfnorm{\bfA_T^0}_{-\kappa})^{\rmn(\kappa)} - q\delta\norm{B_T}_{\Omega^1 H^1}^{1+\tau(\kappa)},\\
    q\cQ_{N}(\bfA_T) &\le C(\kappa,q,\vartheta)(1+\bfnorm{\bfA_T^0}_{-\kappa})^{\rmn(\kappa)},
    \end{align*}
    both uniformly in $N$. Choosing $\delta\le \eps/(3q)$ and noting that $\norm{B_T}_{\Omega^1 H^1}\le \norm{v}_{\bH^\rmA}$, we thus arrive at the desired inequalities for $q\cH_{N,T}^h(\cdot,\vartheta)$ with
    \[
    \mathscr{Q}_T = C(h,\kappa,\eps,q,\vartheta)(1+\bfnorm{\bfA_T^0}_{-\kappa}^{\rmn(\kappa)})
    \]
    independent of $v$ and $N$. As $\vartheta$ varies over a compact set, we may drop the $\vartheta$-dependence by replacing the constant above by its supremum over $\vartheta\in \bfi (0,2\pi]^2$. The uniform bound on $\E^\rmA[\mathscr{Q}_T]$ now follows from Lemma~\ref{lem:bV_properties} and the standard two-dimensional estimate
    \[
    \sup_T \E^\rmA[\triple{\bA_T^0}_{-\kappa}^N]<\infty, \textup{ for all } \kappa>0,N\in\N.
    \]
\end{proof}

To apply the Bou\'e-Dupuis formula to the outer expectation, we need to check that the relevant functional of the noise $X^\rmA$ is tame. We rely on the coercivity of the interaction potential for a lower bound apart from the contribution of $h$, which is allowed small quadratic growth in absolute value.

\begin{lemma} \label{lem:tameness_lower_bound}
    For all test functions $h\in C_{{\rm test}}$ with sufficiently small $c>0$, and all $N,T\ge 0$, the functional $q\cH_{N,T}^h(A_T^0,\vartheta)$ is bounded from below by
    \[
    q\cH_{N,T}^h(A_T^0,\vartheta) \ge - C(q,T,N,h) -2qc \norm{A_T^0}_{\Omega^1 C^{-\kappa/2}}^2
    \]
\end{lemma}
\begin{proof}
    It suffices to consider the case $q=1$. Note that $m_{\bA_T,0} \le 1+\rmc_T<\infty$ for any finite $T$. Moreover, by assumption, we have
    \begin{align*}
    |h(A_T,\phi_{\bA_T,0,N})| &\le c+c \norm{A_T}_{\Omega^1 C^{-\kappa/2}}^2 + c \norm{\phi_{\bA_T,0,N}}_{H^{-\kappa}}^{2n} \\
    &\le C(h,\vartheta)+c\norm{\phi_{\bA_T,0,N}}_{L^{2n}}^{2n} + 2c \norm{A_T^0}_{\Omega^1 C^{-\kappa/2}}^2
    \end{align*}
    for any $T$ and $N$. We may drop the $\vartheta$ dependence as before by compactness. For any $c<c_n$, the lead coefficient of the Higgs interaction potential, we thus see that
    \begin{align*}
    -|h(A_T,\phi_{\bA_T,0,N})|+\scV_N^{m_{\bA_T,0}}(\phi_{\bA_T,0,N}) &\ge -C(N,T,h) - 2c \norm{A_T^0}_{\Omega^1 C^{-\kappa/2}}^2,\\
    \shortintertext{and hence}
    \cW_N^{h,1,m_{\bA_T,0}}(\bA_T,0) &\ge -C(T,N,h) - 2c \norm{A_T^0}_{\Omega^1 C^{-\kappa/2}}^2.
    \end{align*}
    We consider the difference of resolvents in $Q_{N,T}$ next. It suffices to study the trace for $\lambda=0$. With a view to using the heat kernel representation in Lemma \ref{lem:semi-group}, we note that the diamagnetic inequality implies that
    \[
    0\le p_{\bA_T}(t,x,x) \le p_0(t,0)
    \]
    on the diagonal, and so, by definition of $m_{\bA_T,0}$,
    \[
    e^{-t}p_0(t,0)-e^{-(m_{\bA_T,0}-\rmc_T)t}p_{\bA_T}(t,x,x) \ge (1-e^{\rmc_T t})p_0(t,0).
    \]
    Since the spectrum of $L_{\bA_T}-\rmc_T+m_{\bA_T,0}$ is bounded below by $1$, the semi-group property of the covariant heat kernel yields the estimate (for $t\ge 1$):
    \[
    e^{-(m_{\bA_T,0}-\rmc_T)t}p_{\bA_T}(t,x,x) \le e^{-(t-1)}e^{-(m_{\bA_T,0}-\rmc_T)}p_{\bA_T}(1,x,x) \le e^{1+\rmc_T -t}p_1(0).
    \]
    We now see that
    \begin{align*}
    \int_0^\infty\; e^{-t}p_0(t,0) -  e^{-(m_{\bA_T,0}-\rmc_T)t}p_{\bA_T}(t,x,x)\,\diff t &\ge \int_0^1\; (1-e^{\rmc_T t})p_0(t,0)\,\diff t \\
    &\quad\quad + \int_1^\infty\; e^{-t}p_0(t,0) - e^{1-t}e^{\rmc_T}p_1(0)\,\diff t,
    \end{align*}
    which is bounded below depending on $\rmc_T$. It follows that
    \[
    \int_0^N \Tr \Big((L_0+1+\lambda)^{-1}-(L_{\bA_T}-\rmc_T+m_{\bA_T,0}+\lambda)^{-1} \Big)\dif \lambda \ge -C(N,T).
    \]
   Finally, the remaining terms in $Q_{N,T}$ are likewise bounded (from above and below) by a constant depending on $T$ and $N$.
\end{proof}

Tameness now follows easily.

\begin{corollary} \label{cor:tameness_A}
    Fix $q\ge 1$. For all test functions $h\in C_{{\rm test}}$ with sufficiently small $c>0$, the functional $q \cH_{N,T}^h(A_T^0,\vartheta)$ is tame with respect to $\bP^\rmA$, that is, it satisfies
    \[
    \E^\rmA[|q \cH_{N,T}^h(A_T^0,\vartheta)|^{r_1}] + \E^\rmA[e^{-r_2 q \cH_{N,T}^h(A_T^0,\vartheta)}] < \infty
    \]
    for some $r_1,r_2\ge 1$ such that $\frac{1}{r_1}+\frac{1}{r_2}=1$.
\end{corollary}
\begin{proof}
    For the exponential bound, it is enough to consider the quadratic norm stemming from the test function $h$ thanks to Lemma \ref{lem:tameness_lower_bound}. But $\exp(\alpha \triple{\bA_T^0}_{-\kappa}^2)$ is integrable for all sufficiently small $\alpha>0$ by Fernique's theorem, and hence \[
    \E^\rmA[e^{-r_2 q \cH_{N,T}^h(A_T^0,\vartheta)}] < \infty,
    \]
    for a suitably restricted $c$ (essentially depending only on $q$ and analytic constants, as we may choose $r_2$ arbitrarily close to one). For the $L^2(\bP^\rmA)$ bound, we rely on the estimate from Lemma~\ref{lem:BD_A_bounds}, the existence of all moments for Gaussians and Lemma~\ref{lem:bV_properties} in relation to the vacuum polarisation.
\end{proof}

This completes our preparations for the following result.

\begin{proposition} \label{prop:unif_integrability}
    Given a fixed $q\ge 1$, there exists a $\tilde c(q)>0$ such that, for any $h\in C_{{\rm test}}$ with $c<\tilde c$, we have the uniform bounds
    \[
    0 < \inf_{N,T} \cZ_{N,T}(h,q) \le \sup_{N,T} \cZ_{N,T}(h,q) < \infty.
    \]
\end{proposition}
\begin{proof}
    We disintegrate the partition function as follows:
    \[
    \cZ_{N,T}(h,q) = \int_{\bfi (0,2\pi]^2}\;\exp\Big((-)- \log\E^\rmA\big[e^{-q\cH_{N,T}^h(A_T^0,\vartheta)} \big]\Big)\,\bar \diff\vartheta.
    \]
    First, choose $c$ small enough for Lemma~\ref{lem:BD_A_bounds} and Corollary~\ref{cor:tameness_A} to hold. We thereby obtain the representation
    \[
    -\log \E^\rmA[e^{-q\cH_{N,T}^h(A_T^0,\vartheta)}]=\inf_{v\in\bH^\rmA}\E^\rmA\Big[q\cH_{N,T}^h(A_T^0+J_T(v),\vartheta)+ \frac 1 2\|v\|_{\bH^\rmA}^2\Big]
    \]
    for the expression inside the exponential above from the Bou\'e-Dupuis formula. Lemma \ref{lem:BD_A_bounds}, for any $\eps>0$, implies the lower bound
    \begin{align*}
        \E^\rmA\Big[q\cH_{N,T}^h(A_T^0+J_T(v),\vartheta) &+ \frac 1 2\|v\|_{\bH^\rmA}^2\Big] \ge -\E^\rmA\big[\mathscr{Q}_T\big] + (\tfrac 1 2-4qC_E c-\eps)\E^\rmA\big[\|v\|_{\bH^\rmA}^2\big].
    \end{align*}
    We now impose the further requirement $c < (8q C_E)^{-1}$ (which, together with the previous constraints, defines $\tilde c(q)$) and choose $\eps< \tfrac 1 2-4qC_Ec$ to infer that
    \[
    \inf_{v\in\bH^\rmA}\E^\rmA\Big[q\cH_{N,T}^h(A_T^0+J_T(v),\vartheta) + \frac 1 2\|v\|_{\bH^\rmA}^2\Big] \ge -C(h,q,\eps),
    \]
    uniformly in $T$ and $N$. Taking $v=0$ in the upper bound of Lemma~\ref{lem:BD_A_bounds} gives a uniform upper bound for the variational functional. Since both bounds are uniform in $\vartheta$, they remain valid after the final expectation over $\vartheta$, completing the proof.
\end{proof}

Having thus established that the random variables $\exp (-\cH_{N,T}^h(A_T^0,\vartheta))$ are uniformly integrable, we complete our argument by showing that they also converge in a suitable sense as $N,T\to \infty$. We break this statement into the following claims.

\begin{lemma} \label{lem:A_convergence}
    \begin{enumerate}
        \item   For a fixed $T\in [0,\infty]$, the limit
                \[
                \cH_{\infty,T}^h(A_T^0,\vartheta) \coloneqq \lim_{N\to\infty} \cH_{N,T}^h(A_T^0,\vartheta)
                \]
                exists $\bP^\rmA$-almost surely.\\
        \item   For a fixed $N\in\N\cup\{\infty\}$, the random variables
                \[
                \cH_{N,T}^h(A_T^0,\vartheta)
                \]
                converge in probability as $T\to\infty$ on $(\Omega^\rmA,\mathcal F^\rmA,\bP^\rmA)$, uniformly in $N$ and $\vartheta$.
    \end{enumerate}
\end{lemma}
\begin{proof}
    Assume first that $T$ and $N$ are finite. The existence of the $N\to\infty$ limit of $Q_{N,T}(A_T)=\cQ_N(\bfA_T)$ follows from Theorem~\ref{thm:convergence_cQ_N} via Corollary~\ref{cor:bV_well_defined}. The limit of $\cW_N^{h,1,m_{\bA_T,0}}(\bA_T,0)$ is a direct consequence of Theorem~\ref{thm:higgs_convergence}.

    For the second claim, Theorem~\ref{thm:higgs_convergence} states that $\cW_N^{h,1,m_{\bA_T,0}}(\bA_T,0)$ is a locally Lipschitz continuous function of $\bA_T^0$ in the $\bfnorm{\cdot;\cdot}_{-\kappa}$ topology, uniformly in $N$. Lemma~\ref{lem:Rem_det} implies that $\bfA_T^0\mapsto\cQ_N(\bfA_T)$ is also continuous in the $\bfnorm{\cdot;\cdot}_{-\kappa}$ topology, again uniformly in $N$. Lemma~\ref{lem:bV_properties} confirms that $\bfA_T^0$ converges in probability as $T\to\infty$, and $\cH_{N,T}^h(A_T^0,\vartheta)$ thus also converges in probability. Since $\bfnorm{\bfA_T-\bar\bfA_T}_{-\kappa} \lesssim \bfnorm{\bfA_T^0-\bar\bfA_T^0}_{-\kappa}$, the convergence is uniform also in $\vartheta$.

    To conclude, observe that we have actually shown that $\cH_{N,T}^h(A_T^0,\vartheta)$, as a random variable in $(\Omega^\rmA,\mathcal h^\rmA,\bP^\rmA)$, is an equicontinuous (in $N$ and $\vartheta$) function of $\bfA_T^0$. The pointwise limit of that function as $N\to\infty$ thus also implies the convergence (in probability) of $\cH_{\infty,T}^h(A_T^0,\vartheta)$. By the same pointwise convergence, that limit must be the same as the limit obtained by first taking $T\to\infty$, followed by $N\to\infty$.
\end{proof}

\begin{proposition}\label{prop:continuity}
    Let the assumptions of Proposition \ref{prop:unif_integrability} hold, and let $\bA_T^0$ and $\cH_{N,T}^h(A_T^0,\vartheta)$ be as defined in Section~\ref{sec:gauge_variational_setup}. As $T\to\infty$, the sequence of random variables
    \[
    e^{-\cH_{N,T}^h(A_T^0,\vartheta)}
    \]
    converges in distribution on the probability space $(\bfi (0,2\pi]^2,\bar\diff\vartheta)\times(\Omega^\rmA,\mathcal h^\rmA,\bP^\rmA)$ for any $N\in\N\cup\{\infty\}$.
\end{proposition}
\begin{proof}
    Since the random zero mode $\vartheta$ is constant in $T$, the proposition is a direct consequence of Lemma~\ref{lem:A_convergence} and the continuity of the (real) exponential function.
\end{proof}

Propositions \ref{prop:unif_integrability} and \ref{prop:continuity} together now imply our main theorem.

\begin{proof} [Proof of Theorem~\ref{thm:main}]
    Choose $q>1$ arbitrarily close to one, fixing the admissible class $C_{{\rm test}}$ with the range $c \in (0,\tilde c(q))$ depending on $q$ through Proposition~\ref{prop:unif_integrability}. For any test function in this class, Proposition~\ref{prop:unif_integrability} gives uniform integrability of the family $e^{-\cH_{N,T}^h(A_T^0,\vartheta)}$. By Proposition~\ref{prop:continuity}, this family converges in law as $T\to\infty$ for any $N\in\N\cup\{\infty\}$, uniformly in $N$ by Lemma~\ref{lem:A_convergence}. The expectations therefore converge uniformly in $N$:
    \[
    \lim_{T\to\infty}\E^\Theta\E^\rmA\Big[e^{-\cH_{N,T}^h(A_T^0,\vartheta)}\Big] = \E^\Theta\E^\rmA\Big[\lim_{T\to\infty} e^{-\cH_{N,T}^h(A_T^0,\vartheta)}\Big].
    \]
    The bounds $0<\cZ(h)<\infty$ are a direct consequence of Proposition~\ref{prop:unif_integrability}.
\end{proof}

\subsection{Characterisation of the measure and absolute continuity} \label{sec:densities}

Thanks to our method of construction, we obtain explicit densities for $\mu^\YMH$, its gauge marginal (Higgs-annealed) and its conditional Higgs measure ($A$-quenched). Theorem~\ref{thm:higgs-W-bound} and Proposition~\ref{prop:unif_integrability} give uniform integrability of the families
\[
e^{-Q_{N,T}(A_T) -\cW_N^{0,1,m_{\bA_T,0}}(\bA_T,0)}
\qquad\text{and}\qquad
e^{-\scV_N^{m_{\bA_T,0}}(\phi_{\bA_T,0,N})}
\]
for all $N,T$, including $N=\infty$ and $T=\infty$. Continuity in the appropriate topologies (Theorem~\ref{thm:higgs_convergence} and Lemma~\ref{lem:Rem_det}) then lets us pass limits under expectations to obtain the densities below.

Let $\bfA_\infty = \lim_{T\to\infty} (A_T,A_T^* A_T-\rmc_T,\bV_{A_T})$ be the limit in probability of the enhanced field on $(\bfi (0,2\pi]^2,\bar\diff\vartheta)\times(\Omega^\rmA,\mathcal F^\rmA,\bP^\rmA)$, and let $\mu^{A}_\infty$ be the law of $A=\lim_{T\to\infty} A_T$. Lemma~\ref{lem:A_convergence} gives
\[
e^{-Q_{N,T}(A_T) -\cW_N^{0,1,m_{\bA_T,0}}(\bA_T,0)} \xrightarrow{N,T\to\infty} e^{-\cQ_\infty(\bfA_\infty)-\cW_\infty^{0,1,m_{\bfA_\infty,0}}(\bfA_\infty,0)}
\]
in probability. The gauge marginal of $\mu^\YMH$ is therefore
\begin{equation}\label{eq:gauge_marginal_density}
\dif\mu_A^\YMH(A) \sim e^{-\cQ_\infty(\bfA_\infty)-\cW_\infty^{0,1,m_{\bfA_\infty}}(\bfA_\infty,0)} \dif\mu^{\rmA}_\infty(A),
\end{equation}
where the density $\dif\mu_A^\YMH/\dif\mu^{A}_\infty$ is measurable and strictly positive as a function of $A$, and continuous as a function of $\bfA_\infty$.

Moving on to the conditional Higgs field, let $\boldsymbol{\Phi}_{\infty}=\lim_{N,T\to\infty}(\phi_{\bA,0,N},\wick{\phi_{\bA,0,N}^i(\phi_{\bA,0,N}^*)^j}_N:0\le i+j\le 2n)$ and write $\mu_{\bA,\infty,\infty}^\GFF$ for the law of $\phi = \lim_{N,T\to\infty} \phi_{\bA,0,N}$. The proof of Theorem~\ref{thm:higgs_convergence} gives
\[
e^{-\scV_N^{m_{\bA_T,0}}(\phi_{\bA_T,0,N})} \xrightarrow{N,T\to\infty} e^{-\scV(\bA,\boldsymbol{\Phi_{\infty}})}
\]
in probability, defining $\scV(\bA,\boldsymbol{\Phi}_{\infty})$. The conditional Higgs measure is therefore
\[
\dif\mu_{\bA,\infty,\infty}^\Phi(\phi) \sim e^{-\scV(\bA,\boldsymbol{\Phi}_{\infty})} \dif\mu_{\bA,\infty,\infty}^\GFF(\phi).
\]
The density $\dif\mu_{\bA,\infty,\infty}^\Phi/\dif\mu_{\bA,\infty,\infty}^\GFF$ is continuous in its explicit variables and measurable in $\phi$.

Combining the above finally yields the representation
\[
\dif\mu^\YMH(A,\phi) \sim e^{-\cQ_\infty(\bfA_\infty)-\scV(\bA,\boldsymbol{\Phi}_\infty)}\dif\mu_{\bA,\infty,\infty}^\GFF(\phi) \dif\mu^{A}_\infty(A)
\]
of the full $2$d Abelian Yang-Mills-Higgs measure. In heuristic terms, the Boltzmann weight appearing above can be interpreted in the following way,
\[
\dif\mu^\YMH(A,\phi) \sim \underbrace{e^{-\scV(\bA,\boldsymbol{\Phi}_\infty)}}_{\frac{\dif\mu_{\bA,\infty,\infty}^\Phi}{\dif\mu_{\bA,\infty,\infty}^\GFF}(A,\phi)}\dif\mu_{\bA,\infty,\infty}^\GFF(\phi) \; \underbrace{e^{-\cQ_\infty(\bfA_\infty)}}_{\mathclap{e^{\frac{1}{2}\langle \phi;(L_0-L_A)\phi\rangle} \frac{\dif\mu_{0,\infty,\infty}^\GFF}{\dif\mu_{\bA,\infty,\infty}^\GFF}(A,\phi)}} \; \dif\mu^{A}_\infty(A),
\]
which recovers the manipulations at the level of the formal path integral. Of course, we stress that the Radon-Nikodym derivative $\dif\mu_{0,\infty,\infty}^\GFF /\dif\mu_{\bA,\infty,\infty}^\GFF$ does not actually exist.

\subsection{Gauge covariance property}\label{sec:gauge_covariance}

Here we prove the statements in Proposition~\ref{prop:gauge_covariance} and Theorem~\ref{thm:linear_covariant_gauges}. Recall to a gauge transformation $g\in C(\T^2;U(1))$ we had a smooth approximation $g_T\to g$ in the supremum norm. The main ingredient is
\begin{lemma}\label{lem:gauge_invariance_Q_V}
    We have for smooth $A$ and $g$
    \[
    \cQ_{N,T}(A^{g})=\cQ_{N,T}(A), \qquad \cV_{A^{g},N,T}(\Cdot)=\cV_{A,N,T}(\Cdot).
    \]
\end{lemma}

\begin{proof}
    The statement regarding $\cV$ is trivial as the dependence in $A$ is through $\lambda_{A,N,T}$ which is gauge-invariant for being defined through the spectrum of the operator $L_{A}$ and $L_{A^g}=gL_Ag^{-1}$.

    Regarding $\cQ$, we first treat the integral in its definition. We write
\begin{align*}
\int_0^T &\Tr \Big((L_0+1+\lambda)^{-1}-(L_{A^g}-\rmc_T+\lambda_{A^g,T}+\lambda)^{-1}) \Big)\dif \lambda \\
&=
\int_0^T \lim_{\eps\to 0^+}\Tr \Big((L_0+1+\lambda)^{-1-\eps}-(L_{A^g}-\rmc_T+\lambda_{A^g,T}+\lambda)^{-1-\eps}) \Big)\dif \lambda \\
&=
\int_0^T \lim_{\eps\to 0^+}\left\{\Tr \big((L_0+1+\lambda)^{-1-\eps}\big)-\Tr\big(g(L_{A}-\rmc_T+\lambda_{A,T}+\lambda)^{-1-\eps} g^{-1}\big)\right\}\dif \lambda \\
&=
\int_0^T \lim_{\eps\to 0^+}\left\{\Tr \big((L_0+1+\lambda)^{-1-\eps}\big)-\Tr\big((L_{A}-\rmc_T+\lambda_{A,T}+\lambda)^{-1-\eps} \big)\right\}\dif \lambda \\
&=
\int_0^T \lim_{\eps\to 0^+}\Tr \Big((L_0+1+\lambda)^{-1-\eps}-(L_{A}-\rmc_T+\lambda_{A,T}+\lambda)^{-1-\eps} \Big)\dif \lambda,
\end{align*}
where the first equality is by dominated convergence, the second one by $L_{A^g}=gL_Ag^{-1}$, the third one by cyclicity, and the last one by linearity.
Afterwards, we apply dominated convergence to get the gauge-invariance of the integral.
Moreover, the other term in $\cQ_{N,T}$ is defined through $\lambda_{A,N,T}$ which is also gauge-invariant and therefore proving that $\cQ_{N,T}$ is gauge-invariant.
\end{proof}

Now we are ready to prove Proposition~\ref{prop:gauge_covariance}.
\begin{proof}[Proof of Proposition~\ref{prop:gauge_covariance}] We have by the definition of $\mu^{\rmA[g]}_T$ and Lemma~\ref{lem:gauge_invariance_Q_V}
\begin{align*}
\int F(A,\phi) \dif\mu_{N,T}^{\YMH[g]}(A,\phi)&=\frac{1}{\cZ_{N,T}^{[g]}}\int F(A^{g_T},\phi)
    e^{-\cQ_{N,T}(A^{g_T})-\cV_{A^{g_T},N,T}(\phi)}
    \dif\mu^{\GFF}_{A^{g_T},N,T}(\phi)\dif\mu_T^{\rmA}(A)\\
    &=\frac{1}{\cZ_{N,T}^{[g]}}\int F(A^{g_T},\phi)
    e^{-\cQ_{N,T}(A)-\cV_{A,N,T}(\phi)}
    \dif\mu^{\GFF}_{A^{g_T},N,T}(\phi)\dif\mu_T^{\rmA}(A).
\end{align*}
Furthermore, recalling~\ref{eq:intro_covariant_GFF},  $\mu^{\GFF}_{A^{g_T},N,T}$ is the law of
    \begin{equation*}
    \phi_{A^{g_T},N,T}
    =\int_0^N
    \bigl(L_{A^{g_T}}-\rmc_T+\lambda_{A^{g_T},T}+\lambda\bigr)^{-1}
    \dif X^{\Phi}_\lambda,
\end{equation*}
and $\lambda_{A,T}$ is gauge-invariant can be written as
    \begin{equation*}
    \phi_{A^{g_T},N,T}
    =g_T\int_0^N
    \bigl(L_{A}-\rmc_T+\lambda_{A,T}+\lambda\bigr)^{-1}
    g_T^{-1}\dif X^{\Phi}_\lambda.
\end{equation*}
Since $g_T^{-1}$ is unitary we have that $g_T^{-1} X_\Lambda^\Phi$ has the same law as $X_\lambda^\Phi$ and as such $\mu^{\GFF}_{A^{g_T},N,T}$ is the law of  $g_T\phi_{A,N,T}$.

This allow us to write
\begin{align*}
\int F(A,\phi) \dif\mu_{N,T}^{\YMH[g]}(A,\phi)&=\frac{1}{\cZ_{N,T}^{[g]}}\int F(A^{g_T},g_T\phi)
    e^{-\cQ_{N,T}(A)-\cV_{A,N,T}(g_T\phi)}
    \dif\mu^{\GFF}_{A,N,T}(\phi)\dif\mu_T^{\rmA}(A)\\
  &=  \frac{1}{\cZ_{N,T}^{[g]}}\int F(A^{g_T},g_T\phi)
    e^{-\cQ_{N,T}(A)-\cV_{A,N,T}(\phi)}
    \dif\mu^{\GFF}_{A,N,T}(\phi)\dif\mu_T^{\rmA}(A),
\end{align*}
where we have used that $\cV_{A,N,T}(g_T\phi)=\cV_{A,N,T}(\phi)$ due to the fact that it is a polynomial in $|\phi|$ and $g_T$ is unitary.  Hence,
\begin{align}\label{eq:gauge_covariance_cutoff}
\int F(A,\phi) \dif\mu_{N,T}^{\YMH[g]}(A,\phi)&=\int F(A^{g_T},g_T\phi) \dif\mu_{N,T}^{\YMH}(A,\phi),
\end{align}
and taking limits yields the result.
\end{proof}

Now we prove Theorem~\ref{thm:linear_covariant_gauges}.
\begin{proof}[Proof of Theorem~\ref{thm:linear_covariant_gauges}]
    By the proof of Proposition~\ref{prop:gauge_covariance}, we have found
    \begin{equation}\label{eq:linear_covariant_gauges_cutoff}
\int F(A,\phi)\dif \mu^{\YMH(\alpha)}_{N,T}(A,\phi)=\int F(A^{\scg_T^{\alpha^{-1/2}\zeta}},\scg_T^{\alpha^{-1/2}\zeta}\phi)\dif \mu^{\YMH}_{N,T}(A,\phi)  \dif\nu(\zeta),
\end{equation}
so that taking limits yields the desired statement. Finally, the limit $\alpha\to \infty$ follows from the fact that $g^{\alpha^{-1/2}\zeta}\to 1$ in $C^{1-\delta}$ for any $\delta>0$ and hereby we obtain the result for $\alpha\to\infty$.
\end{proof}

\section{Large deviations}\label{sec:large_deviations}

Now that we have completed the construction of the Abelian Yang--Mills--Higgs measure $\mu^\YMH$, we show how the tools in our analysis allow us to obtain large deviations for the measure. The first step is verifying that, for any $\eps > 0$, our arguments also give the construction of $\mu^{\YMH;\eps}$ which is formally defined as the Yang--Mills--Higgs measure with action $\eps^{-1}\cS_\YMH(A,\phi)$ (see \eqref{eq:intro_action}). More concretely, for each fixed $\eps>0$, we construct
\begin{align}\label{eq:def_mu_YMH_eps}
\mu^{\YMH;\eps}
=\lim_{T\to\infty}\lim_{N\to\infty}
  \mu_{N,T}^{\YMH;\eps},
\end{align}
with convergence understood as in Theorem~\ref{thm:main} and $ \mu_{N,T}^{\YMH;\eps}$ defined by setting
\[
 \dif\mu^{\YMH;\eps}_{N,T}(A,\phi)=
 \frac 1 {\cZ_{N,T}^\eps}\exp \big[ -Q_{N,T}^\eps(A)-\eps^{-1}\cV_{A,N,T}^\eps(\phi)\big]\dif\mu_{A,N,T}^{\GFF;\eps}(\phi)\dif\mu_T^{\rmA;\eps}(A).
\]
We now review the definitions of the objects on the right hand side above. Let $\mu^{\rmA;\eps}_T$ denote the law of the gauge field
\[
A_T^\eps:=\vartheta+\sqrt\eps\,\rmP_T^\rmA\diff^*(\diff\diff^*)^{-1}\int_0^1\dif X_t^\rmA\;,
\]
where we recall that $\vartheta$ has normalized Lebesgue law $\bar\dif\vartheta$ on $\bfi(0,2\pi]^2$ and is independent of $X^\rmA$. Let $\mu_{A,N,T}^{\GFF;\eps}$ denote the law of the Higgs field
\[
\phi_{A,N,T}^\eps:=\sqrt\eps\,\int_0^N (L_A-\eps\rmc_T+\lambda_{A,T}^\eps+\lambda)^{-1}\dif X_\lambda^\Phi,
\]
where
\[
\lambda_{A,T}^\eps:=1-\left(0\wedge \min\sigma(L_A-\eps\rmc_T)\right).
\]
Recall the polynomial $\mathcal V$ from~\eqref{eq:def_V} and the counterterm functions $\rmv_l$ from~\eqref{eq:def_counterterm_family}. The Higgs potential functional $\cV_{A,N,T}^\eps$ is given by
\[
\cV_{A,N,T}^\eps(\phi)
:=
\int_{\T^2}
\Bigl(
\mathcal V(|\phi|^2)
-\sum_{l=0}^{n-1}\rmv_l(\eps\rmp_N)|\phi|^{2l}
\Bigr)
-\frac{\lambda_{A,T}^\eps}{2}
\int_{\T^2}\bigl(|\phi|^2-\eps\rmp_N\bigr).
\]
We also define the renormalised logarithmic determinant ratio
\[
Q_{N,T}^\eps(A):=-\frac{\lambda_{A,T}^\eps-1}{2}\rmp_N+\eps\rmo_{N,T}+\int_0^N \Tr \Big((L_0+1+\lambda)^{-1}-(L_A-\eps\rmc_T+\lambda_{A,T}^\eps+\lambda)^{-1} \Big)\dif \lambda.
\]
The scalar mass terms in $Q_{N,T}^\eps+\eps^{-1}\cV_{A,N,T}^\eps$ again sum to $\rmp_N/2$, independent of both fields and absorbed into the normalisation.
We now relate this definition to the enhanced functional $\cQ_N$ defined in~\eqref{eq:def_cQ_N}. Recall that $A_T=\vartheta+A_T^0$, so that $A_T^\eps=\vartheta+\sqrt\eps A_T^0$. Define the enhanced pair and triple by
\[
\bA_T^\eps
:=\bigl(A_T^\eps,(A_T^\eps)^*A_T^\eps-\eps\rmc_T\bigr),
\qquad
\bfA_T^\eps:=(\bA_T^\eps,\bV_{A_T}^\eps),
\]
where
\[
\bV_{A_T}^\eps(\lambda)
:=
\boldsymbol{\Pi}^\lambda(A_T^\eps,A_T^\eps)
-\eps\,\E^\Theta\E^\rmA
  [\boldsymbol{\Pi}^\lambda(A_T,A_T)].
\]
Here the expectation in the subtraction is evaluated at the unscaled field $A_T$. This convention matches the counterterm $\eps\rmo_{N,T}$, with $\rmo_{N,T}$ defined in Section~\ref{sec:construction_measure}.

These definitions give
\[
L_{\bA_T^\eps,0}=L_{A_T^\eps}-\eps\rmc_T,
\qquad
m_{\bA_T^\eps,0}=\lambda_{A_T^\eps,T}^\eps,
\]
and, by the definition of $\cQ_N$,
\[
Q_{N,T}^\eps(A_T^\eps)=\cQ_N(\bfA_T^\eps,0).
\]
For each fixed $\eps>0$, we repeat the convergence and uniform-integrability arguments in the proof of Theorem~\ref{thm:main}, using the scaled fields and counterterms above. Under $\phi=\sqrt\eps\,\psi$, the interaction $\eps^{-1}\cV_{A,N,T}^\eps(\phi)$ is a Wick-ordered polynomial of degree $2n$ in $\psi$, with positive leading coefficient $c_n\eps^{n-1}$. Thus the construction estimates apply for each fixed $\eps>0$, with constants allowed to depend on $\eps$. This gives finite, strictly positive limits of the partition functions with observables, justifying \eqref{eq:def_mu_YMH_eps}.

Throughout this section, observables are real-valued functions on $\Omega^1 C^{-\kappa/2}(\T^2)\times H^{-\kappa}(\T^2)$, for a fixed sufficiently small $\kappa>0$. Lipschitz continuity is understood with respect to the usual product norm.

For a bounded continuous observable $f$, we then define
\[
D^{f,\eps}:=-\eps\log \E_{\mu^{\YMH;\eps}}[e^{-\eps^{-1}f(A,\phi)}].
\]
We analyze large deviations by studying the limit $\lim_{\eps \downarrow 0} D^{f,\eps}$. As in the construction of the measure $\mu^{\YMH;\eps}$, we express the partition functions with source terms through iterated Bou\'e--Dupuis variational formulas, conditioning on the harmonic component $\vartheta$. We then obtain matching upper and lower bounds as $\eps\downarrow0$, and therefore identify a limiting \emph{deterministic} variational problem. Applying these bounds both with and without the observable gives the limit of the normalized Laplace functional $D^{f,\eps}$.

For $\vartheta\in\bfi[0,2\pi]^2$, let
\[
\mathcal A_\vartheta
:=\Big\{A\in\Omega^1_\Coul H^1:
  |\T^2|^{-1}\int_{\T^2}A=\vartheta\Big\}
\]
denote the space of $H^1$ Coulomb-gauge fields with harmonic component $\vartheta$.

Recall the action~\eqref{eq:intro_action}:
\[
\cS_{\YMH}(A,\phi)
=\frac12\|\diff A\|_{L^2}^2
+\frac12\|\diff_A\phi\|_{L^2}^2
+\int_{\T^2}\mathcal V(|\phi|^2).
\]

For an enhanced gauge field $\bfA=(\bA,\bV)$, write $A$ for the first component of $\bA$ and set
\[
\phi_{\bA,N}^\eps
:=
\sqrt\eps\int_0^N
(L_{\bA,0}+m_{\bA,0}+\lambda)^{-1}\dif X_\lambda^\Phi,
\qquad
\phi_\bA^\eps:=\lim_{N\to\infty}\phi_{\bA,N}^\eps.
\]
Define the regularised potential by
\[
\cV_{\bA,N}^\eps(\phi)
:=
\int_{\T^2}
\Bigl(
\mathcal V(|\phi|^2)
-\sum_{l=0}^{n-1}\rmv_l(\eps\rmp_N)|\phi|^{2l}
\Bigr)
-\frac{m_{\bA,0}}{2}
\int_{\T^2}\bigl(|\phi|^2-\eps\rmp_N\bigr).
\]
For $\bA=\bA_T^\eps$, this agrees with $\cV_{A_T^\eps,N,T}^\eps$. Define the conditional Higgs free energy by
\[
\cW^{f;\eps}(\bA)
:=
\lim_{N\to\infty}
\left(
-\eps\log\E^\Phi\Big[
\exp\Big(
-\eps^{-1}\bigl(
f(A,\phi_{\bA,N}^\eps)
+\cV_{\bA,N}^\eps(\phi_{\bA,N}^\eps)
\bigr)
\Big)
\Big]
\right).
\]
Set $\cH^{f;\eps}(\bfA):=\eps\cQ_\infty(\bfA,0)+\cW^{f;\eps}(\bA)$ and define
\begin{equation}\label{eq:normalized_laplace_functional}
\cZ^\eps(f):=\E^\Theta\E^\rmA\Big[
  \exp\bigl(-\eps^{-1}\cH^{f;\eps}(\bfA^\eps)\bigr)
\Big],
\qquad
D^{f,\eps}=-\eps\log\frac{\cZ^\eps(f)}{\cZ^\eps(0)}.
\end{equation}
Here $\cZ^\eps(0)$ is the normalizing constant. For fixed $\vartheta$, with its occurrence in $\bfA^\eps$ kept implicit, let
\[
\mathcal F^{f;\eps}(\vartheta)
:=-\eps\log\E^\rmA\Big[
  \exp\bigl(-\eps^{-1}\cH^{f;\eps}(\bfA^\eps)\bigr)
\Big].
\]
Then $\cZ^\eps(f)=\E^\Theta\Big[ \exp\bigl(-\eps^{-1}\mathcal F^{f;\eps}(\vartheta)\bigr)\Big]$.

Write
\[
I^\bA(u):=\lim_{N\to\infty}I_N^{\bA,0}(u),
\qquad
J(v):=\diff^*(\diff\diff^*)^{-1}\int_0^1 v_t\,\dif t.
\]
For $u\in\bH_a^\Phi$, define the controlled Higgs functional
\[
\cG^{f;\eps}(\bA,u)
:=\E^\Phi\Big[
  f(A,\phi^\eps_\bA+I^\bA(u))
  +\cV^\eps_\bA(\phi^\eps_\bA+I^\bA(u))
  +\frac12\|u\|_{\bH^\Phi}^2
\Big],
\]
where $\cV^\eps_\bA$ denotes the renormalised limit of $\cV_{\bA,N}^\eps(\phi_{\bA,N}^\eps+I_N^{\bA,0}(u))$ as $N\to\infty$. Recall that $\bplus$ denotes the translation of the enhanced gauge field data. Applying the Bou\'e--Dupuis formula for both the Higgs and gauge fields gives
\begin{align*}
\cW^{f;\eps}(\bA)
&=\inf_{u\in\bH_a^\Phi}\cG^{f;\eps}(\bA,u),\\
\mathcal F^{f;\eps}(\vartheta)
&=\inf_{v\in\bH_a^\rmA}\E^\rmA\Big[
  \eps\cQ_\infty(\bfA^\eps\bplus J(v),0)
  +\inf_{u\in\bH_a^\Phi}\cG^{f;\eps}(\bA^\eps\bplus J(v),u)
  +\frac12\|v\|_{\bH^\rmA}^2
\Big].
\end{align*}
We now study this variational problem as $\eps\downarrow0$. Our aim is to show that the determinant term does not affect the limiting variational value, and to identify the limit through a variational problem over deterministic controls with $\vartheta$ fixed. In particular, we will prove convergence of $\mathcal F^{f;\eps}(\vartheta)$ to a continuous function $\mathcal F^{f;0}(\vartheta)$, uniformly in $\vartheta\in\bfi[0,2\pi]^2$. We control the limit $\eps\downarrow0$ of $D^{f,\eps}$ in~\eqref{eq:normalized_laplace_functional} by applying a Laplace principle to the numerator and denominator.

To begin the deterministic comparison, let $\bH^\rmA_0\subset\bH^\rmA$ and $\bH^\Phi_0\subset\bH^\Phi$ denote the deterministic gauge and Higgs controls. Set
\[
\cV_\bA^0(\phi)
:=\int_{\T^2}\mathcal V(|\phi|^2)
-\frac{m_{\bA,0}}{2}\int_{\T^2}|\phi|^2,
\]
and define
\[
\cG^{f;0}(\bA,u)
:=\E^\Phi\Big[
  f(A,I^\bA(u))
  +\cV_\bA^0(I^\bA(u))
  +\frac12\|u\|_{\bH^\Phi}^2
\Big].
\]
\begin{lemma}\label{lem:deterministic_variational_problem}
For every $\vartheta\in\bfi[0,2\pi]^2$, one has
\begin{align*}
\mathcal F^{f;0}(\vartheta)
&:=\inf_{\substack{v\in\bH^\rmA_0\\u\in\bH^\Phi_0}}
\Big\{
  \cG^{f;0}\bigl(\0\bplus(\vartheta+J(v)),u\bigr)
  +\frac12\|v\|_{\bH^\rmA}^2
\Big\}\\
&=\inf_{\substack{A\in\mathcal A_\vartheta\\\phi\in H^1}}
\Big\{
  f(A,\phi)+\cS_{\YMH}(A,\phi)
\Big\}.
\end{align*}
Moreover, if $f$ is bounded and Lipschitz, then $\vartheta\mapsto\mathcal F^{f;0}(\vartheta)$ is bounded and continuous on $\bfi[0,2\pi]^2$.
\end{lemma}

\begin{proof}
We rewrite the infimum over controls as an infimum over the fields $A=\vartheta+J(v)$ and $\phi=I^{\0\bplus A}(u)$. For fixed $(A,\phi)\in\mathcal A_\vartheta\times H^1$, the observable and potential are fixed, so we only need to minimize the control energy. Set $\bA=\0\bplus A$. Then $L_{\bA,0}=L_A$ and $m_{\bA,0}=1$, and the minimum control energy is
\[
\inf_{\substack{J(v)=A-\vartheta\\I^\bA(u)=\phi}}
\frac12\bigl(\|v\|_{\bH^\rmA}^2+\|u\|_{\bH^\Phi}^2\bigr)
=\frac12\bigl(\|\diff A\|_{L^2}^2+\|\diff_A\phi\|_{L^2}^2+\|\phi\|_{L^2}^2\bigr).
\]
Cauchy--Schwarz gives the lower bound. The Hodge decomposition and spectral calculus show that it is attained by
\[
v_t^*=\mathbf1_{[0,1]}(t)\,\diff A,
\qquad
u_t^*=(L_A+1)(L_A+1+t)^{-1}\phi.
\]
The term $-\frac12\|\phi\|_{L^2}^2$ in $\cV_\bA^0(\phi)$ cancels the last term above, so the minimum cost at fixed $(A,\phi)$ is $f(A,\phi)+\cS_{\YMH}(A,\phi)$. Every pair in $\mathcal A_\vartheta\times H^1$ is produced by these controls, so taking the infimum over the pairs proves the identity.

Since $f$ is bounded and $\mathcal V$ is bounded below, comparison with $(A,\phi)=(\vartheta,0)$ gives
\[
-C\leq\mathcal F^{f;0}(\vartheta)
\leq f(\vartheta,0)+|\T^2|\mathcal V(0)
\leq C,
\]
with $C$ independent of $\vartheta$.

For continuity, fix $0<\eta\leq1$ and choose $(A,\phi)\in\mathcal A_\vartheta\times H^1$ such that
\[
f(A,\phi)+\cS_{\YMH}(A,\phi)
\leq\mathcal F^{f;0}(\vartheta)+\eta.
\]
Since $\mathcal V(r)\ge r-C$ for $r\ge0$, the preceding upper bound and boundedness of $f$ give $\|\diff_A\phi\|_{L^2}^2+\|\phi\|_{L^2}^2\leq C$, independently of $\vartheta$ and $\eta$. Set $h=\vartheta'-\vartheta$. The pair $(A+h,\phi)$ belongs to $\mathcal A_{\vartheta'}\times H^1$, and $\diff_{A+h}\phi=\diff_A\phi+h\phi$. Comparing the two costs and using that $f$ is Lipschitz gives
\[
\mathcal F^{f;0}(\vartheta')
\leq\mathcal F^{f;0}(\vartheta)+\eta+C(|h|+|h|^2).
\]
Letting $\eta\downarrow0$ and interchanging $\vartheta$ and $\vartheta'$ proves continuity.
\end{proof}
Restricting both variational problems to deterministic controls, set
\begin{align*}
\mathcal F_0^{f;\eps}(\vartheta)
&:=\inf_{v\in\bH^\rmA_0}\E^\rmA\Big[
  \eps\cQ_\infty(\bfA^\eps\bplus J(v),0)
  +\inf_{u\in\bH^\Phi_0}
  \cG^{f;\eps}(\bA^\eps\bplus J(v),u)
  +\frac12\|v\|_{\bH^\rmA}^2
\Big].
\end{align*}

\begin{theorem}\label{thm:easy_theorem}
Let $f$ be a bounded Lipschitz observable. Then, for every $\eps>0$ and $\vartheta\in\bfi[0,2\pi]^2$,
\[
\mathcal F^{f;\eps}(\vartheta)
\leq \mathcal F_0^{f;\eps}(\vartheta)
\qquad\text{and}\qquad
\limsup_{\eps\downarrow0}
\sup_{\vartheta\in\bfi[0,2\pi]^2}
\Big(
\mathcal F_0^{f;\eps}(\vartheta)
-\mathcal F^{f;0}(\vartheta)
\Big)
\leq 0.
\]
\end{theorem}

\begin{proof}
The first inequality follows by restricting both infima to deterministic controls.

For the second inequality, we first note that the pair $\bA=\0\bplus(\vartheta+J(v))$ has $m_{\bA,0}=1$. Since $f$ is bounded and $r\mapsto\mathcal V(r)-r/2$ is bounded below on $[0,\infty)$, the expression minimized in $\mathcal F^{f;0}(\vartheta)$ is at least $\frac12(\|v\|_{\bH^\rmA}^2+\|u\|_{\bH^\Phi}^2)-C$, uniformly in $\vartheta$. Testing with zero controls gives a uniform upper bound, so the infimum defining $\mathcal F^{f;0}(\vartheta)$ can be restricted to
\[
\|v\|_{\bH^\rmA}^2+\|u\|_{\bH^\Phi}^2\leq R
\]
for some $R$ independent of $\vartheta$.

On this set, the resolvent estimate \eqref{eq:pushed_resolvent_difference}, the GFF decomposition of Theorem~\ref{thm:cov_GFF_decomposition}, and Lemma~\ref{lem:wick-powers} give convergence of the expected controlled Higgs costs to their deterministic limits. Indeed, the shifts and masses converge, while each term of positive noise Wick degree carries a positive power of $\sqrt\eps$. The polynomial bounds in these estimates and the Gaussian moment bounds make the convergence uniform over the bounded controls and $\vartheta\in\bfi[0,2\pi]^2$. The determinant contribution tends to zero uniformly on the same set by Theorem~\ref{thm:convergence_cQ_N} and the enhanced gauge-field moment bounds.

Testing $\mathcal F_0^{f;\eps}$ with these bounded controls and taking their infimum therefore gives
\[
\mathcal F_0^{f;\eps}(\vartheta)
\leq\mathcal F^{f;0}(\vartheta)+o(1),
\]
uniformly in $\vartheta$, which proves the second inequality.
\end{proof}
We now prove the more difficult lower bound for the conditional free energy.
\begin{lemma}\label{lem:lower_bound}
Let $f$ be a bounded Lipschitz observable. Then
\[
\liminf_{\eps\downarrow0}
\inf_{\vartheta\in\bfi[0,2\pi]^2}
\Bigl(
\mathcal F^{f;\eps}(\vartheta)
-\mathcal F^{f;0}(\vartheta)
\Bigr)
\geq0.
\]
\end{lemma}

\begin{proof}
We compare the stochastic and deterministic costs on events where the control energy is bounded. Coercivity makes large control energies unlikely for controls whose expected cost is close to the infimum, while a moment bound on the negative part of the cost limits how much these rare events can lower the expectation.

Write $\E=\E^\rmA\E^\Phi$. For fixed $\vartheta$ and controls $v,u$, set
\[
B=\vartheta+J(v),
\qquad
\bA_v^\eps=\bA^\eps\bplus J(v),
\qquad
\bA_v^0=\0\bplus B,
\qquad
U=\|v\|_{\bH^\rmA}^2+\|u\|_{\bH^\Phi}^2.
\]
Let $K_\eps(v,u)$ denote the cost before taking expectations, with the determinant left out:
\[
K_\eps(v,u)
:=
f\bigl(A^\eps+J(v),\phi_{\bA_v^\eps}^\eps+I^{\bA_v^\eps}(u)\bigr)
+\cV_{\bA_v^\eps}^\eps\bigl(\phi_{\bA_v^\eps}^\eps+I^{\bA_v^\eps}(u)\bigr)
+\frac12 U.
\]
Define $K_0$ by the same formula at $\eps=0$, with enhanced pair $\bA_v^0$ and Higgs noise zero.

For sufficiently small $\kappa$, there are $p>1$, $\alpha<2$, constants $C,M<\infty$, and nonnegative random variables $R_\eps(v)$, independent of $u$, such that
\begin{equation}\label{eq:ld_coercive_estimates}
K_\eps(v,u)\geq\frac14 U-R_\eps(v),
\qquad
\E^\Phi[R_\eps(v)^p]
\leq C\bigl(1+\bfnorm{\bfA^\eps}_{-\kappa}\bigr)^M
\bigl(1+\|v\|_{\bH^\rmA}\bigr)^\alpha,
\end{equation}
uniformly in $0<\eps\leq1$ and $\vartheta$. To obtain this estimate, repeat the shifted Wick expansion in Lemma~\ref{lem:potential-bound-higgs} with noise $\sqrt\eps\,\phi_{0,B}$. The pure leading term is still $c_n\|Z\|_{L^{2n}}^{2n}$, where $Z=I^{\bA_v^\eps}(u)+\phi_{\bA_v^\eps}^\eps-\sqrt\eps\,\phi_{0,B}$. Young's inequality absorbs the lower powers and the mass correction into a fraction of this term. Lemma~\ref{lem:coercive-bound-higgs}, the drift estimate from~\eqref{eq:pushed_resolvent}, and Theorem~\ref{thm:cov_GFF_decomposition} absorb the remaining $\|Z\|_{H_B^s}^{3/2}$ term into the Higgs control energy, with $1-1/n<s<1-4\kappa$. Taking moments of the random remainder, using Lemma~\ref{lem:wick-powers} and the GFF decomposition, gives the conditional moment bound.

Next, choose adapted controls $v^\eps,u^\eps$, with $u^\eps$ chosen measurably in the gauge realization, so that their total expected cost in the iterated variational formula is within $2\eps$ of $\mathcal F^{f;\eps}(\vartheta)$. Theorem~\ref{thm:lower_bound_det}, with $1+\tau(\kappa)<2$, gives, writing $x^-:=\max\{-x,0\}$,
\[
\eps\E^\rmA\bigl[\cQ_\infty(\bfA^\eps\bplus J(v),0)^-\bigr]
\leq C\eps\bigl(1+\E^\rmA\|v\|_{\bH^\rmA}^2\bigr).
\]
The estimates in~\eqref{eq:ld_coercive_estimates}, Young's inequality, and the uniform upper bound from Theorem~\ref{thm:easy_theorem} therefore imply, for all sufficiently small $\eps$,
\[
\E U(v^\eps,u^\eps)\leq C,
\qquad
\E\bigl[(K_\eps(v^\eps,u^\eps)^-)^p\bigr]\leq C.
\]
For the second bound, use $K_\eps^-\leq R_\eps$ from the first estimate in~\eqref{eq:ld_coercive_estimates} and the conditional moment bound from the second. Young's inequality absorbs the subquadratic power of $\|v^\eps\|$ using its second-moment bound and the gauge-field moment bounds.

For each fixed $R\geq1$, the resolvent difference estimate~\eqref{eq:pushed_resolvent_difference}, Lemma~\ref{lem:difference_masses}, and the same Wick expansion give
\[
\E\Bigl[\mathbf1_{\{U\leq R\}}
\bigl|K_\eps(v,u)-K_0(v,u)\bigr|\Bigr]
\leq r_\eps(R),
\qquad
\lim_{\eps\downarrow0}r_\eps(R)=0.
\]
For fixed $R$, the estimates above show that $\|I^{\bA_v^\eps}(u)-I^{\bA_v^0}(u)\|_{H_B^s}$ and $|m_{\bA_v^\eps,0}-1|$ tend to zero in mean as $\eps\downarrow0$, after restriction to $\{U\leq R\}$. Here $m_{\bA_v^0,0}=1$. Each term of positive noise Wick degree carries a positive power of $\sqrt\eps$. The polynomial bounds, Gaussian moments, and Lipschitz continuity of $f$ make the convergence uniform in $\vartheta$ and in the controls.

Every realization of the controls satisfies $K_0(v,u)\geq\mathcal F^{f;0}(\vartheta)$ by Lemma~\ref{lem:deterministic_variational_problem}. Split the expectation according to $U(v^\eps,u^\eps)\leq R$. Markov and H\"older inequalities, the two uniform bounds above, and boundedness of $\mathcal F^{f;0}$ yield
\begin{align*}
\mathcal F^{f;\eps}(\vartheta)+2\eps
&\geq\E K_\eps(v^\eps,u^\eps)-C\eps\\
&\geq\mathcal F^{f;0}(\vartheta)
-r_\eps(R)-C R^{-(1-1/p)}-C\eps.
\end{align*}
By taking $R$ large, then $\eps$ small depending on $R$, we obtain the claimed uniform lower bound.
\end{proof}

\begin{proof}[Proof of Theorem~\ref{thm:large_deviations}]
Theorem~\ref{thm:easy_theorem} and Lemma~\ref{lem:lower_bound}, applied with $f$ and $0$, give
\[
\lim_{\eps \downarrow 0}
\sup_{\vartheta\in\bfi[0,2\pi]^2}
\bigl|\mathcal F^{g;\eps}(\vartheta)-\mathcal F^{g;0}(\vartheta)\bigr|
=0,
\qquad g\in\{f,0\}.
\]
By Lemma~\ref{lem:deterministic_variational_problem}, the limiting functions are continuous in $\vartheta$.

Laplace principle applied to the integration of $\vartheta$ on $\bfi[0,2\pi]^2$, along with the uniform convergence above therefore imply
\[
\lim_{\eps\downarrow0}\bigl(-\eps\log\cZ^\eps(g)\bigr)
=\inf_{\vartheta\in\bfi[0,2\pi]^2}\mathcal F^{g;0}(\vartheta),
\qquad g\in\{f,0\}.
\]
Taking the difference, as in~\eqref{eq:normalized_laplace_functional}, yields
\[
\lim_{\eps\downarrow0}D^{f,\eps}
=\inf_{\vartheta\in\bfi[0,2\pi]^2}\mathcal F^{f;0}(\vartheta)
-\inf_{\vartheta\in\bfi[0,2\pi]^2}\mathcal F^{0;0}(\vartheta).
\]
By Lemma~\ref{lem:deterministic_variational_problem}, these are the infima of $f+\cS_{\YMH}$ and $\cS_{\YMH}$ over Coulomb-gauge fields with harmonic component in $\bfi[0,2\pi]^2$. Proposition~\ref{prop:coulomb_representative} and gauge invariance identify them with the corresponding infima over $\Omega^1 H^1\times H^1$, proving the theorem.
\end{proof}

{\appendix

\section{Interpolation estimates for operators}

\begin{lemma}\label{lem:norm_comparison_pos_op}
    Let $X$ and $Y$ be positive self-adjoint (unbounded) operators on a Hilbert space $V$ satisfying $\|Xv\|_V\leq C\|Yv\|_V$. Then, for any $s\in[0,1]$,
    \[
    \|X^s v\|_V\leq C^s\|Y^sv\|_V.
    \]
\end{lemma}

\begin{proof}
The hypothesis gives $X^2\leq C^2Y^2$ in the sense of quadratic forms. Since $t\mapsto t^s$ is operator monotone for $s\in[0,1]$, we obtain
\[
X^{2s}=(X^2)^s\leq(C^2Y^2)^s=C^{2s}Y^{2s}.
\]
Taking quadratic forms yields $\|X^sv\|_V^2\leq C^{2s}\|Y^sv\|_V^2$, and taking square roots proves the claim.
\end{proof}

We also use the following interpolation result from \cite[Sec.~2.4]{Triebel}.
\begin{lemma}[Interpolation]\label{lem:interpolation_hilbert_scale}
Let $V$ be a Hilbert space and let $A$ be a self-adjoint operator on $V$ satisfying $A\ge cI$ for some $c>0$. For $s\in\R$, let $V^s$ be the completion of $\scD(A^{s/2})$ in the norm $\|v\|_{V^s}:=\|A^{s/2}v\|_V$.
If $T$ is linear and bounded $T:V^{s_j}\to V^{t_j}$ for $j=0,1$, then for every
$\theta\in[0,1]$,
\begin{gather*}
T:V^{(1-\theta)s_0+\theta s_1}\to V^{(1-\theta)t_0+\theta t_1}, \\
\|T\|_{V^{(1-\theta)s_0+\theta s_1}\to V^{(1-\theta)t_0+\theta t_1}}
\le
\|T\|_{V^{s_0}\to V^{t_0}}^{1-\theta}\,
\|T\|_{V^{s_1}\to V^{t_1}}^{\theta}.
\end{gather*}
\end{lemma}

\section{Galerkin approximation}\label{app:galerkin}
For sufficiently smooth $A_T$, we identify $e^{-\cQ_\infty(\bfA_T,0)}$ as a renormalised limit of finite-dimensional determinant ratios.
Set $H_{A,T}\coloneq L_{A_T}-\rmc_T+\lambda_{A_T,T}$, which is bounded
below by the identity by the definition of $\lambda_{A_T,T}$.
We use the finite-dimensional covariance $2P_MH_{A,T}^{-1}P_M$ on $P_ML^2$, where, for $M > 0$, $P_M$ is the Fourier projection onto modes $k\in2\pi\mathbb Z^2$ with $|k|\leq M$.

\begin{lemma} \label{lem:ratio_det_Galerkin}
  Let
  \[
    \bar p_M\coloneq2\Tr\int_0^\infty
      P_M(L_0+1+\lambda)^{-2}P_M\dif\lambda.
  \]
  With the deterministic counterterm
  \[
    \rmo_{T,M}\coloneq\frac{\rmc_T}{2}\bar p_M
      -\E^\Theta\E^\rmA\left[\int_0^\infty
        \bfPi^\lambda(A_T,A_T)\dif\lambda\right],
  \]
  one has
  \begin{align*}
    e^{-\cQ_\infty(\bfA_T,0)}
    =\lim_{M\rightarrow\infty}
      &\exp\left(\frac{\lambda_{A_T,T}-1}{2}\bar p_M-\rmo_{T,M}\right) \\
      &\times
      \frac{\det_M(P_MH_{A,T}^{-1}P_M)}
           {\det_M(P_M(L_0+1)^{-1}P_M)}\exp\left(\frac{1}{4\pi}\|A_{T}\|^2_{L^2}\right).
  \end{align*}
  Here $\det_M$ denotes the determinant on the complex space
  $\operatorname{range}P_M$.
\end{lemma}
\begin{remark}\label{rem:det-counterterm}
    After the mass subtraction, $\exp(-\cQ_\infty(\bfA_T,0))$ matches with the finite dimensional ratio of determinants only up to a constant counterterm, and more importantly a quadratic (but finite) counterterm in $A$. This quadratic counterterm, which is also present in \cite{BC24_YM} is a reflection of the breaking of gauge invariance due to the Galerkin approximation. The cutoff in the main body of the paper is Gauge invariant, as would be true for a lattice cutoff, and there such a counterterm is not present.
\end{remark}
\begin{proof}
  Set $B_M\coloneq(P_MH_{A,T}^{-1}P_M)^{-1}$ on $P_ML^2$. The  identity
  \[
    \log a-\log b=\int_0^\infty
      \left(\frac{1}{b+\lambda}-\frac{1}{a+\lambda}\right)\dif\lambda
  \]
  and the spectral theorem give
  \begin{align*}
    \log
      \frac{\det_M(P_MH_{A,T}^{-1}P_M)}
           {\det_M(P_M(L_0+1)^{-1}P_M)}
    =\int_0^\infty\Tr\bigl(& (B_M+\lambda)^{-1} \\
      &{}-P_M(L_0+1+\lambda)^{-1}P_M\bigr)\dif\lambda.
  \end{align*}
  We claim that
  \begin{equation}\label{eq:compressed-resolvents}
    \int_0^\infty\left|\Tr\left((B_M+\lambda)^{-1}
      -P_M(H_{A,T}+\lambda)^{-1}P_M\right)\right|\dif\lambda
      \xrightarrow{M\rightarrow\infty}0.
  \end{equation}

  Assuming the claim, define
  \begin{align*}
    \bfPi^{\lambda,M}(A_T,A_T)
    &\coloneq \sum_{p\in 2\pi \mathbb Z^2}\sum_{j,k=1}^2
      \widehat{A_{T,j}}(p)^*\Pi_{j,k}^{\lambda,M}(p)
      \widehat{A_{T,k}}(p), \\
    \Pi_{j,k}^{\lambda,M}(p)
    &\coloneq \sum_{\substack{q\in 2\pi\mathbb Z^2\\|q|\leqslant M}}
      \frac{1}{(1+\lambda+|q|^2)^2}
      \left(\delta_{j,k}
      -\frac{(2q+p)_j(2q+p)_k}{1+\lambda+|q+p|^2}\right).
  \end{align*}
  Applying \eqref{eq:decomp_scq} with $\bA=\bA_T$ and $B=0$, and taking the trace after projection by $P_M$, gives the following identity. The first-order term vanishes by symmetry, as in \eqref{eq:tr_q1_vanishes}.
  \begin{align}\label{eq:projected-decomposition}
    &\Tr P_M\bigl((L_0+1+\lambda)^{-1}
      -(H_{A,T}+\lambda)^{-1}\bigr)P_M \nonumber\\
    &\quad=(\lambda_{A_T,T}-1-\rmc_T)
      \Tr P_M(L_0+1+\lambda)^{-2}P_M
      +\bfPi^{\lambda,M}(A_T,A_T) \nonumber\\
    &\qquad+\Tr P_M\Rem_{\bA_T,0}(\lambda)P_M.
  \end{align}
  Now set,
  \begin{align*}
    \cQ_{T,M}^{\mathrm{Gal}}(A_T)
    \coloneq {}&-\frac{\lambda_{A_T,T}-1}{2}\bar p_M+\rmo_{T,M}\\
      &+\int_0^\infty\Tr P_M
      \bigl((L_0+1+\lambda)^{-1}
      -(H_{A,T}+\lambda)^{-1}\bigr)P_M\dif\lambda,
  \end{align*}
  then \eqref{eq:projected-decomposition} implies
  \begin{equation}\label{eq:Gal-Q-decomposition}
    \cQ_{T,M}^{\mathrm{Gal}}(A_T)
    =\int_0^\infty\bigl(\bfPi^{\lambda,M}(A_T,A_T)
      -\E^\Theta\E^\rmA[\bfPi^\lambda(A_T,A_T)]\bigr)\dif\lambda
      +\int_0^\infty
        \Tr P_M\Rem_{\bA_T,0}(\lambda)P_M\dif\lambda.
  \end{equation}
  Now for fixed $T$ it follows from Lemma \ref{lemma:8pi} below, that
  \[\int_{0}^{\infty} \bfPi^{\lambda,M}(A_T,A_T) \diff \lambda \to \int_0^{\infty}\bfPi^{\lambda} (A_T,A_T)\diff\lambda+\frac{1}{4\pi}\|A_{T}\|^2_{L^2}\]
  Furthermore  $
    \left|\Tr P_M\Rem_{\bA_T,0}(\lambda)P_M\right|
    \leqslant\Tr\left|\Rem_{\bA_T,0}(\lambda)\right|,$  and the right-hand side is integrable in $\lambda$. These two facts together with dominated convergence  therefore yield
  \begin{align*}
    \cQ_{T,M}^{\mathrm{Gal}}(A_T)-\frac{1}{4\pi}\|A_{T}\|^2_{L^2}
      \longrightarrow&\int_0^\infty \bV_{A_T}(\lambda) \dif\lambda
      +\int_0^\infty
        \Tr\Rem_{\bA_T,0}(\lambda)\dif\lambda\\
      =&\cQ_\infty(\bfA_T,0).
 \end{align*}

  Finally, \eqref{eq:compressed-resolvents} shows that the logarithm of the
  expression on the right-hand side of the lemma equals
  \begin{align*}
    &\frac{\lambda_{A_T,T}-1}{2}\bar p_M-\rmo_{T,M} \\
    &\quad+\int_0^\infty\Tr P_M\bigl((H_{A,T}+\lambda)^{-1}
      -(L_0+1+\lambda)^{-1}\bigr)P_M\dif\lambda+\frac{1}{4\pi}\|A_T\|^2_{L^2}+o_M(1),
  \end{align*}
  which is $-\cQ_{T,M}^{\mathrm{Gal}}(A_T)+\frac{1}{4\pi}\|A_T\|^2_{L^2}+o_M(1)$ and therefore converges to $-\cQ_\infty(\bfA_T,0)$.
\end{proof}

We now prove \eqref{eq:compressed-resolvents}. Fix $T$, abbreviate
$H=H_{A,T}$, and let $P=P_M$ be the projection onto the Fourier modes
$|n|\leqslant M$. Let $Q\coloneq1-P$ and, for $a,b\in\{P,Q\}$, write
$G_{ab}=aHb$. Define
$S_{2,M}(\lambda)\coloneq\bigl(P_M(H+\lambda)^{-1}P_M\bigr)^{-1}$.
The Schur complement formulas for $H$ and $H+\lambda$ give
\begin{align*}
  (P_MH^{-1}P_M)^{-1}
    &=G_{PP}-G_{PQ}G_{QQ}^{-1}G_{QP}, \\
  S_{1,M}(\lambda)
    &\coloneq G_{PP}-G_{PQ}G_{QQ}^{-1}G_{QP}+\lambda P_M, \\
  S_{2,M}(\lambda)
    &=G_{PP}+\lambda P_M
      -G_{PQ}(G_{QQ}+\lambda)^{-1}G_{QP}.
\end{align*}
The resolvent identity then yields
\begin{align*}
  S_{1,M}(\lambda)-S_{2,M}(\lambda)
    &=G_{PQ}\bigl((G_{QQ}+\lambda)^{-1}-G_{QQ}^{-1}\bigr)G_{QP} \\
    &=-\lambda G_{PQ}(G_{QQ}+\lambda)^{-1}G_{QQ}^{-1}G_{QP}, \\
  S_{1,M}(\lambda)^{-1}-S_{2,M}(\lambda)^{-1}
    &=\lambda S_{1,M}(\lambda)^{-1}G_{PQ}(G_{QQ}+\lambda)^{-1} \\
    &\quad\times G_{QQ}^{-1}G_{QP}S_{2,M}(\lambda)^{-1}.
\end{align*}
Observe that $\|\lambda S_{1,M}(\lambda)^{-1}\|_{L^2\to L^2}\leqslant1$. Therefore,
\begin{align}
  &\left|\Tr\bigl(S_{1,M}(\lambda)^{-1}
    -S_{2,M}(\lambda)^{-1}\bigr)\right| \nonumber\\
  &\quad\leqslant
    \|G_{PQ}(G_{QQ}+\lambda)^{-1}\|_{\HS}
    \|G_{QQ}^{-1}G_{QP}S_{2,M}(\lambda)^{-1}\|_{\HS}.
  \label{eq:trace-HS-bound}
\end{align}

We use the following elliptic estimates, whose proof we leave to the reader.

\begin{lemma}\label{lem:compressed-elliptic}
  For every $f\in QL^2$, uniformly in $M$ and $\lambda\geqslant0$,
  \begin{align*}
    \|(-\Delta+\lambda+1)(G_{QQ}+\lambda)^{-1}f\|_{L^2}
      &\leqslant C(A_T)\|f\|_{L^2}, \\
    \|(1-\Delta)G_{QQ}^{-1}f\|_{L^2}
      &\leqslant C(A_T)\|f\|_{L^2}.
  \end{align*}
\end{lemma}

Because the constant-coefficient part of $H_{A_T,T}$ commutes with $P_M$, the
off-diagonal blocks $G_{PQ}$ and $G_{QP}$ contain only the smooth first- and
zeroth-order coefficients of $H_{A_T,T}$. Lemma~\ref{lem:compressed-elliptic} and a
direct Fourier summation therefore give, for every
$\varepsilon\in(0,\frac12)$,
\begin{align}
  \|G_{PQ}(G_{QQ}+\lambda)^{-1}\|_{\HS}^2
  &\leqslant C(A_T)(1+\lambda)^{-\varepsilon}
    M^{-1+2\varepsilon}, \label{eq:first-HS}\\
  \|G_{QQ}^{-1}G_{QP}\|_{\HS}^2
  &\leqslant C(A_T)M^{-1+2\varepsilon}.
  \label{eq:second-HS}
\end{align}
Indeed, the first-order contribution to \eqref{eq:first-HS} is
bounded by a constant times
\[
  \sum_{\substack{|j|>M\\|i|\leqslant M}}
  \frac{|j|^2|\widehat{A_T}(i-j)|^2}
       {(1+\lambda+|j|^2)^2};
\]
splitting into $|j|\leqslant2|i-j|$ and $|j|>2|i-j|$ proves
\eqref{eq:first-HS}, since $A_T$ is smooth. The zeroth-order contribution and
\eqref{eq:second-HS} are easier.

Finally, \mbox{$\|S_{2,M}(\lambda)^{-1}\|_{L^2\to L^2}\leqslant(1+\lambda)^{-1}$}
because $S_{2,M}(\lambda)^{-1}=P_M(H_{A_T,T}+\lambda)^{-1}P_M$
and $H\geqslant\id$.
Consequently, \eqref{eq:trace-HS-bound}--\eqref{eq:second-HS} imply
\[
  \left|\Tr\bigl(S_{1,M}(\lambda)^{-1}
      -S_{2,M}(\lambda)^{-1}\bigr)\right|
  \leqslant C(A_T)(1+\lambda)^{-1-\varepsilon/2}
    M^{-1+2\varepsilon}.
\]
This bound is integrable in $\lambda$ and tends to zero as $M\to\infty$.
Dominated convergence proves \eqref{eq:compressed-resolvents}.

\begin{lemma}\label{lemma:8pi}
  For fixed smooth $A_T$, we have
  \[ \lim_{M \rightarrow \infty} \int^{\infty}_0 \boldsymbol{\Pi}^{\lambda , M} (A_T,
     A_T) \diff \lambda = \int^{\infty}_0 \boldsymbol{\Pi}^{\lambda} (A_T, A_T) \diff
     \lambda + \frac{1}{4 \pi} \| A_T \|^2_{L^2} . \]
\end{lemma}

\begin{proof}
  Write $\Pi^{\lambda,M}(0)=c_M(\lambda)I_2$ and
  $\Pi^\lambda(0)=c(\lambda)I_2$.
  Since $A_T$ is smooth, it is sufficient to show the following things:
  \begin{enumerate}
    \item\label{item:Gal-symbol} $\Pi^{\lambda, M} (p) - \Pi^{\lambda, M} (0) \rightarrow
    \Pi^{\lambda} (p) - \Pi^{\lambda} (0)$ $\forall p \neq 0$ , and $\|
    \Pi^{\lambda, M} (p) - \Pi^{\lambda, M} (0) \| \leqslant | p |^2 (1 +
    \lambda)^{- 2}$.

    \item\label{item:Gal-zero-mode} $\int c_M(\lambda) \diff \lambda - \int c(\lambda)
    \diff \lambda \rightarrow \frac{1}{4 \pi}$.
  \end{enumerate}
  Then the statement will follow since, applying Fourier isometry, we get

  \begin{align}
     \int^{\infty}_0 \boldsymbol{\Pi}^{\lambda , M} (A_T, A_T) \diff \lambda
    = & \int^{\infty}_0 \sum_{p \in 2 \pi \mathbb{Z}^2} \hat{A}_T(p)^* \Pi^{\lambda, M} (p)
    \hat{A}_T(p) \diff \lambda \nonumber\\
    = & \int^{\infty}_0 \sum_{p \in 2 \pi \mathbb{Z}^2} \hat{A}_T(p)^* (\Pi^{\lambda, M} (p)
    - \Pi^{\lambda, M} (0)) \hat{A}_T(p) \diff\lambda \nonumber\\
    &\quad + \| A_T \|^2_{L^2} \int^{\infty}_0 c_M(\lambda) \diff \lambda . \nonumber
  \end{align}

  By~(\ref{item:Gal-symbol}), dominated convergence applies to the
  first term, and~(\ref{item:Gal-zero-mode}) then gives the result.
  For~(\ref{item:Gal-symbol}), repeat the estimates in the proof of
  Lemma~\ref{lem:vac_pol_bound} for each matrix entry. The terms
  odd in $q$ still cancel under the restriction $|q|\le M$.
  For~(\ref{item:Gal-zero-mode}), the definition of $c_M(\lambda)$ gives

  \begin{align}
    \int c_M(\lambda) \diff \lambda & = \int^{\infty}_0
    \sum_{q \in 2 \pi \mathbb{Z}^2, | q | \leqslant M} \frac{(1 + \lambda - |
    q |^2)}{(1 + \lambda + | q |^2)^3} \diff\lambda \nonumber\\
    & = - \sum_{q \in 2 \pi \mathbb{Z}^2, | q | \leqslant M}
    \int^{\infty}_0 \frac{\diff}{\diff\lambda} \frac{(1 + \lambda)}{(1 +
    \lambda + | q |^2)^2} \diff \lambda \nonumber\\
    & = \sum_{q \in 2 \pi \mathbb{Z}^2, | q | \leqslant M}
    \frac{1}{(1 + | q |^2)^2} . \nonumber
  \end{align}

  To compare this with $\int^{\infty}_0 c(\lambda) \diff \lambda$, we
  fist recall that that $c(\lambda)$ is integrable in $\lambda$, so
  $\int^{\infty}_0 c(\lambda) \diff \lambda = \lim_{L \rightarrow
  \infty} \int^L_0 c(\lambda) \diff \lambda$. Then recalling that the
  sum which defined $c(\lambda)$ converges absolutely (for fixed $\lambda
  !$), we obtain by Fubini
  \begin{align}
    \int^{\infty}_0 c(\lambda) \dif \lambda & = \lim_{L \rightarrow
    \infty} \int^L_0 \sum_{q \in 2 \pi \mathbb{Z}^2} \frac{(1 + \lambda - | q
    |^2)}{(1 + \lambda + | q |^2)^3} \dif \lambda \nonumber\\
    & = \lim_{L \rightarrow \infty} \sum_{q \in 2 \pi \mathbb{Z}^2} \int^L_0
    \frac{(1 + \lambda - | q |^2)}{(1 + \lambda + | q |^2)^3} \dif \lambda
    \nonumber\\
    & = \lim_{L \rightarrow \infty} - \sum_{q \in 2 \pi \mathbb{Z}^2}
    \int^L_0 \frac{\dif}{\dif \lambda} \frac{(1 + \lambda)}{(1 + \lambda +
    | q |^2)^2} \dif \lambda \nonumber\\
    & = \sum_{q \in 2 \pi \mathbb{Z}^2} \frac{1}{(1 + | q |^2)^2} - \lim_{L
    \rightarrow \infty} \sum_{q \in 2 \pi \mathbb{Z}^2} \frac{1 + L}{(1 + L +
    | q |^2)^2} . \nonumber
  \end{align}
  Dominated convergence easily shows that $\sum_{q \in 2 \pi \mathbb{Z}^2, | q
  | \leqslant M} \frac{1}{(1 + | q |^2)^2} - \sum_{q \in 2 \pi \mathbb{Z}^2,}
  \frac{1}{(1 + | q |^2)^2}$ tends to $0$ as $M \rightarrow \infty$. It thus
  remains to show that the last term equals $\frac{1}{4 \pi}$. Indeed

  \begin{align}
    & \lim_{L \rightarrow \infty} \sum_{q \in 2 \pi \mathbb{Z}^2} \frac{1 +
    L}{(1 + L + | q |^2)^2} \nonumber\\
    = & \lim_{L \rightarrow \infty} (1 + L)^{- 1} \sum_{q \in 2 \pi
    \mathbb{Z}^2} \frac{1}{\left( 1 + \frac{| q |^2}{1 + L} \right)^2} .
    \nonumber
  \end{align}
  Now relabelling $\sqrt{1 + L}^{- 1} q = \tilde{q}$, this becomes

  \begin{align}
    \lim_{L \rightarrow \infty} (1 + L)^{- 1} \sum_{\tilde{q} \in \frac{2
    \pi}{\sqrt{1 + L}} \mathbb{Z}^2} \frac{1}{(1 + | \tilde{q} |^2)^2} & =
    \frac{1}{4 \pi^2} \int \frac{1}{(1 + | x |^2)^2} \dif x = \frac{1}{4
    \pi}, \nonumber
  \end{align}
  where the first equality follows by Riemann sum approximation. This concludes the proof.
\end{proof}

\section{Gauge fixing on the torus}\label{app:gauge_fixing}
We continue to work on $\T^2=\R^2/\Z^2$, with volume one.
Gauge transformations act by
\[
g\Cdot(\bar A,\psi)=(\bar A-\diff g\,g^{-1},g\psi).
\]

\begin{proposition}\label{prop:coulomb_representative}
For every $(\bar A,\psi)\in\Omega^1H^1\times H^1(\T^2;\C)$, there exists $g\in H^2(\T^2;U(1))$ such that
\[
(A,\phi):=g\Cdot(\bar A,\psi)\in\Omega^1H^1\times H^1(\T^2;\C),
\qquad
\diff^*A=0,\qquad \int_{\T^2}A\,\dif x\in\bfi(0,2\pi]^2.
\]
The gauge field $A$ is unique within the gauge orbit.
The transformation $g$ is unique up to a constant in $U(1)$, which acts on $\phi$ by a constant phase rotation.
Moreover,
\[
\diff A=\diff\bar A,\qquad
\diff_A\phi=g\,\diff_{\bar A}\psi,\qquad
|\phi|=|\psi|,
\]
so the Yang--Mills--Higgs action is unchanged.
\end{proposition}

\begin{proof}
The Hodge decomposition splits $\bar A$ into orthogonal exact, co-exact and harmonic parts:
\[
\bar A=\diff\alpha+\diff^*\beta+\gamma,
\]
where $\alpha\in H^2(\T^2;\bfi\R)$ has zero mean, $\beta$ is an $H^2$ $2$-form, and $\gamma=\int_{\T^2}\bar A\,\dif x$ is constant.
Choose $k\in\Z^2$ such that $\vartheta:=\gamma-2\pi\bfi k\in\bfi(0,2\pi]^2$, and set
\[
g_k(x):=e^{2\pi\bfi k\cdot x},\qquad
g:=e^\alpha g_k,\qquad
A=\diff^*\beta+\vartheta.
\]
Then $\diff g\,g^{-1}=\diff\alpha+2\pi\bfi k$, so $A=\bar A^g$ satisfies \eqref{eq:Coulomb}.
Sobolev product and composition estimates give $g\in H^2$ and $g\psi\in H^1$, and the product rule gives the stated identities and invariance of the action.

For uniqueness, write any $H^2$ transformation between two such representatives as $h=e^\eta g_\ell$, with $\eta\in H^2(\T^2;\bfi\R)$ and $\ell\in\Z^2$.
The Coulomb condition gives $\diff^*\diff\eta=0$, so $\eta$ is constant.
Since both harmonic components belong to $\bfi(0,2\pi]^2$, $\ell=0$ and $h$ is constant.
\end{proof}

For gauge-invariant $f$, Proposition~\ref{prop:coulomb_representative} lets us restrict the infima in Section~\ref{sec:large_deviations} to Coulomb representatives.

Set $a=\diff^*(\diff\diff^*)^{-1}\xi$ as in \eqref{eq:intro_pure_YM_field}, with the inverse taken on mean-zero $2$-forms.
Then $\diff^*a=0$, $\int_{\T^2}a\,\dif x=0$, and $\diff a=\xi$.
The Yang--Mills energy depends only on $a$.
After truncating to finitely many Fourier modes, the change of variables $a\mapsto\diff a$ has constant Jacobian.
Taking $\vartheta$ independently and uniformly on $\bfi(0,2\pi]^2$ gives, for bounded measurable $F$,
\[
\int F(A)\,\dif\mu^{\rmA}(A)
=
\int_{\bfi(0,2\pi]^2}
\E\big[F(\vartheta+\diff^*(\diff\diff^*)^{-1}\xi)\big]\,\bar\diff\vartheta.
\]
The finite-rank approximation of its Gaussian part is specified in Section~\ref{sec:gauge_variational_setup}.

For a gauge-invariant observable $F$, write $G(A,\phi):=F(A,\phi)e^{-\cS_{\YMH}(A,\phi)}$.
The following calculation is formal: the symbols $\dif\bar A$ and $\dif\psi$ do not denote infinite-dimensional Lebesgue measures.
Writing $\bar A=A+\diff\alpha+2\pi\bfi k$ gives, with $\propto$ suppressing factors independent of $F$,
\begin{align}\label{eq:formal_computation_gauge_change}
\int G(\bar A,\psi)\,\dif\bar A\,\dif\psi
&\propto
\sum_{k\in\Z^2}\int
G(A+\diff\alpha+2\pi\bfi k,\psi)\,\dif A\,\dif\alpha\,\dif\psi
\nonumber\\
&=
\sum_{k\in\Z^2}\int
G(A,e^\alpha g_k\psi)\,\dif A\,\dif\alpha\,\dif\psi
\propto\int G(A,\phi)\,\dif A\,\dif\phi.
\end{align}
This formally gives \eqref{eq:intro_measure} restricted to Coulomb gauge.
The rigorous construction follows from \eqref{eq:intro_regularised_measure} and Theorem~\ref{thm:main}.

\section{Symbolic index}

Frequently used symbols and their meanings are listed below, with page references to their definitions.
\begin{center}
{\small
\renewcommand{\arraystretch}{1.1}
\begin{longtable}{lll}
\toprule
Symbol(s) & Meaning & Page(s)\\
\midrule
\endfirsthead
\toprule
Symbol(s) & Meaning & Page(s)\\
\midrule
\endhead
\bottomrule
\endfoot
\bottomrule
\endlastfoot

$\vartheta$
& Harmonic component of the gauge field
& \pageref{eq:intro_pure_YM_field} \\

$A_T,\,A_T^0$
& Regularised gauge field/mean-zero Gaussian part
& \pageref{sec:def_regularised_measure}, \pageref{def:gauge_field_decomposition} \\

$\bA,\,\bfA$
& Enhanced gauge fields $(A,A^2)$ / $(A,A^2,\bV)$
& \pageref{sec:domain_renorm_laplacian}, \pageref{def:full_enhancement} \\

$\triple{\bA-\bar\bA}_{-\kappa},\,\bfnorm{\bfA-\bar\bfA}_{-\kappa}$
& Distances in $\cX^{-\kappa}$ / $\cX_{\rmV}^{-\kappa}$
& \pageref{def:enhanced_distance}, \pageref{def:full_enhancement_distance} \\

$\rmc_T$
& Covariant-Laplacian counterterm
& \pageref{sec:def_regularised_measure} \\

$\fC_\kappa(\bA,\bar\bA;B),\fC_\kappa(\bA;B)$ & Polynomial control factor in $\bA,\bar\bA$ and $B$
& \pageref{eq:fC_def} \\

$J_{\bA,B}$ & Rough perturbation $2A^*\diff_B+A^2$ & \pageref{def:rough_perturbation} \\

$m_{\bA,B}$
& Mass making $L_{\bA,B}+m_{\bA,B}$ positive
& \pageref{eq:m_def} \\

$\mu_T^\rmA,\,\mu_{A,N,T}^{\GFF}$
& Laws of $A_T$ / $\phi_{A,N,T}$
& \pageref{def:gauge_reference_law}, \pageref{eq:intro_regularised_measure} \\

$\mu_{N,T}^{\YMH},\,\cZ_{N,T}$
& Regularised YMH measure/partition function
& \pageref{eq:intro_regularised_measure} \\

$\Omega^1 X,\,\Omega^1_\Coul X$
& General/Coulomb $1$-forms with components in $X$
&  \pageref{def:one_forms}, \pageref{def:coulomb_forms} \\

$\rmp_N$
& Flat Wick constant
& \pageref{def:flat_Wick_constant} \\

$\Pi^\lambda$
& Vacuum-polarisation Fourier multiplier
& \pageref{def:vacuum_polarisation_multiplier} \\

$\bfPi^\lambda(A,B),\,\bfPi(A,B)$
& Vacuum-polarisation bilinear form/its $\lambda$-dependent family
& \pageref{def:vacuum_polarisation_bilinear}, \pageref{def:vacuum_polarisation_family} \\

$\phi_{A,N,T}$
& Regularised covariant Gaussian field
& \pageref{eq:intro_covariant_GFF} \\

$\phi_{\bA,B,N},\,\phi_{\0,B,N}$
& Truncated covariant GFFs for $L_{\bA,B}$ / $L_B$
& \pageref{def:covariant_GFF} \\

$Q_{N,T}(A)$
& Regularised determinant-ratio functional
& \pageref{eq:determinant_computation} \\

$\cQ_N(\bfA,B),\,\cQ_\infty(\bfA,B)$
& Rough determinant ratio/its $N\to\infty$ limit
& \pageref{eq:def_cQ_N}, \pageref{thm:convergence_cQ_N} \\

$R_B(z),\,R_{\bA,B}(z)$
& Resolvents $(L_B+z)^{-1}$ / $(L_{\bA,B}+z)^{-1}$
& \pageref{def:classical_resolvent}, \pageref{sec:resolvent} \\

$\Sigma_t^{\bA,B}$
& Resolvent $(L_{\bA,B}+m_{\bA,B}+t)^{-1}$
& \pageref{def:covariant_GFF_resolvent} \\

$\cV,\,\cV_{A,N,T}$
& Higgs interaction polynomial/regularised potential
& \pageref{eq:def_V}, \pageref{eq:intro_regularised_measure} \\

$\scV_N^m$
& Wick-renormalised Higgs potential with mass $m$
& \pageref{eq:def_V_N_m} \\

$\bV_A$
& Renormalised vacuum-polarisation field
& \pageref{lem:bV_properties} \\

$\cW_N^{h,q,m}(\bA,B)$
& Conditional Higgs free-energy functional
& \pageref{eq:def_cW} \\

$X^\Phi,\,X^\rmA$
& Higgs/gauge cylindrical Brownian motions
& \pageref{eq:intro_covariant_GFF}, \pageref{def:gauge_noise} \\

$\cX^{-\kappa},\,\cX_{\rmV}^{-\kappa}$
& Space of enhanced fields/its product with $L^1(\R_+;\R)$
& \pageref{sec:domain_renorm_laplacian}, \pageref{def:full_enhancement} \\

\end{longtable}
}
\end{center}

}

    \addcontentsline{toc}{section}{References}
\bibliographystyle{alpha}
\bibliography{references}

\end{document}